\documentclass[a4paper,10pt]{amsart}

\pdfoutput=1

\usepackage[T1]{fontenc}
\usepackage[utf8]{inputenc}
\usepackage{newunicodechar}

\usepackage[usenames,dvipsnames,svgnames,x11names,table]{xcolor}

\usepackage[]{accents,amsmath,amsfonts,amssymb,amsthm,amscd,amsxtra,amstext,latexsym,syntonly,tabularx}

\usepackage{pgf,tikz}
\usetikzlibrary{arrows}
\usepackage{enumitem}
\usepackage{graphicx}
\usepackage{mathrsfs}
\usepackage{stmaryrd}
\usepackage[normalem]{ulem}
\usepackage[all]{xy}
\usepackage[pagebackref, colorlinks=true, linkcolor=
Firebrick4,
citecolor=
Turquoise4,
filecolor=SteelBlue4,
urlcolor=
SteelBlue4
]{hyperref}

\theoremstyle{plain}
\newtheorem{xx}{xx}[section]
\newtheorem{thm}[xx]{Theorem}%[section]
\newtheorem*{thm*}{Theorem}
\newtheorem*{thmA*}{Theorem A}
\newtheorem*{thmB*}{Theorem B}
\newtheorem{prop}[xx]{Proposition}%[theorem]
\newtheorem{cor}[xx]{Corollary}%[theorem]
\newtheorem*{corA*}{Corollary A}
\newtheorem*{corB*}{Corollary B}
\newtheorem{lem}[xx]{Lemma}%[theorem]
\newtheorem*{lem*}{Lemma}%[theorem]

\theoremstyle{definition}
\newtheorem{defn}[xx]{Definition}%[theorem]
\newtheorem*{defn*}{Definition}%[theorem]
\newtheorem{ex}[xx]{Example}%[theorem]
\theoremstyle{remark}
\newtheorem{rem}[xx]{Remark}
\newtheorem*{rem*}{Remark}

\numberwithin{equation}{xx}

\DeclareMathOperator{\Add}{Add}

\DeclareMathOperator{\at}{at}

\DeclareMathOperator{\Char}{char}

\DeclareMathOperator{\coker}{coker}
\DeclareMathOperator{\Def}{Def}

\DeclareMathOperator{\Der}{Der}
\DeclareMathOperator{\gDer}{{}_{{\ccirc}}\hspace{-0.1em}{\Der}}%\textnormal{\tiny o}
\DeclareMathOperator{\cE}{E}

\DeclareMathOperator{\End}{End}
\DeclareMathOperator{\gEnd}{{}_{{\ccirc}}\hspace{-0.1em}End}%\textnormal{\tiny o}

\DeclareMathOperator{\gExal}{{}_{{\ccirc}}\hspace{-0.1em}Exal}
\DeclareMathOperator{\Ext}{Ext}
\DeclareMathOperator{\gExt}{{}_{{\ccirc}}\hspace{-0.1em}Ext} %\textnormal{\tiny o}
\DeclareMathOperator{\sgExt}{{{}_{\ast}\hspace{-0.06em}}{\Ext}}

\DeclareMathOperator{\cH}{H}

\DeclareMathOperator{\Hom}{Hom}
\DeclareMathOperator{\gHom}{{}_{{\ccirc}}\hspace{-0.1em}{\Hom}}%\textnormal{\tiny o}

\DeclareMathOperator{\gRHom}{\BB{R}\hspace{-0.1em}\gHom}%\textnormal{\tiny o}
\DeclareMathOperator{\id}{{id}}
\DeclareMathOperator{\im}{im}

\DeclareMathOperator{\Mod}{mod}
\DeclareMathOperator{\gmod}{{}_{{\ccirc}}\hspace{-0.1em}mod}

\DeclareMathOperator{\NL}{\mathnormal{N\hspace{-0.24em}L}}
\DeclareMathOperator{\ob}{ob}
\DeclareMathOperator{\pdim}{pdim}

\DeclareMathOperator{\pr}{pr}

\DeclareMathOperator{\Proj}{Proj}

\DeclareMathOperator{\Rad}{rad}

\DeclareMathOperator{\Sets}{\cat{Sets}}

\DeclareMathOperator{\Spec}{Spec}

\DeclareMathOperator{\Syz}{Syz}
\DeclareMathOperator{\cT}{T}

\DeclareMathOperator{\Tor}{Tor}

\newcommand{\co}{\colon}
\newcommand{\ra}{\rightarrow}
\newcommand{\la}{\leftarrow}
\newcommand{\lra}{\longrightarrow}
\newcommand{\lla}{\longleftarrow}

\newcommand{\thr}{\twoheadrightarrow}
\newcommand{\thl}{\twoheadleftarrow}
\newcommand{\hra}{\hookrightarrow}

\newcommand{\ot}{\hspace{0.06em}{\otimes}\hspace{0.06em}}
\newcommand{\op}{\hspace{0.06em}{\oplus}\hspace{0.04em}}

\newcommand{\sbeq}{\subseteq}

\newcommand{\vare}{\varepsilon}
\newcommand{\vD}{\varDelta}
\newcommand{\vG}{\varGamma}
\newcommand{\vL}{\varLambda}

\newcommand{\dprime}{\prime\hspace{-0.1em}\prime}

\newcommand{\limproj}{\varprojlim}
\newcommand{\liminj}{\varinjlim}
\newcommand{\wh}{\widehat}

\newcommand{\He}{{{}_{\vL\hspace{-0.1em}}\cat{H}_{k}}}
\newcommand{\Ar}{{{}_{\vL\hspace{-0.1em}}\cat{A}_{k}}}

\newcommand{\hot}{{\tilde{\otimes}}}

\newcommand\bdot{\ensuremath{%
  \mathchoice%
   {\mskip\thinmuskip\lower0.2ex\hbox{\scalebox{1.5}{$\cdot$}}\mskip\thinmuskip}}%
   {\mskip\thinmuskip\lower0.2ex\hbox{\scalebox{1.5}{$\cdot$}}\mskip\thinmuskip}%        
   {\lower0.3ex\hbox{\scalebox{1.2}{$\cdot$}}}%  
   {\lower0.3ex\hbox{\scalebox{1.2}{$\cdot$}}}%
}

\newcommand{\ccirc}{\mathbin{\mathchoice
  {\xcirc\scriptstyle}
  {\xcirc\scriptstyle}
  {\xcirc\scriptscriptstyle}
  {\xcirc\scriptscriptstyle}
}}
\newcommand{\xcirc}[1]{\vcenter{\hbox{$#1\circ$}}}

\renewcommand{\phi}{\varphi}
\renewcommand{\geq}{\geqslant}

\newcommand{\BB}[1]{\mathbb{{#1}}}

\newcommand{\fr}[1]{\mathfrak{{#1}}}

\newcommand{\cat}[1]{\mathsf{{#1}}}

\newcommand{\mc}[1]{{\mathcal{#1}}}
\newcommand{\mr}[1]{\mathrm{{#1}}}

\newcommand{\tn}{\textnormal}

\newcommand{\wbar}[1]{\mkern 3.5mu\overline{\mkern-3mu{#1}\mkern-0.3mu}\mkern -0.4mu}

\newcommand{\der}[3]{{\Der}_{#1}({#2},{#3})}
\newcommand{\gder}[3]{{\gDer}_{#1}({#2},{#3})}
\newcommand{\df}[2]{{\Def}_{\hspace{-0.08em}#2}^{#1}}

\newcommand{\hm}[4]{{\Hom}_{#2}^{#1}({#3},{#4})}
\newcommand{\ghm}[4]{{\gHom}_{#2}^{#1}({#3},{#4})}

\newcommand{\grhm}[4]{{\gRHom}_{#2}^{#1}({#3},{#4})}

\newcommand{\nd}[3]{{\End} _{#2}^{#1}({#3})}
\newcommand{\gnd}[3]{{\gEnd} _{#2}^{#1}({#3})}

\newcommand{\ol}[1]{\overline{{#1}}}

\newcommand{\tor}[4]{{\Tor}_{#2}^{#1}({#3},{#4})}
\newcommand{\ul}[1]{\underline{{#1}}}

\newcommand{\ulset}[2]{{\hspace{-0.1em}}^{#1}{\hspace{-0.2em}#2}}
\newcommand{\xla}[1]{\xleftarrow{{#1}}}
\newcommand{\xra}[1]{\xrightarrow{{#1}}}
\newcommand{\gxal}[3]{{\gExal}_{#1}({#2},{#3})}
\newcommand{\xt}[4]{{\Ext}_{#2}^{#1}({#3},{#4})}
\newcommand{\gxt}[4]{{\gExt}_{#2}^{#1}({#3},{#4})}
\newcommand{\sgxt}[4]{{\sgExt}_{#2}^{#1}({#3},{#4})}

\newcommand{\syz}[2]{{\Syz}_{#2}^{#1}}
\begin{document}
\title[Versal deformations and Cohen-Macaulay approximation]
{Versal deformations of pairs\\ and Cohen-Macaulay approximation}
\author{Runar Ile}
\address{BI Norwegian Business School}
\email{runar.ile@bi.no}
\begin{abstract}
For a pair (algebra, module) with equidimensional and isolated singularity we establish the existence of a versal henselian deformation. Obstruction theory in terms of an Andr{\'e}-Quillen cohomology for pairs is a central ingredient in the Artin theory used. In particular we give a long exact sequence relating the algebra cohomology and the module cohomology with the cohomology of the pair and define a Kodaira-Spencer class for pairs. Cohen-Macaulay approximation induces maps between versal base spaces for pairs and cohomology conditions imply properties like smoothness, isomorphism and linear section.
\end{abstract}
\maketitle
%%%%%%%%%%%%%%%%%%%%%%%%%%%%%%%%%%%%%%%%%%%%%%%%%%%%%%%%
%%%%%%%%%%%%%%%%%%% SECTION %%%%%%%%%%%%%%%%%%%%%%%%%%%%%%%%
%%%%%%%%%%%%%%%%%%%%%%%%%%%%%%%%%%%%%%%%%%%%%%%%%%%%%%%%
\section{Introduction}\label{sec.intro}
In this article we prove a versality result for pairs (algebra, module) and give applications to parametrised Cohen-Macaulay approximation.

Let $k$ be field, $A$ a $k$-algebra and $N$ a finite $A$-module. By applying M. Artin's approximation theory \cite{art:69}, R. Elkik proved in \cite{elk:73} that there exists a (henselian) versal deformation of $A$ if $A$ is algebraic and equidimensional with isolated singularity. On the other hand, if $N$ is locally free except at the closed point Elkik's results also implies the existence of a versal deformation of $N$ where the algebra is trivially deformed, as H. von Essen showed in \cite[2.3]{ess:90}. Our main result Theorem \ref{thm.vers} says that the pair $(A, N)$ with the combined conditions has a versal deformation. 

The study of singularities in terms of the representation theory has in recent years given interesting results about the derived equivalences for flops, initiated by M. Van den Bergh's \cite{vdber:04} and developed in several directions; see M. Wemyss \cite{wem:18} for results and references.  
In \cite{gus/ile:18} we study deformation theory of pairs (rational surface singularity, MCM module) and show in particular that flops are obtained by blowing up parametrised singularities in parametrised modules. This was conjectured by C. Curto and D. Morrison to hold for the simple flops \cite{cur/mor:13}. Not only does the module add an extra layer of structure to the deformation theory of the singularity, it also seems that the pair can behave better with respect to deformations than the singularity or the module do alone.
In \cite[5.6]{gus/ile:18} we show how Theorem \ref{thm.vers} implies that there exists a versal henselian deformation of a pair (partial resolution of rational surface singularity, strict transform of MCM module). This result is used in the proofs of the Curto-Morrison conjectures.

For the deformation theory it is crucial to have a cohomology of pairs. It should contain both the algebra cohomology and the cohomology of modules as special cases. Our solution is to regard the pair $(A,N)$ as a graded ring $A\oplus N$ with $N^2=0$ and employ graded Andr{\'e}-Quillen cohomology. Section \ref{sec.coh} gives a summary of vital results based on L. Illusie's \cite{ill:71} and the Stacks Project \cite{SP}. 
In Theorem \ref{thm.lang} we establish a long exact sequence linking the $\Ext$ groups and the algebra cohomology with the cohomology groups ${}_{\ccirc\!}\mr{T}^{n}_{\vD/k}(\vD)$ of a pair $\vD=(k\ra A,N)$:
\begin{equation}\label{eq.cohomseq}
\begin{aligned}
0\ra\, & \hm{}{A}{N}{N}\ra \gder{k}{\vD}{\vD}\ra \der{k}{A}{A}\xra{\,\partial\,}\dots 
\\
\ra\, &\xt{n}{A}{N}{N}\ra {}_{\ccirc\!}\mr{T}^{n}_{\vD/k}(\vD)\ra \mr{T}^{n}_{A/k}(A)
\xra{\,\partial\,} \xt{n+1}{A}{N}{N}\ra\dots\,.
\end{aligned}
\end{equation}
From a deformation theoretic perspective we find this sequence rather suggestive. It can be viewed as a cohomological image of the sequence of deformation functors $\df{}{N}\ra\df{}{(A,N)}\ra \df{}{A}$.
The obstruction calculus of the pair is related to the obstruction calculus of the algebra and of the module through maps in \eqref{eq.cohomseq}. 
If $\vG=(S\ra\mc{A},\mc{N})$ is a deformation of the pair $\vD$ and $\pi\co S'\ra S$ is a small extension, there is an obstruction class $\ob(\pi,\vG)\in {}_{\ccirc\!}\mr{T}^{2}_{\vD/k}(\vD)\ot\ker\pi$ for deforming $\vG$ to $S'$. 
The image of $\ob(\pi,\vG)$ in $\mr{T}^{2}_{A/k}(A)\ot\ker\pi$ is the obstruction $\ob(\pi,\mc{A})$ for deforming the algebra $S\ra \mc{A}$ to $S'$. If $\ob(\pi,\mc{A})=0$ and $S'\ra\mc{A}'$ is a deformation of $S\ra \mc{A}$ there is an obstruction class $\ob(S'\ra\mc{A}',\mc{N})\in\xt{2}{A}{N}{N}\ot\ker\pi$ for deforming $\mc{N}$ to $S'$ as $\mc{A}'$-module which maps to $\ob(\pi,\vG)$. There are also compatibilities among the torsor actions. E.g. `adding' an element $\xi\in \mr{T}^{1}_{A/k}(A)\ot\ker\pi$ to $S'\ra\mc{A}'$ gives another deformation $S'\ra\mc{A}^{\dprime}$ of $S\ra \mc{A}$ and $\ob(S'\ra\mc{A}^{\dprime},\mc{N})=\ob(S'\ra\mc{A}',\mc{N})+\partial(\xi)$.
It follows from Theorem \ref{thm.lang} that the Zariski tangent space $\df{}{(A,N)}(k[\vare])$ is finite dimensional in the case both the algebra and the module have isolated singularity which is crucial for proving Theorem \ref{thm.vers}.

We define a cohomological Kodaira-Spencer class $\kappa(\vG/A/\vL)\in {}_{\ccirc\!}\mr{T}^{1}_{\vG/A}(\Omega_{A/\vL}\ot\vG)$ where $\vL\ra A\ra B$ are ring homomorphisms and $\vG=(B,N)$ is a pair. The Kodaira-Spencer class induces a Kodaira-Spencer map 
\begin{equation}
g^{\vG}\!\co \Der_{\vL}(A)\lra{}_{\ccirc\!}\mr{T}^{1}_{\vG/A}(\vG)\,.
\end{equation}
Again the long exact sequence in Theorem \ref{thm.lang} connects the various classes: The restriction maps $\kappa(\vG/A/\vL)$ to $\kappa(B/A/\vL)\in \mr{T}^{1}_{B/A}(\Omega_{A/\vL}\ot B)$. If $\kappa(B/A/\vL)=0$ and $\sigma\co B\ra P^1_{A/\vL}(B)$ is a chosen principal parts splitting there is a class $\kappa(\sigma,N)\in \xt{1}{B}{N}{\Omega_{A/\vL}\ot N}$ which maps to $\kappa(\vG/A/\vL)$. Moreover, $\kappa(\sigma,N)$ is induced by the Atiyah class \(\at_{B\!/\!\vL}({N})\linebreak[1]\in\xt{1}{B}{{N}}{\Omega_{B\!/\!\vL}\ot_{B}{N}}\) which corresponds to $\kappa(\vG/B/\vL)$. In particular the connecting map $\partial\co \hm{}{A}{\Omega_{A/k}}{A}\ra\xt{1}{A}{N}{N}$ in \eqref{eq.cohomseq} is given by $\partial(f)=(f\ot{\id})_*\at_{A/k}({N})$. See Proposition \ref{prop.at} for details. 
We expect Theorem \ref{thm.lang} to have many applications beyond the ones we give here.

As a special case we obtain the obstruction result for extensions of modules over a given extension of algebras in Corollary \ref{cor.ext} (e.g. `classical' deformation theory of modules). We also comment on some flaws in the treatment of this result in one of the standard references \cite{ill:71}; see Remark \ref{rem.ext} and Lemma \ref{lem.ext}. 

The first version of these results were given in \cite{ile:11x}\nocite{ile:19bx}. Y.-T. Siu and G. Trautmann constructed a versal deformation of a coherent analytic sheaf with compact support by complex-analytic techniques and showed that existence of a versal deformation of a pair (compact complex space, coherent analytic sheaf) follows by applying flatifying stratification; see \cite[\S 1 Thm. II]{siu/tra:81}. Since compact support is essential to this argument it cannot be applied to our situation.
L. Huang \cite{hua:95} and K. Chan and Y.-H. Suen \cite{cha/sue:16} have studied the deformation theory of pairs (compact complex manifold, holomorphic vector bundle) in terms a differential graded Lie algebra (DGLA) by complex analytic techniques and in particular showed that the Kuranishi family is versal. Chan and Suen also show compatibility of Kuranishi obstruction maps in the long-exact sequence of global cohomology for the (holomorphic) Atiyah extension analogous to \eqref{eq.cohomseq}.
Based locally on derivations of pairs, D. Iacono and M. Manetti \cite{iac/man:19} have given an algebraic DGLA defined in characteristic zero which `controls' infinitesimal deformations of pairs (smooth projective variety, coherent sheaf). 
Cf. \cite[Appendix A]{cha/sue:16}. 
%%%%%%%%%%%%%%%%% SUBSECTION %%%%%%%%%%%%%%%%%%%%
\subsection*{Applications}\label{sec.appli} 
Suppose $A$ is a Cohen Macaulay ring which is algebraic (or of finite type) over $k$. 
Then $A$ has a canonical (dualising) module $\omega_A$. 
To any finite $A$-module $N$ there are short exact sequences of $A$-modules
\begin{equation}\label{eq.CMseq0}
0\ra L\lra M\lra N\ra 0\qquad\text{and}\qquad 0\ra N\lra L'\lra M'\ra 0\,.
\end{equation}
where $M$ and $M'$ are maximal Cohen-Macaulay (MCM) and $L$ and $L'$ have finite injective dimension (FID). In the local case the sequences are unique. The theory of Cohen-Macaulay approximations was founded by M. Auslander and R.-O. Buchweitz in \cite{aus/buc:89}. The theory was extended to flat families in \cite{ile:12}. 
If \(h\co S\ra \mc{A}\) is a Cohen-Macaulay map (finite type or local) and \(\mc{N}\) an \(S\)-flat finite \(\mc{A}\)-module there are short exact sequences of \(S\)-flat finite \(\mc{A}\)-modules 
\begin{equation}\label{eq.CMseq}
0\ra \mc{L}\lra \mc{M}\xra{\,\;\pi\;\,} \mc{N}\ra 0\qquad\text{and}\qquad 0\ra \mc{N}\xra{\,\;\iota\;\,} \mc{L}'\lra \mc{M}'\ra 0
\end{equation}
such that the fibres of these sequences give Cohen-Macaulay approximation sequences as in \eqref{eq.CMseq0}.
In \cite{ile:21} these sequences were used to define maps 
\begin{equation}\label{eq.sigmaX}
\sigma_X\co\df{}{(A,N)}\lra\df{}{(A,X)}\quad\text{for}\quad X=M, M', L\,\,\text{and}\,\, L'
\end{equation}
of deformation functors of pairs (algebra, module) in the case $k\ra A$ is algebraic.
The applications of the deformation theory in this article are statements about these maps, in particular in the case when the pair $(A,N)$ has an isolated singularity. The flexibility the maps offer of switching to a module with nice homological properties makes such results attractive to the deformation theory of pairs.
%%%%%%%%%%%%% THEOREM %%%%%%%%%%%%%
\begin{thmA*}[Cf. \ref{thm.defext1} and \ref{thm.defext2}]
Suppose $\df{}{(A,N)}$\textup{,} $\df{}{(A,L')}$ and $\df{}{(A,M)}$ have formally versal elements with minimal base rings $R_{N}$\textup{,} $R_{L'}$ and $R_{M}$\textup{,} respectively\textup{.} 
\begin{enumerate}[leftmargin=2.4em, label=\textup{(\roman*)}]
\item If\, \(\xt{j}{A}{N}{L'}\ra\xt{j}{A}{N}{M'}\) is surjective for \(j=0,1\) then $R_{N}\cong R_{L'}/J$ where the ideal \(J\) is generated by linear forms\textup{.}

\item If\, \(\xt{j}{A}{M}{N}\ra\xt{j}{A}{L}{N}\) is surjective for \(j=0,1\) then $R_{N}\cong R_{M}/J$ where the ideal \(J\) is generated by linear forms\textup{.}
\end{enumerate}
\end{thmA*}
For isolated singularities, Theorem \ref{thm.vers} makes it possible to strengthen several results in \cite{ile:21}. 
%%%%%%%%%%%%% THEOREM %%%%%%%%%%%%%
\begin{thmB*}[Cf. \ref{thm.defgrade} and \ref{thm.defgrade2}]
Suppose \((A,N)\) has an isolated singularity\textup{.} Then\textup{:}
\begin{enumerate}[leftmargin=2.4em, label=\textup{(\roman*)}]
\item The functors \(\df{}{(A,X)}\) for \(X=N, M, M', L\,\,\text{and}\,\, L'\,\) have versal elements\textup{.}

\item If\, \(\xt{1}{A}{N}{M'}=0\) then the map \(\sigma_{L'}\co \df{}{(A,N)}\ra\df{}{(A,L')}\) is smooth\textup{.}

\item If\, \(\xt{1}{A}{L}{N}=0\) then the map \(\sigma_{M}\co \df{}{(A,N)}\ra\df{}{(A,M)}\) is smooth\textup{.}
\end{enumerate}
\end{thmB*}
The cohomological conditions in Theorems A and B are crude, but the results are very general and not empty. Several examples are given in Section \ref{sec.frame}. More can of course be said in particular cases, as in Example \ref{ex.dim1dim0} and Corollary \ref{cor.ADE}.

If we in Theorem B (ii) also know that \(\hm{}{A}{N}{M'}=0\), the map \(\sigma_{L'}\) is an isomorphism, as in the following corollary.
%%%%%%%%%%%% COROLLARY %%%%%%%%%%%%%
\begin{corA*}[Cf. \ref{thm.defgrade}]
Suppose \(A\) has an isolated singularity\textup{,} $\dim A\geq 2$ and $\dim N=0$\textup{.} Then $\sigma_{L'}\co \df{}{(A,N)}\lra \df{}{(A,L')}$ is an isomorphism\textup{.}
\end{corA*}
Note that in general $\df{}{(A,L')}\cong\df{}{(A,Q')}$ where $Q':=\hm{}{A}{\omega_A}{L'}$ has finite projective dimension; cf. \eqref{eq.LQ}. A result in \cite{ile:14} allows us to give some peculiar examples. 
%%%%%%%%%%%% COROLLARY %%%%%%%%%%%%%
\begin{corB*}[Cf. Section \ref{sec.frame}]
Suppose \(A\) has an isolated singularity\textup{,} $A/\fr{m}_A\cong k$ and 
$R\ra\mc{A}$ is a versal family for $\df{}{A}$\textup{.} 
\begin{enumerate}[leftmargin=2.4em, label=\textup{(\roman*)}]
\item If $\dim A\geq 2$ there is a finite module $Q'$ with $\pdim Q'=\dim A$ such that $\mc{A}$ is a versal base ring for $ \df{}{(A,Q')}$\textup{.} 
In particular\textup{,} if $A$ is a complete intersection then $\df{}{(A,Q')}$ is smooth\textup{.}
\item If $\dim A=2$\textup{,} $A$ is a complete intersection and $\mr{T}^1=\mr{T}^1_{A/k}(A)$\textup{,} there is finite module $Q'$ with $\pdim Q'=2$ such that $\df{}{(A,\mr{T}^1)}\cong \df{}{(A,Q')}$\textup{.}
\item If $\dim A= 1$ and $A$ is Gorenstein there is a maximal Cohen-Macaulay module $M$ such that $\mc{A}$ is a versal base ring for $\df{}{(A,M)}$\textup{.}
\end{enumerate}
\end{corB*}
Several of the results are combined to prove Corollary \ref{cor.ADE} which relates the versal base of $\df{}{(A,F)}$ to the versal family of $\df{}{A}$ where $F$ is the fundamental module of a rational double point in dimension two.

For some relevant references on Cohen-Macaulay approximation, see \cite{ile:21}.
%%%%%%%%%%%%%%%%% SUBSECTION %%%%%%%%%%%%%%%%%%%%
\subsection*{Summary of the contents}\label{sec.cont} 
In Section \ref{sec.coh} cohomology of pairs is defined and a cohomology-and-base-change result for graded Andr{\'e}-Quillen cohomology is given in Proposition \ref{prop.bc}. Proposition \ref{prop.ext} gives the obstruction class, the torsor actions and the automorphisms for extensions of graded algebras while Proposition \ref{prop.omap} gives the necessary functoriality of classes and actions. In particular, lllusie's definition of the obstruction is compared to the definition in The Stacks Project \cite{SP}. Corollary \ref{cor.ext} and Lemma \ref{lem.ext} concerns the classical case of extending modules where the algebra extension is given (e.g. trivial). These results culminates with the long exact sequence in Theorem \ref{thm.lang} relating the various obstruction classes and actions. In Section \ref{sec.KS} we define the Kodaira-Spencer class and map for pairs and relates the various classes in Proposition \ref{prop.at}. 
Section \ref{sec.versal1} introduces terminology and recalls results concerning Artin's \cite{art:74},
in particular H. Flenner's condition for formal versality in terms of the category of infinitesimal extensions; see Proposition \ref{prop.ex}, and von Essen's simplification of \cite[3.3]{art:74}; see Proposition \ref{prop.fvers}.
In Section \ref{sec.versal2} we show that the general conditions for using \cite[3.3]{art:74} are fulfilled for our cofibred category of deformations of pairs and prove the versality result Theorem \ref{thm.vers}. Corollary \ref{cor.vers} is an `expected' application: the forgetful $\df{}{(A,N)}\ra \df{}{A}$ is smooth if $\xt{2}{A}{N}{N}=0$. Section \ref{sec.cohmap} contains results for the maps of cohomology obtained by Cohen-Macaulay approximation, while Section \ref{sec.appl} contains the applications to the induced maps of deformation functors. Section \ref{sec.frame} introduces deformations of framings. Proposition \ref{prop.Ak} shows that forgetting the framing is a smooth map. Examples follow from a description of the versal base for $\df{}{(A,k)}$ and from K{\"a}hler differentials. The syzygy induces a map of $\Ext$s and in Section \ref{sec.ADE} we show that it anti-commutes with the Kodaira-Spencer maps. Finally Corollary \ref{cor.ADE} compares versal bases of $\df{}{(A,F)}$ and $\df{}{(A,k)}$ if $F$ is the fundamental module of a $2$-dimensional rational double point.

All rings are commutative with $1$-element. 
%%%%%%%%%%%%%%%%%%%%%%%%%%%%%%%%%%%%%%%
%%%%%%%%%%%%%%%%%% SECTION %%%%%%%%%%%%%%%%
%%%%%%%%%%%%%%%%%%%%%%%%%%%%%%%%%%%%%%%
\section{Cohomology of pairs and obstruction theory}\label{sec.coh}
We will need some results concerning the cohomology of pairs $(B,N)$ where $B$ is an $A$-algebra and $N$ is a $B$-module. Then \(\vG=B{\oplus}N\) is a graded \(A\)-algebra with \(B\) in degree \(0\) and \(N\) in degree \(1\). Maps of pairs \((B_{1},N_{1})\ra (B_{2},N_{2})\) correspond to maps of 
graded $A$-algebras \(\vG_1\ra \vG_2\). We generalise and consider a homogeneous morphism of graded rings \(g\co A\ra \vG\) (always \(\BB{Z}_{\geq 0}\)-graded) and a ($\BB{Z}$\,-)graded \(\vG\)-module \(J\). There are Andr{\'e}-Quillen cohomology groups
\begin{equation}
\gxt{i}{\vG}{L^{\text{gr}}_{\vG/A}}{J}=\hm{}{D(\gmod_{\vG})}{L^{\text{gr}}_{\vG/A}}{J[i]}
\end{equation}
which also is called graded algebra cohomology.
Here \(L^{\text{gr}}_{\vG/A}\) is the graded cotangent complex defined as \(\Omega_{P/A}\ot_{P}\vG\) where \(P=P^{\text{gr}}_{A}(\vG)\) is a graded simplicial degree-wise free \(A\)-algebra resolution of \(\vG\) and \(\Omega_{P/S}\) denotes the associated complex of K{\"a}hler differentials. To a \(\BB{Z}_{\geq 0}\)-graded set \(U\) there is a `graded free' \(A\)-algebra \(A[U]\) -- the polynomial ring over \(A\) with free variables \(x_{u}\) indexed by elements \(u\in U\) with \(\deg x_{u}=\deg u\). In particular there is a canonical graded \(A\)-algebra (augmentation) map \(A[\vG]\ra \vG\) defined by \(x_{\gamma}\mapsto \gamma\) with kernel \(K\). Iterating the procedure gives
\begin{equation}
P^{\text{gr}}_{A}(\vG)\co\hspace{1em}
\xymatrix@C+0pt@R-8pt@H-0pt{
\cdots\hspace{0.5em} A[A[A[\vG]]]\ar@<1.6ex>[r]\ar[r]\ar@<-1.6ex>[r] & A[A[\vG]]\ar@<0.8ex>[l]\ar@<-0.8ex>[l]\ar@<0.8ex>[r]\ar@<-0.8ex>[r] & A[\vG]\ar[l]
}
\end{equation}
which is functorial in $A\ra \vG$; see \cite[Chap. IV, 1.3]{ill:71}. 
The two term complex \(K/K^{2}\ra\Omega_{A[\vG]/A}\ot_{A[\vG]}\vG\) obtained from the augmentation map is denoted by \(\NL^{\text{gr}}_{\vG/A}\). There is a natural map of complexes \(L^{\text{gr}}_{\vG/A}\ra\NL^{\text{gr}}_{\vG/A}\) 
which induces the equivalence \(\tau^{\geq -1}L^{\text{gr}}_{\vG/A}\simeq\NL^{\text{gr}}_{\vG/A}\); cf. \cite[\href{https://stacks.math.columbia.edu/tag/08RB}{Tag 08RB}]{SP}. 
If \(A\ra B\ra \vG\) are maps of graded rings then there is a distinguished triangle of transitivity in \(D(\gmod_{\vG})\): 
\begin{equation}\label{eq.dist}
L^{\tn{gr}}_{\vG/B/A}\co\hspace{1em}
L^{\tn{gr}}_{B/A}\ot^{\BB{L}}_{B}\vG\lra L^{\tn{gr}}_{\vG/A}\lra L^{\tn{gr}}_{\vG/B}\lra L^{\tn{gr}}_{B/A}\ot^{\BB{L}}_{B}\vG[1]
\end{equation}
See \cite[Chap. IV, 2.3]{ill:71}.
Note that the \(\gxt{i}{\vG}{L^{\text{gr}}_{\vG/A}}{J}\) are \(\vG_{0}\)-modules, but in general not \(\vG\)-modules. However, 
\begin{equation}
\sgxt{i}{\vG}{L^{\text{gr}}_{\vG/A}}{J}:=\bigoplus_{n\in\BB{Z}}\gxt{i}{\vG}{L^{\text{gr}}_{\vG/A}}{J[n]}
\end{equation}
is a \(\vG\)-module. It is closely related to the global Andr{\'e}-Quillen cohomology of \(\Proj \vG\); cf. H.C. Pinkham's \cite{pin:74} and J. Kleppe's \cite{kle:79}.
 
If \(\vG\) is a finitely generated \(A\)-algebra, \(A\) is Noetherian and \(J\) is finite as \(\vG\)-module, then \(\sgxt{i}{\vG}{L^{\text{gr}}_{\vG/A}}{J}\) is finite as \(\vG\)-module and equals the ungraded $\xt{i}{\vG}{L^{\text{gr}}_{\vG/A}}{J}$. See \cite[\href{https://stacks.math.columbia.edu/tag/08PV}{Tag 08PV}]{SP}. 
%%%%%%%%%%%LEMMA%%%%%%%%%%%%%
\begin{lem}[{cf. \cite[\href{https://stacks.math.columbia.edu/tag/08QQ}{Tag 08QQ}]{SP}}]\label{lem.bc}%2
Given a commutative diagram of graded ring homomorphisms 
\begin{equation*}
\xymatrix@C+0pt@R-6pt@H-0pt{
\vG' & \vG\ar[l]
\\
A'\ar[u] & A\ar[l]\ar[u]
}
\end{equation*}
such that the induced map \(\vG\ot_{A} A'\ra\vG'\) is an isomorphism and \(\tor{A}{i}{\vG}{A'}=0\) for all \(i>0\)\textup{.} Then the natural map of complexes \(L^{\tn{gr}}_{\vG/A}\ot^{\BB{L}}_{\!A}A'\ra L^{\tn{gr}}_{\vG'/A'}\) is an isomorphism in \(D(\gmod_{\vG})\)\textup{.}
\end{lem}
Suppose \(\vG\) is \(A\)-flat and \(J'\) is a graded \(\vG'\)-module. Lemma \ref{lem.bc} implies that there is an isomorphism of \(\vG'\)-modules: 
\begin{equation}\label{eq.e}
\sgxt{n}{\vG'}{L^{\tn{gr}}_{\vG'/A'}}{J'}\xra{\;\simeq\;}\sgxt{n}{\vG}{L^{\tn{gr}}_{\vG/A}}{J'}
\end{equation}
If \(I'\) is a graded \(A'\)-module, there is an exchange map (cf. \cite[Sec. 4]{ogu/ber:72} or \cite[7.2.2]{EGAIII2}):
\begin{equation}\label{eq.exc}
e^{n}_{I'}\co \sgxt{n}{\vG}{L^{\tn{gr}}_{\vG/A}}{J}\ot_{A}I'\lra\sgxt{n}{\vG}{L^{\tn{gr}}_{\vG/A}}{J\ot_{A}I'}
\end{equation}
Composing \(e_{I'}^{i}\) with the inverse of \eqref{eq.e} gives the base change map
\begin{equation}
c^{n}_{I'}\co \sgxt{n}{\vG}{L^{\tn{gr}}_{\vG/A}}{J}\ot_{A}I'\lra\sgxt{n}{\vG'}{L^{\tn{gr}}_{\vG'/A'}}{J\ot_{A}I'}
\end{equation} 
%%%%%%%%%%%%PROPOSITION%%%%%%%%%%%%%%%%%%
\begin{prop}\label{prop.bc}%1(3)
Suppose \(A\) is a local noetherian ring \textup{(}concentrated in degree 0\textup{)} with residue field \(k\)\textup{,} \(\vG\) is a graded and flat \(A\)\textup{-}algebra and \(J\) is an \(A\)-flat\textup{,} finite and graded \(\vG\)-module\textup{.} Assume that the following finiteness condition holds for all \(i\)\textup{:} 
\begin{itemize}[leftmargin=2.4em, label=\((\ast)\)]
\item \(\sgxt{i}{\vG}{L^{\tn{gr}}_{\vG/A}}{J\ot_{\!A} I}\) is a finite \(\vG\)-module for all finite \(A\)-modules \(I\)\textup{.}
\end{itemize}
Suppose the base change map
\begin{equation*}
c^{n}_{k}\co \sgxt{n}{\vG}{L^{\tn{gr}}_{\vG/A}}{J}\ot_{\!A} k\lra\sgxt{n}{\vG\ot k}{L^{\tn{gr}}_{\vG\ot k/k}}{J\ot_{\!A} k}
\end{equation*}
is surjective\textup{.} Then\textup{:}
\begin{enumerate}[leftmargin=2.4em, label=\textup{(\roman*)}]
\item For all maps of local rings \(A\ra A'\) and \(A'\)-modules \(I'\)\textup{,} the base change map \(c_{I'}^{n}\) is an isomorphism\textup{.}
\item The following statements are equivalent\textup{:}
\begin{enumerate}[leftmargin=1.8em, label=\textup{(\alph*)}]
\item \(c_{k}^{n-1}\) is surjective\textup{.}
\item The \(\vG\)-module \(\sgxt{n}{\vG}{L^{\tn{gr}}_{\vG/A}}{J}\) is \(A\)-flat\textup{.}
\end{enumerate} 
\end{enumerate}
\end{prop}
%%%%%%%%%%%%%
\begin{proof}
Put $F^{n}(I)=\sgxt{n}{\vG}{L^{\tn{gr}}_{\vG/A}}{J\ot_{\!A} I}$. Since \(J\) is \(A\)-flat \(\{F^{n}\}_{n\in\BB{Z}}\) is a \(\delta\)-functor satisfying the conditions of \cite[5.1-2]{ogu/ber:72} which gives the result.
\end{proof}
%%%%%%%%%LEMMA%%%%%%%%%%%%%%
\begin{lem}\label{lem.ft}%2
Suppose \(A\) is a local noetherian ring \textup{(}concentrated in degree 0\textup{),} \(\vG^{\tn{f}}\) is a graded and flat \(A\)-algebra\textup{,} \(J^{\tn{f}}\) a graded and finite \(\vG^{\tn{f}}\)-module\textup{.} Let \(B^{\tn{f}}\) denote the degree \(0\)-part \(\vG^{\tn{f}}_{0}\) and suppose \(B^{\tn{f}}\ra B\) is a flat ring homomorphism\textup{.} Put \(\vG=\vG^{\tn{f}}\ot_{B^{\tn{f}}}B\) and \(J=J^{\tn{f}}\ot_{B^{\tn{f}}}B\)\textup{.} 
Assume that the following condition holds\textup{:}
\begin{itemize}[leftmargin=2.4em, label=\((\ast\ast)\)]
\item The natural map \(L_{B^{\tn{f}}/A}\ot_{B^{\tn{f}}}B\lra L_{B/A}\) is an equivalence\textup{.}
\end{itemize}
Then\textup{:} 
\begin{enumerate}[leftmargin=2.4em, label=\textup{(\roman*)}]
\item The natural maps 
\begin{equation*}
\sgxt{i}{\vG}{L^{\tn{gr}}_{\vG/A}}{J}\lra \sgxt{i}{\vG^{\tn{f}}}{L^{\tn{gr}}_{\vG^{\tn{f}}/A}}{J}\lla\sgxt{i}{\vG^{\tn{f}}}{L^{\tn{gr}}_{\vG^{\tn{f}}/A}}{J^{\tn{f}}}\ot_{B^{\tn{f}}}B
\end{equation*}  
are isomorphisms of \(\vG\)-modules for all \(i\)\textup{.} 
\item If \((\vG^{\tn{f}},J^{\tn{f}})\) satisfies the finiteness condition \((\ast)\) in \textup{Proposition \ref{prop.bc},} so does \((\vG,J)\)\textup{.} 
\end{enumerate}
In particular\textup{,} \((\ast\ast)\) and \((\ast)\) are satisfied if \(\vG^{\tn{f}}\) is of finite type as \(A\)-algebra and \(B^{\tn{f}}\ra B\) is a Zariski localisation or a henselisation\textup{.}
\end{lem}
%%%%%%%%%%%%%
\begin{proof}
The map of ring maps \((A\ra B^{\tn{f}}\ra\vG^{\tn{f}})\ra(A\ra B\ra\vG)\) gives a map of distinguished triangles \(L^{\tn{gr}}_{\vG^{\tn{f}}/B^{\tn{f}}/A}\ot_{B^{\tn{f}}}B\ra L^{\tn{gr}}_{\vG/B/A}\). By Lemma \ref{lem.bc} and $(\ast\ast)$ two out of three maps are equivalences, hence also the third; \(L^{\tn{gr}}_{\vG^{\tn{f}}/A}\ot_{B^{\tn{f}}}B\simeq L^{\tn{gr}}_{\vG/A}\), which gives the first isomorphism in (i). The second isomorphism follows from Lazard's theorem \cite[\href{https://stacks.math.columbia.edu/tag/058G}{Tag 058G}]{SP} since taking inductive limits commutes with \(\Hom\) in the second factor and is exact; \cite[\href{https://stacks.math.columbia.edu/tag/00DB}{Tag 00DB}]{SP}. More precisely, suppose \(\liminj_{j}I_{j}=B\) where \(I_{j}\) are finite and free \(B^{\tn{f}}\)-modules. Then: 
\begin{equation}
\begin{aligned}
&\textstyle{\liminj_{j}}\sgxt{i}{\vG^{\tn{f}}}{L^{\tn{gr}}_{\vG^{\tn{f}}/A}}{J^{\tn{f}}}\ot_{B^{\tn{f}}}I_{j}
\cong \textstyle{\liminj_{j}}\textstyle{\bigoplus}_{\nu}\cH^{i}\ghm{\bdot}{\vG^{\tn{f}}}{L^{\tn{gr}}_{\vG^{\tn{f}}/A}}{J^{\tn{f}}(\nu)}\ot_{B^{\tn{f}}}I_{j}
\\
&\cong\textstyle{\bigoplus}_{\nu}\liminj_{j}\cH^{i}\ghm{\bdot}{\vG^{\tn{f}}}{L^{\tn{gr}}_{\vG^{\tn{f}}/A}}{J^{\tn{f}}(\nu)\ot_{B^{\tn{f}}}I_{j}}
\\
&\cong \textstyle{\bigoplus}_{\nu}\cH^{i}\ghm{\bdot}{\vG^{\tn{f}}}{L^{\tn{gr}}_{\vG^{\tn{f}}/A}}{J^{\tn{f}}(\nu)\ot_{B^{\tn{f}}}\liminj_{j}I_{j}}
\\
&\cong \sgxt{i}{\vG^{\tn{f}}}{L^{\tn{gr}}_{\vG^{\tn{f}}/A}}{J^{\tn{f}}\ot_{B^{\tn{f}}}\textstyle{\liminj_{j}I_{j}}}
\end{aligned}
\end{equation}
(ii) follows directly from (i). For the last part see \cite[\href{https://stacks.math.columbia.edu/tag/08QY}{Tag 08QY}]{SP}.
\end{proof}
%%%%%%%%%%%% DEFINITION %%%%%%%%%%%%
\begin{defn}\label{def.ext}%6
An extension of a graded ring $A$ by a graded $A$-module $I$ is a surjective graded ring homomorphism \(A'\ra A\) with kernel \(I\) such that \(I^{2}=0\). Let \(\alpha\) denote the extension \(I\ra A'\ra A\). Given a graded ring homomorphism \(A\ra \vG\), a graded \(\vG\)-module \(J\) and a graded \(A\)-linear map \(\gamma\co I\ra J\). A $\gamma$-extension \(\beta\) of \(\vG\) by \(J\) above \(\alpha\) is a commutative diagram of extensions 
\begin{equation}
\begin{aligned}
\xymatrix@C+0pt@R-9pt@H-6pt{
\beta : & 0\ar[r] & J\ar[r] & \vG'\ar[r] & \vG\ar[r] & 0
\\
\alpha: & 0\ar[r] & I\ar[r]\ar[u]^{\gamma} & A'\ar[r]\ar[u] & A\ar[r]\ar[u] & 0
}
\end{aligned}
\end{equation}
where \(A'\ra\vG'\) is a graded ring homomorphism.
\end{defn}
%%%%%%%%%%%%%% PROPOSITION %%%%%%%%%%%%%%
\begin{prop}\label{prop.ext}%11(13)
Consider the situation in \textup{Definition \ref{def.ext}.}
\begin{enumerate}[leftmargin=2.4em, label=\textup{(\roman*)}]
\item There exists an element \(\ob(\alpha,\vG,\gamma)\in \gxt{2}{\vG}{L^{\tn{gr}}_{\vG/A}}{J}\) such that a $\gamma$-extension \(\beta\) above \(\alpha\) exists if and only if \(\ob(\alpha,\vG,\gamma)=0\)\textup{.} 
\item If \(\ob(\alpha,\vG,\gamma)=0\) the set of isomorphism classes of $\gamma$-extensions \(\beta\) above \(\alpha\) is a torsor over \(\gxt{1}{\vG}{L^{\tn{gr}}_{\vG/A}}{J}\)\textup{.}
\item The set of automorphisms of a given $\gamma$-extension \(\beta\) above \(\alpha\) is canonically isomorphic to \(\gxt{0}{\vG}{L^{\tn{gr}}_{\vG/A}}{J}\cong\gDer_{A}(\vG,J)\)\textup{.}
\end{enumerate}
\end{prop}
If all rings and modules are concentrated in degree \(0\) the cohomology groups equals the ungraded Andr{\'e}-Quillen cohomology and Proposition \ref{prop.ext} is \cite[III 2.1.2.3]{ill:71} and \cite[\href{https://stacks.math.columbia.edu/tag/08SP}{Tag 08SP}]{SP}. The proof carries over to the graded case.
 
We recall the definition of the obstruction class in \cite[\href{https://stacks.math.columbia.edu/tag/08SM}{Tag 08SM}]{SP} nominally extended to the graded case. Note that $\alpha$ (up to equivalence of extensions) can be considered as an element in $\gxt{1}{A}{L^{\tn{gr}}_{A/\BB{Z}}}{I}$; cf. \cite[\href{https://stacks.math.columbia.edu/tag/08S8}{Tag 08S8}]{SP}.
The obstruction class \(\ob(\alpha,\vG,\gamma)\) in Proposition \ref{prop.ext} is defined as the image of \(\alpha\) along the natural maps 
\begin{equation}
\gxt{1}{A}{L^{\tn{gr}}_{A/\BB{Z}}}{I}\xra{\hspace{0.2em}\gamma_{*}\hspace{0.2em}}\gxt{1}{A}{L^{\tn{gr}}_{A/\BB{Z}}}{J}\cong{}\gxt{1}{\vG}{L^{\tn{gr}}_{A/\BB{Z}}\ot^{\BB{L}}_{A}\vG}{J}
\end{equation}
composed with the connecting \(\partial=\partial(L^{\tn{gr}}_{\vG/A/\BB{Z}})\) in the Jacobi-Zariski sequence obtained from \eqref{eq.dist}:
\begin{equation}
\dots\ra\gxt{1}{\vG}{L^{\tn{gr}}_{\vG/\BB{Z}}}{J}\ra \gxt{1}{\vG}{L^{\tn{gr}}_{A/\BB{Z}}\ot^{\BB{L}}_{A}\vG}{J}\xra{\hspace{0.2em}\partial\hspace{0.2em}}\gxt{2}{\vG}{L^{\tn{gr}}_{\vG/A}}{J}\ra\dots
\end{equation}
In \cite[III 2.2.3]{ill:71} the definition of \(\ob(\alpha,\vG,\gamma)\) is given as the image of \(\gamma\) along the natural maps
\(\ghm{}{A}{I}{J}\cong\gxt{1}{A}{L^{\tn{gr}}_{A/A'}}{J}\cong\gxt{1}{\vG}{L^{\tn{gr}}_{A/A'}\ot^{\BB{L}}\vG}{J}\); see \cite[III 1.2.8.1]{ill:71}, composed with $\partial(L^{\tn{gr}}_{\vG/A/A'})\co \gxt{1}{\vG}{L^{\tn{gr}}_{A/A'}\ot^{\BB{L}}\vG}{J}\lra \gxt{2}{\vG}{L^{\tn{gr}}_{\vG/A}}{J}$.
The two definitions agree by inspection of the following commutative diagram of natural maps
\begin{equation}\label{eq.ill}
\begin{aligned}
\xymatrix@C-12pt@R-6pt@H-0pt{
\hspace{-1.6em}\id\in\gnd{}{A}{I}\ar@<-0.48em>[d]^(0.47){\gamma_{*}}\ar[r]^(0.45){\simeq} 
&
\hspace{-0.12em}\gxt{1}{A}{L^{\tn{gr}}_{A/A'}}{I}\ar@<-2.2em>[d]^(0.45){\gamma_{*}}\ar[r] 
&
\gxt{1}{A}{L^{\tn{gr}}_{A/\BB{Z}}}{I}\ni\alpha\hspace{0.45em}\ar@<-3.1em>[d]^(0.45){\gamma_{*}} 
\\
\gamma\in\ghm{}{A}{I}{J}\ar[r]^(0.48){\simeq} 
&
\gxt{1}{A}{L^{\tn{gr}}_{A/A'}}{J}\ar@<-2.2em>[d]^(0.45){\simeq}\ar[r] &
\gxt{1}{A}{L^{\tn{gr}}_{A/\BB{Z}}}{J}\hspace{2.15em}\ar@<-3.1em>[d]^(0.45){\simeq}
\\
&
\hspace{2.16em}\gxt{1}{\vG}{L^{\tn{gr}}_{A/A'}\ot^{\BB{L}}\vG}{J}\ar@<-2.2em>[d]^{\partial(L^{\tn{gr}}_{\vG/A/A'})}\ar[r] 
&
\gxt{1}{\vG}{L^{\tn{gr}}_{A/\BB{Z}}\ot^{\BB{L}}\vG}{J}\ar@<-3.1em>[d]^{\partial(L^{\tn{gr}}_{\vG/A/\BB{Z}})}
\\
&
\hspace{-4.75em}\ob^{\tn{Illusie}}\in\gxt{2}{\vG}{L^{\tn{gr}}_{\vG/A}}{J}\ar@{=}[r] &
\gxt{2}{\vG}{L^{\tn{gr}}_{\vG/A}}{J}\ni\ob^{\tn{SP}}\hspace{-1.2em}
}
\end{aligned}
\end{equation}
where the three upper maps from the second to the third column are induced from \(L^{\tn{gr}}_{A/A'/\BB{Z}}\) while the lower square comes from the map \(L^{\tn{gr}}_{\vG/A/\BB{Z}}\ra L^{\tn{gr}}_{\vG/A/A'}\).

Assume \(A\) and \(A'\) are concentrated in degree \(0\). There is a \(\gamma'\co I\ot\vG\ra J\) induced from $\gamma$ in Definition \ref{def.ext}. If $\gamma'_1$ and $\gamma'_2$ are two isomorphisms \(I\ot \vG\cong J\), put $\theta=\gamma'_2(\gamma'_1)^{-1}$. Then \(\ob(\alpha,\vG,\gamma_2)=\theta_*\ob(\alpha,\vG,\gamma_1)\). If $\beta_1$ is a $\gamma_1$-extension of $\vG$ above $\alpha$ then $\beta_2=\theta_*\beta_1$ is a $\gamma_2$-extension. In the case \(\gamma'=\id\) we write \(\ob(\alpha,\vG)\) or \(\ob(A'\ra A,\vG)\) for \(\ob(\alpha,\vG,\gamma)\).

A $\gamma$-extension \((A'\ra\vG')\ra(A\ra\vG)\) is called a \emph{deformation} of \(A\ra\vG\) along \(A'\ra A\) if $\gamma'$ is an isomorphism. In that case the natural map \(\vG'\ot_{A'}A\ra \vG\) is an isomorphism and \(\tor{A'}{1}{\vG'}{A}=0\). If in addition \(\vG\) is \(A\)-flat then \(\vG'\) is \(A'\)-flat. 
%%%%%%%%%%%%%% PROPOSITION %%%%%%%%%%%%%%
\begin{prop}\label{prop.omap}%8(10)
Given an extension \(\alpha\) as in \textup{Definition \ref{def.ext},} graded ring homomorphisms \(A\ra\vG_{1}\xra{\hspace{0.2em}\rho\hspace{0.2em}}\vG_{2}\)\textup{,} a graded \(\vG_{i}\)-module \(J_{i}\) with \(A\)\textup{-}linear graded map \(\gamma_{i}\co I\ra J_{i}\) for \(i=1,2\) and a \(\vG_{1}\)\textup{-}linear graded map \(\sigma\co J_{1}\ra J_{2}\) such that \(\gamma_{2}=\sigma\gamma_{1}\)\textup{.}
\begin{enumerate}[leftmargin=2.4em, label=\textup{(\roman*)}]
\item The natural maps 
\begin{equation*}
\gxt{2}{\vG_{1}}{L^{\tn{gr}}_{\vG_{1}/A}}{J_{1}}\xra{\hspace{0.2em}\sigma_{*}\hspace{0.2em}} \gxt{2}{\vG_{1}}{L^{\tn{gr}}_{\vG_{1}/A}}{J_{2}}\xla{\hspace{0.25em}\rho^{*}\hspace{0.15em}}\gxt{2}{\vG_{2}}{L^{\tn{gr}}_{\vG_{2}/A}}{J_{2}}
\end{equation*}
relate the obstruction classes\textup{:} 
\begin{equation*}
\sigma_{*}\ob(\alpha,\vG_{1},\gamma_{1})=\ob(\alpha,\vG_{1},\gamma_{2})=\rho^{*}\ob(\alpha,\vG_{2},\gamma_{2})
\end{equation*}
\item The torsor actions \textup{(}denoted by $+$\textup{)} are compatible with \(\rho^{*}\) and \(\sigma_{*}\)\textup{:} given a $\gamma_i$-extension \(\beta_{i}\co 0\ra J_{i}\ra \vG_{i}'\ra \vG_i\ra 0\) above \(\alpha\) and \(\xi_{i}\in\gxt{1}{\vG_{i}}{L^{\tn{gr}}_{\vG_{i}/A}}{J_{i}}\) for $i=1,2$\textup{,} then \(\sigma_{*}(\beta_{1}+\xi_{1})=\sigma_{*}\beta_{1}+\sigma_{*}\xi_{1}\) and \(\rho^{*}(\beta_{2}+\xi_{2})=\rho^{*}\beta_{2}+\rho^{*}\xi_{2}\) in \(\gxt{1}{\vG_{1}}{L^{\tn{gr}}_{\vG_{1}/\BB{Z}}}{J_{2}}\)\textup{.} If there is a commutative diagram of $\gamma_i$-extensions
\begin{equation*}
\xymatrix@C-12pt@R-9pt@H-0pt{
\beta_{2}\co & 0\ar[r] & J_{2}\ar[r] & \vG_{2}'\ar[r] & \vG_{2}\ar[r] & 0
\\
\beta_{1}\co & 0\ar[r] & J_{1}\ar[r]\ar[u]^{\sigma} & \vG_{1}'\ar[r]\ar[u] & \vG_{1}\ar[r]\ar[u]_{\rho} & 0
}
\end{equation*}
above $\alpha$ then \(\sigma_{*}\beta_{1}=\rho^{*}\beta_{2}\) in \(\gxt{1}{\vG_{1}}{L^{\tn{gr}}_{\vG_{1}/\BB{Z}}}{J_{2}}\)\textup{.}
\end{enumerate}
\end{prop}
%%%%%%%%%%%%%
\begin{proof}
(i) The following pointed commutative diagram of canonical maps
\begin{equation}\label{eq.omap}
\xymatrix@C-0pt@R-6pt@H-0pt{
\hspace{-4.1em}(\gamma_2)_*\alpha\in\gxt{1}{A}{L^{\tn{gr}}_{A/\BB{Z}}}{J_{2}}\ar@<-2.2em>[d]^{\partial(L^{\tn{gr}}_{\vG_{2}/A/\BB{Z}})}\ar@{=}[r] &
\gxt{1}{A}{L^{\tn{gr}}_{A/\BB{Z}}}{J_{2}}\ar@<-2.9em>[d]^{\partial(L^{\tn{gr}}_{\vG_{1}/A/\BB{Z}})}\ni(\gamma_2)_*\alpha
\hspace{-2.6em}
\\
\hspace{-6.0em}\ob(\alpha,\vG_{2},\gamma_{2})\in\gxt{2}{\vG_2}{L^{\tn{gr}}_{\vG_{2}/A}}{J_{2}}\ar[r]^(0.5){\rho^{*}} &
\gxt{2}{\vG_1}{L^{\tn{gr}}_{\vG_{1}/A}}{J_{2}}\ni \ob(\alpha,\vG_{1},\gamma_{2})\hspace{-5.8em}
}
\end{equation}
shows that \(\rho^{*}\ob(\alpha,\vG_{2},\gamma_{2})=\ob(\alpha,\vG_{1},\gamma_{2})\). Similarly, the pointed commutative diagram of canonical maps
\begin{equation}\label{eq.omap2}
\begin{aligned}
\xymatrix@C-0pt@R-6pt@H-0pt{
(\gamma_{1})_{*}\alpha\in\gxt{1}{A}{L^{\tn{gr}}_{A/\BB{Z}}}{J_{1}}\ar@<-0.15em>[d]^{\partial(L^{\tn{gr}}_{\vG_{1}/A/\BB{Z}})}\ar[r]^(0.49){\sigma_{*}} 
& \gxt{1}{A}{L^{\tn{gr}}_{A/\BB{Z}}}{J_{2}}\ni (\gamma_{2})_{*}\alpha\ar@<-4.22em>[d]^{\partial(L^{\tn{gr}}_{\vG_{1}/A/\BB{Z}})}
\\
\hspace{-2.2em}\ob(\alpha,\vG_{1},\gamma_{1})\in\gxt{2}{\vG}{L^{\tn{gr}}_{\vG_{1}/A}}{J_{1}}\ar[r]^(0.5){\sigma_{*}} 
& \gxt{2}{B}{L^{\tn{gr}}_{\vG_{1}/A}}{J_{2}}\ni\ob(\alpha,\vG_{1},\gamma_{2})\hspace{-2.96em}
}
\end{aligned}
\end{equation}
shows that \(\sigma_{*}\ob(\alpha,\vG_{1},\gamma_{1})=\ob(\alpha,\vG_{1},\gamma_{2})\).

(ii) The commutative diagram of ring maps 
\begin{equation}
\begin{aligned}
\xymatrix@C-3pt@R-9pt@H-0pt{
\BB{Z}\ar[r]\ar@{=}[d] & A\ar[r]\ar@{=}[d] & \vG_{2}
\\
\BB{Z}\ar[r] & A\ar[r] & \vG_{1}\ar[u]
}
\end{aligned}
\end{equation} 
induces the following, partially pointed (left resp. right) commutative diagram of canonical \(A\)-linear maps:
\begin{equation}\label{eq.torsor2}
\xymatrix@C-6pt@R-6pt@H-0pt{
\hspace{0.96em}\xi_{2}\in\gxt{1}{\vG_{2}}{L^{\tn{gr}}_{\vG_{2}/A}}{J_{2}}\ar@<-1em>[d]^(0.4){\rho^{*}}\ar[r]^{\theta_{2}}
&
\gxt{1}{\vG_{2}}{L^{\tn{gr}}_{\vG_{2}/\BB{Z}}}{J_{2}}\ni\beta_{2}\ar@<-3.9em>[d]^(0.4){\rho^{*}}
\ar[r]\hspace{0.58em}
&
\gxt{1}{A}{L^{\tn{gr}}_{A/\BB{Z}}}{J_{2}}\ni\alpha\ar@<-3.1em>@{=}[d]
\\
\rho^{*}\xi_{2}\in\gxt{1}{\vG_{1}}{L^{\tn{gr}}_{\vG_{1}/A}}{J_{2}}\ar[r]^(0.48){\theta_{1}}
&
\hspace{0.36em}\gxt{1}{\vG_{1}}{L^{\tn{gr}}_{\vG_{1}/\BB{Z}}}{J_{2}}\ni\rho^{*}\beta_{2}\ar[r] 
&
\gxt{1}{A}{L^{\tn{gr}}_{A/\BB{Z}}}{J_{2}}\ni\alpha
}
\end{equation}
Let \(K_{A,i}\) and \(K_{\BB{Z},i}\) denote the kernels of the augmentation maps of the group algebras \(A[\vG_{i}]\ra A\) and \(\BB{Z}[\vG_{i}]\ra \BB{Z}\). Then \(\beta_{i}\) is represented in \(\gxt{1}{\vG_{i}}{\NL_{\vG_{i}/\BB{Z}}}{J_{i}}\) by a map \(\tilde{\beta}_{i}\co K_{\BB{Z},i}/K_{\BB{Z},i}^{2}\ra J_{i}\). Similarly, \(\xi_{i}\) is represented by a map \(\tilde{\xi}_{i}\co K_{A,i}/K_{A,i}^{2}\ra J_{i}\). The map \(\theta_{i}\) is induced by pullback along the natural map \(\iota_{i}\co \NL_{\vG_{i}/\BB{Z}}\ra \NL_{\vG_{i}/A}\). One obtains a (not necessarily commutative) diagram
\begin{equation}\label{eq.tri}
\xymatrix@C-3pt@R-9pt@H-0pt{
K_{\hspace{-0.1em}\BB{Z},i}/K_{\hspace{-0.1em}\BB{Z},i}^{2}\ar[rr]^{\iota_{i}}\ar[rd]_{\tilde{\beta}_{i}} 
&& 
K_{\hspace{-0.15em}A,i}/K_{\hspace{-0.15em}A,i}^{2}\ar[ld]^{\tilde{\xi}_{i}}
\\
& 
J_{i}
}
\end{equation}
Then \(\vG'_{i}+\xi_{i}\) is given by pushout of \(K_{\hspace{-0.1em}\BB{Z},i}\ra\BB{Z}[\vG_{i}]\) by the composition of \(K_{\hspace{-0.1em}\BB{Z},i}\ra K_{\hspace{-0.1em}\BB{Z},i}/K_{\hspace{-0.1em}\BB{Z},i}^{2}\) with \(\tilde{\beta}_{i}+\tilde{\xi}_{i}\iota_{i}\). 
See the proof of \cite[\href{https://stacks.math.columbia.edu/tag/08S7}{Tag 08S7}]{SP}.
This implies that \(\rho^{*}(\beta_{2}+\xi_{2})=\rho^{*}\beta_{2}+\rho^{*}\xi_{2}\). Similarly, \(\sigma_{*}(\beta_{1}+\xi_{1})\) is represented by \(\sigma(\tilde{\beta}_{1}+\tilde{\xi}_{1}\iota_{1})=\sigma\tilde{\beta}_{1}+\sigma\tilde{\xi}_{1}\iota_{1}\) which again represents \(\sigma_{*}\beta_{1}+\sigma_{*}\xi_{1}\). The final claim also follows from this representation since pullback of \(\beta_{2}\) along \(\rho\) is represented by pullback of \(\tilde{\beta}_{2}\) along the natural map \(K_{\hspace{-0.1em}\BB{Z},1}/K_{\hspace{-0.1em}\BB{Z},1}^{2}\ra K_{\hspace{-0.1em}\BB{Z},2}/K_{\hspace{-0.1em}\BB{Z},2}^{2}\) while pullback of \(\xi_{2}\) along \(\rho\) is represented by pullback of \(\tilde{\xi}_{2}\) along the natural map \(K_{\hspace{-0.1em}A,1}/K_{\hspace{-0.1em}A,1}^{2}\ra K_{\hspace{-0.1em}A,2}/K_{\hspace{-0.1em}A,2}^{2}\). 
\end{proof}
In the following we will only be interested in graded rings and modules concentrated in degree \(0\) and \(1\).
%%%%%%%%%%%%%% KOROLLAR %%%%%%%%%%%%%%
\begin{cor}\label{cor.ext}%2(4)
Fix an extension of \textup{(}ungraded\textup{)} rings \(\beta_{0}\co 0\ra J_{0}\ra B'\ra B\ra 0\)\textup{,} two \(B\)-modules \(J_{1}\) and \(N\)\textup{,} and a \(B\)\textup{-}linear map \(\phi\co J_{0}\ot N\ra J_{1}\)\textup{.} 
\begin{enumerate}[leftmargin=2.4em, label=\textup{(\roman*)}]
\item There exists an element \(\ob(\beta_{0},N,\phi)\in\xt{2}{B}{N}{J_{1}}\) such that an extension of modules \(\beta_{1}\co 0\ra J_{1}\ra N'\ra N\ra 0\) over \(\beta_{0}\) with induced map \(J_{0}\ot N\ra J_{1}\) equal to \(\phi\) exists if and only if \(\ob(\beta_{0},N,\phi)=0\)\textup{.} The element \(\ob(\beta_{0},N,\phi)\) is natural in \(N\) and \(J_{1}\)\textup{.}
\item If \(\ob(\beta_{0},N,\phi)=0\) the set of isomorphism classes of extensions \(\beta_{1}\) over \(\beta_{0}\) with induced map \(\phi\) is a torsor over \(\xt{1}{B}{N}{J_{1}}\)\textup{.} The torsor action is natural in \(N\) and \(J_{1}\)\textup{.}
\item The set of automorphisms of a given extension \(\beta_{1}\) is canonically isomorphic to \(\hm{}{B}{N}{J_{1}}\)\textup{.}
\end{enumerate}
\end{cor}
\begin{proof}
Consider the graded ring \(\vG=B{\oplus} N\) with \(B\) in degree \(0\) and \(N\) in degree \(1\). The \(B\)-linear map \(\phi\co J_{0}\ot N\ra J_{1}\) gives the \(B\)-module \(J=J_{0}{\oplus} J_{1}\) an \(N\)-action, i.e. \(J\) becomes a graded \(\vG\)-module. Note that an extension of graded rings \(\beta\co 0\ra J\ra \vG'\ra \vG\ra 0\) over \(\beta_{0}\) in degree \(1\) is an extension of \(B'\)-modules \(\beta_{1}\) with the given induced map \(\phi\). We have a translation of the extension problem of modules to a problem of extensions of graded algebras as treated in Propositions \ref{prop.ext} and \ref{prop.omap}. Since there is a trivial extension of $\BB{Z}$-algebras $J_0\ra B{\oplus} J_0\ra B$ defining a zero-element, Proposition \ref{prop.ext} implies that (the equivalence class of) $\beta_{0}$ in a canonical way gives an element of \(\xt{1}{B}{L_{B/\BB{Z}}}{J_{0}}\). We define \(\ob(\beta_{0},N,\phi)\) as the image of \(\beta_{0}\) by the composition
\begin{equation}\label{eq.SPR}
\begin{aligned}
\xymatrix@C-0pt@R-6pt@H-0pt{
\beta_{0}\in\xt{1}{B}{L_{B/\BB{Z}}}{J_{0}}\ar[r]^(0.54){(\id,0)_{*}} &
\gxt{1}{B}{L_{B/\BB{Z}}}{J}\ar[r]^(0.44){\simeq} &
\gxt{1}{\vG}{L_{B/\BB{Z}}\ot^{\BB{L}}\vG}{J}\ar@<-7.3ex>[d]^{\partial(L^{\text{gr}}_{\vG/B/\BB{Z}})}
\\
& 
\hspace{-7.5em}\ob(\beta_{0},N,\phi)\in\xt{2}{B}{N}{J_{1}} & 
\gxt{2}{\vG}{L^{\text{gr}}_{\vG/B}}{J}\hspace{2.2em}\ar[l]_(0.59){\simeq}
}
\end{aligned}
\end{equation}
For the last isomorphism note that \((L^{\text{gr}}_{\vG/B})_{0}=0\) and \((L^{\text{gr}}_{\vG/B})_{1}\) equals the standard \(B\)-free resolution of \(N\); cf. \cite[IV 2.2]{ill:71}, which gives 
\begin{equation}\label{eq.AQext}
\gxt{n}{\vG}{L^{\text{gr}}_{\vG/B}}{J}\cong \xt{n}{B}{N}{J_{1}}
\end{equation}
for all \(n\). The rest then follows from Propositions \ref{prop.ext} and \ref{prop.omap}.
\end{proof}
%%%%%%%%%% REMARK %%%%%%%%%
\begin{rem}\label{rem.ext}
In \cite{ill:71} a main application of the graded cotangent complex is the construction of the obstruction for extensions of modules as in Corollary \ref{cor.ext}, see \cite[IV 3.1.5]{ill:71}. Unfortunately the argument seems to have some flaws. The sequence $(*)$ on page 247 (in our notation)
\begin{equation}
0\ra J_1\lra B{\oplus} N'\lra B{\oplus} N\ra 0
\end{equation}
is claimed to be an extension of graded $B'$-algebras 
which is true if and only if $N'$ is a $B$-module as $B'$-module 
and implies that the multiplication map $\phi\co J_0\ot N\ra J_1$ is zero. But in the deformation case $\phi$ is an isomorphism. In \cite[IV 3.1.2]{ill:71} it is claimed that the set of equivalence classes of graded $B'$-algebra extensions $\gxal{B'}{B{\oplus} N}{J_1}$ is isomorphic to $\xt{1}{B}{N}{J_1}$ which is wrong for the same reason. In \cite[IV 3.1.3]{ill:71} there is also claimed an isomorphism $\hm{}{B}{J_0\ot_{\!B} N}{J_1}\cong \ghm{}{\vG}{J_0\ot_{\!B} \vG}{J_1}$, but the latter is isomorphic to $\ghm{}{B}{J_0}{J_1}=0$. The claim is based on the mistaken adjointness in \cite[IV 1.1.3.1]{ill:71} (e.g. if $A=\BB{Z}[x]$ with $\deg(x)=n>0$, $E=A$ and $F=\BB{Z}$ then $\hm{}{A_0}{E_n}{F}\cong \BB{Z}$ while $\ghm{}{A}{E}{F[-n]}=0$). An extension of this `adjointness' to the derived category \cite[IV 1.2.2.1]{ill:71} is used to connect the exact sequence \cite[IV 3.1.4]{ill:71}, which is used to define the obstruction class in \cite[IV 3.1.5]{ill:71}, to previous results through \cite[IV 3.1.3]{ill:71}. In particular this involves the identification of $\xt{1}{B}{(L_{B/B'}\ot^{\BB{L}}\vG)_1}{J_1}\cong \hm{}{B}{J_0\ot N}{J_1}$ with $\gxt{1}{\vG}{L_{B/B'}\ot^{\BB{L}}\vG}{J_1}$, but the latter is $0$. However, the definition of the obstruction in \cite[IV 3.1.5]{ill:71} gives the same class as \eqref{eq.SPR}:
\end{rem}
%%%%%%%%%%%%%% LEMMA %%%%%%%%%%%%%%
\begin{lem}[cf. {\cite[IV 3.1.5]{ill:71}}]\label{lem.ext}%0(2)
With assumptions as in \textup{Corollary \ref{cor.ext}} put $\vG=B{\oplus}N$\textup{.}
Applying $\hm{}{D(\gmod_{B})}{-}{J_1}$ to the degree $1$-part $(L^{\textup{gr}}_{\vG/B/B'})_1$ of the transitivity triangle\textup{,} see \eqref{eq.dist}\textup{,} gives an exact sequence
\begin{equation*}
0\ra\xt{1}{B}{N}{J_1}\ra\xt{1}{B'}{N}{J_1}\xra{\;u\;}\hm{}{B}{J_0\ot N}{J_1}\xra{\;\partial\;}\xt{2}{B}{N}{J_1}
\end{equation*}
where $\partial(\phi)=\ob(\beta_{0},N,\phi)$ and $u(\beta_1)=\phi$\textup{.} 
\end{lem}
\begin{proof}
First note that $\ob(\beta_{0},N,\phi)$ also can be defined as the image of $(\id,\phi)$ shown in the following pointed commutative diagram of natural maps:
\begin{equation}\label{eq.ill3}
\xymatrix@C-6pt@R-6pt@H0pt{
\id\in\nd{}{B}{J_0}\ar@<-0.14em>[d]^(0.45){\simeq}\ar[r]^(0.45){\simeq} 
&
\xt{1}{B}{L_{B/B'}}{J_0}\ar@<-2.2em>[d]^(0.45){\simeq}\ar[r] 
&
\xt{1}{B}{L_{B/\BB{Z}}}{J_0}\ni\beta_0\hspace{-0.75em}\ar@<-3.12em>[d]^(0.45){\simeq} 
\\
\hspace{-0.6em}(\id,\phi)\in\ghm{}{\vG}{J_0\ot\vG}{J}\ar[r]^(0.52){\simeq} 
&
\hspace{0em}\gxt{1}{\vG}{L_{B/B'}\ot^{\BB{L}}\vG}{J}\ar@<-2.2em>[d]^{\partial(L^{\tn{gr}}_{\vG/B/B'})}\ar[r] 
&
\gxt{1}{\vG}{L_{B/\BB{Z}}\ot^{\BB{L}}\vG}{J}\ar@<-3.12em>[d]^{\partial(L^{\tn{gr}}_{\vG/B/\BB{Z}})}
\\
\ob(\beta_0,N,\phi)\in\xt{2}{B}{N}{J_1}
&
\gxt{2}{\vG}{L^{\tn{gr}}_{\vG/B}}{J}\ar@{=}[r]\ar[l]_(0.4){\simeq} 
&
\gxt{2}{\vG}{L^{\tn{gr}}_{\vG/B}}{J}\hspace{2em}
}
\end{equation}
Just as the $\vG$-module $J$ gives a (generally non-split) extension of graded $\vG$-modules $0\ra J_1\ra J\ra J_0\ra 0$ where $J_i$ is a $B$-module as $\vG$-module, there is an extension of distinguished triangles 
\begin{equation}
0\ra (L^{\text{gr}}_{\vG/B/B'})_1\lra L^{\text{gr}}_{\vG/B/B'}\lra (L^{\text{gr}}_{\vG/B/B'})_0\ra 0
\end{equation}
where the outer terms are $B$-modules as $\vG$-modules. Applying $\hm{}{D(\gmod_{\vG})}{-}{J}$ gives maps of long exact sequences
\begin{equation}
\xt{}{B}{L_{B/B/B'}}{J_0}\lra\gxt{}{\vG}{L^{\text{gr}}_{\vG/B/B'}}{J}\lra \gxt{}{B}{(L^{\text{gr}}_{\vG/B/B'})_1}{J_1}
\end{equation}
In addition to \eqref{eq.AQext} note that 
\begin{equation}
\gxt{1}{B}{(L^{\text{gr}}_{\vG/B/B'})_1}{J_1}\cong\xt{1}{B}{N\ot^{\BB{L}}_{B'}B}{J_1}\cong \xt{1}{B'}{N}{J_1}
\end{equation}
by \cite[IV 2.4.4]{ill:71}. The exact sequence is thus obtained by the low-degree terms of $\xt{}{B}{(L^{\text{gr}}_{\vG/B/B'})_1}{J_1}$. 
The pointed commutative diagram
\begin{equation}
\begin{aligned}
\xymatrix@C-6pt@R-6pt@H0pt{
\hspace{-2em}(\id,\phi)\in\ghm{}{\vG}{J_0\ot\vG}{J}\ar[r]^{\simeq}\ar[d]
& \gxt{1}{\vG}{L_{B/B'}\ot^{\BB{L}}\vG}{J}\ar[r]^(0.54){\partial}\ar[d]
& \gxt{2}{\vG}{L^{\text{gr}}_{\vG/B}}{J}\ar[d]^(0.44){\simeq}
\\
\hspace{0.95em}\phi\in\hm{}{B}{J_0\ot N}{J_1}\ar[r]^{\simeq}
& \gxt{1}{B}{L_{B/B'}\ot^{\BB{L}}N}{J_1}\ar[r]^(0.58){\partial}
& \xt{2}{B}{N}{J_1}
}
\end{aligned}
\end{equation}
and \eqref{eq.ill3} shows that $\partial(\phi)=\ob(\beta_{0},N,\phi)$. The composition
\begin{equation}
\gxt{1}{\vG}{L^{\text{gr}}_{\vG/B'}}{J}\lra \gxt{1}{\vG}{L^{\text{gr}}_{B/B'}\ot^{\BB{L}}\vG}{J}\cong\ghm{}{\vG}{J_0\ot_B\vG}{J}
\end{equation}
maps, by the graded version of \cite[III 2.1.2.2]{ill:71}, the extension $\beta=\beta_0{\oplus}\beta_1$ to $(\id,\phi)$. Restricted to degree one this gives $u(\beta_1)=\phi$.
\end{proof}
The obstructions and the torsor actions of the pair, the algebra, and the module are related by a natural long exact sequence.
%%%%%%%%%% TEOREM %%%%%%%%%%%%%
\begin{thm}\label{thm.lang}%11(14)
Suppose \(A\ra B\) is a ring homomorphism and \(N\) is a \(B\)-module\textup{.} Let \(\vG=B{\oplus} N\) be the graded \(A\)\textup{-}algebra with \(B\) in degree \(0\) and \(N\) in degree \(1\) with inclusion $i\co B\ra\vG$\textup{.} Let \(J=J_{0}{\oplus}J_{1}\) be a graded \(\vG\)\textup{-}module with \(N\)\textup{-}action \(\phi\co J_{0}\ot_BN\ra J_{1}\)\textup{.} Then there is a natural long exact sequence of \(B\)-modules\textup{:}
\begin{align*}
0\ra\, & \hm{}{B}{N}{J_{1}}\xra{\,u\,} \gder{A}{\vG}{J}\xra{\,i^*\,} \der{A}{B}{J_{0}}\xra{\,\partial\,}\dots 
\\
\ra\, &\xt{n}{B}{N}{J_{1}}\xra{\,u\,} \gxt{n}{\vG}{L^{\tn{gr}}_{\vG/A}}{J}\xra{\,i^*\,} \xt{n}{B}{L_{B/A}}{J_{0}}\xra{\,\partial\,} \xt{n+1}{B}{N}{J_{1}}\ra\dots
\end{align*}
Given an extension \(\alpha\co 0\ra I\ra A'\ra A\ra 0\) and an \(A\)-linear map \(\gamma_0\co I\ra J_{0}\). Let \(\gamma\) be \(\gamma_0\)\ composed with the inclusion \(J_{0}\ra J\)\textup{.}
\begin{enumerate}[leftmargin=2.4em, label=\textup{(\roman*)}]
\item The obstruction \(\ob(\alpha,\vG,\gamma)\) maps by \(i^*\) to \(\ob(\alpha,B,\gamma_0)\)\textup{.} 
\item If \(\beta\) is a graded $\gamma$-extension above \(\alpha\) as in \textup{Definition \ref{def.ext}} the torsor actions on \(\beta\) and its degree $0$-part \(\beta_{0}\) are compatible through \(i^*\)\textup{.}
\end{enumerate}
Suppose \(\beta_{0}\co 0\ra J_{0}\ra B'\ra B\ra 0\) is a $\gamma_0$-extension above \(\alpha\)\textup{.}
\begin{enumerate}[resume*]
\item Then \(u(\ob(\beta_{0},N,\phi))=\ob(\alpha,\vG,\gamma)\)\textup{.} If \(\xi\in\xt{1}{B}{L_{B/A}}{J_{0}}\) then 
\begin{equation*}
\ob(\beta_{0}+\xi,N,\phi)=\ob(\beta_{0},N,\phi)+\partial(\xi)\,.
\end{equation*}
\item Suppose \(\beta_{1}\co 0\ra J_{1}\ra N'\ra N\ra 0\) is an extension above \(\beta_{0}\) with induced map \(J_{0}\ot_BN\ra J_{1}\) equal to \(\phi\)\textup{.} Then the torsor actions on \(\beta_{1}\) and \(\beta=\beta_0{\oplus}\beta_1\) are compatible through \(u\)\textup{.}
\end{enumerate}
\end{thm}
%%%%%%%%%%%%%
\begin{proof}
The long exact sequence follows from $\hm{}{D(\gmod_{\vG})}{L^{\tn{gr}}_{\vG/B/A}}{J}$\textup{,} see \textup{\eqref{eq.dist},}  
since \(\gxt{n}{\vG}{L^{\text{gr}}_{\vG/B}}{J}\cong\xt{n}{B}{N}{J_{1}}\) by \textup{\eqref{eq.AQext}} and \(\gxt{n}{\vG}{L^{\tn{gr}}_{B/A}\ot^{\BB{L}}\vG}{J}\linebreak[1]\cong \xt{n}{B}{L_{B/A}}{J_{0}}\). Then (i), (ii) and (iv) are special cases of Propositions \ref{prop.ext} and \ref{prop.omap}.

(iii) The following commutative diagram of canonical maps
\begin{equation}\label{eq.mapobs}
\begin{aligned}
\xymatrix@C-24pt@R-6pt@H-0pt{
&\beta_{0}\in\xt{1}{B}{L_{B/\BB{Z}}}{J_{0}}\hspace{1.5em}
\ar@<-2.08em>[d]^(0.45){(\id,0)_{*}}
&
\xt{1}{A}{L_{A/\BB{Z}}}{I}\ni\alpha\hspace{3.7em}
\ar@<-4.94em>[d]^(0.45){\gamma_{*}}  
\\
&\gxt{1}{B}{L_{B/\BB{Z}}}{J}\hspace{0.05em}\ar@<-2.08em>[d]^{\partial(L^{\tn{gr}}_{\vG/B/\BB{Z}})}\ar[r]
&
\gxt{1}{A}{L_{A/\BB{Z}}}{J}\hspace{5.85em}\ar@<-4.94em>[d]^{\partial(L^{\tn{gr}}_{\vG/A/\BB{Z}})} 
\\
\hspace{-0.3em}\ob(\beta_{0},N,\phi)\in\xt{2}{B}{N}{J_{1}}\ar[r]^(0.58){\simeq}
&
\gxt{2}{\vG}{L^{\tn{gr}}_{\vG/B}}{J}\ar[r]
&
\gxt{2}{\vG}{L^{\tn{gr}}_{\vG/A}}{J}\ni\ob(\alpha,\vG,\gamma)
}
\end{aligned}
\end{equation}
is pointed since \((\beta_{0},0)\) maps to \(((\gamma_0)_{*}\alpha,0)\) by \cite[\href{https://stacks.math.columbia.edu/tag/08S8}{Tag 08S8}]{SP}, i.e. \(u(\ob(\beta_{0},N,\phi))=\ob(\alpha,\vG,\gamma)\).

The commutative diagram of ring maps 
\begin{equation}\label{eq.mapobs2}
\begin{aligned}
\xymatrix@C-3pt@R-8pt@H-0pt{
\BB{Z}\ar[r]\ar[d] & B\ar[r]\ar@{=}[d] & \vG\ar@{=}[d]
\\
A\ar[r] & B\ar[r] & \vG
}
\end{aligned}
\end{equation} 
induces the following commutative diagram of \(B\)-modules:
\begin{equation}\label{eq.theta}
\begin{aligned}
\xymatrix@C-3pt@R-6pt@H-0pt{
\xt{1}{B}{L_{B/A}}{J_{0}}\ar@<-2.48em>[d]^{\partial(L^{\tn{gr}}_{\vG/B/A})}\ar[r]^(0.48){\theta}
&
\xt{1}{B}{L_{B/\BB{Z}}}{J_{0}}\hspace{-0.74em}\ar@<-2.06em>[d]^{\partial(L^{\tn{gr}}_{\vG/B/\BB{Z}})}
\\
\hspace{-0.8em}\gxt{2}{\vG}{L^{\tn{gr}}_{\vG/B}}{J}\ar@{=}[r]
&
\gxt{2}{\vG}{L^{\tn{gr}}_{\vG/B}}{J}
}
\end{aligned}
\end{equation}
Suppose \(\xi\in \xt{1}{B}{L_{B/A}}{J_{0}}\). Let \(\beta_{\xi}\co 0\ra J_{0}\ra B'_{\xi}\ra B\ra 0\) denote the $\gamma_0$-extension obtained from the torsor action of \(\xi\) on \(\beta_{0}\). Then \(\beta_{\xi}-\beta_{0}=\theta(\xi)\) by Proposition \ref{prop.omap}. The commutativity of \eqref{eq.theta} implies the second statement of (iii).
\end{proof}
%%%%%%%%%%%%%%%%%%%%%%%%%%%%%%%%%%%%
%%%%%%%%%%% SECTION %%%%%%%%%%%%%%%%%%
%%%%%%%%%%%%%%%%%%%%%%%%%%%%%%%%%%%%
\section{The Kodaira-Spencer classes and maps}\label{sec.KS}
The long exact sequence in Theorem \ref{thm.lang} relates the Kodaira-Spencer class and map of a pair (algebra, module) to the corresponding notions for the algebra and for the module. 
The following is a graded version of \cite[II 2.1.5.7]{ill:71}.
%%%%%%%%%% DEFN %%%%%%%%%%%
\begin{defn}\label{defn.KS}%0
Let \(\vL\ra A\) and \(A\ra \vG\) be graded ring homomorphisms with \(\vL\) and \(A\) concentrated in degree \(0\). Put $B=\vG_0$, the degree $0$ part. The map \(L^{\text{gr}}_{\vG/A}\ra L_{A/\!\vL}\ot_{A}\vG[1]\) in $L^{\tn{gr}}_{\vG/A/\vL}$, see \eqref{eq.dist}, is called the \emph{Kodaira-Spencer class} of \(\vL\ra A\ra\vG\).
\end{defn}
Composing the Kodaira-Spencer class map with the natural augmentation map 
$
L_{A/\!vL}\ot_{A}\vG[1]\lra\Omega_{A/\vL}\ot_{A}\vG[1]
$
induces the \emph{cohomological Kodaira-Spencer class} \(\kappa(\vG/A/\vL)\in
\gxt{1}{\vG}{L^{\tn{gr}}_{\vG/A}}{\Omega_{A/\!\vL}\ot_{A}\vG}\), also given as follows. Let \({P^1}=P^1_{A/\!\vL}\) denote \(A\ot_{\vL}A/I^{2}\) where \(I=\ker\{\mu\co A\ot_{\vL}A\ra A\}\). The ring maps \(j_{1}\) and \(j_{2}\) from \(A\) to \({P^1}\) defined by \(j_{1}\co s\mapsto s\ot 1\) and \(j_{2}\co s\mapsto 1\ot s\) gives the universal derivation \(d_{A/\!\vL}\co A\ra \Omega_{A/\!\vL}=I/I^{2}\) induced by \(j_{2}-j_{1}\). Let \(\gamma_{A/\!\vL}\co \Omega_{A/\!\vL}\ra P^1_{A/\!\vL}\) denote the inclusion. The \emph{principal parts of \(\vG\)} is \(P^1_{A/\!\vL}(\vG)=P^1\ot_{A}\vG\) which is an \(A\)-algebra via \(j_{1}\ot 1_{\vG}\co A\ra {P^1}\ot_{A}\vG\). The \(A\)-algebra extension
%%%
\begin{equation}\label{eq.pp}
\kappa(\vG/A/\vL)\co \quad0\ra \Omega_{A/\!\vL}\ot_{A}\vG\xra{\gamma_{A/\!\vL}\ot{\id_{\vG}}} P^1_{A/\!\vL}\ot_{A}\vG\xra{\bar{\mu}\ot{\id_{\vG}}} \vG \ra 0
\end{equation}
represents the Kodaira-Spencer class; see \cite[III 1.2.6]{ill:71}. Since \({P^1}\ot_{A}\vG\) has a natural \({P^1}\)-algebra structure, \eqref{eq.pp} is also a (graded) algebra ${\id}\ot 1\co\Omega_{A/\!\vL}\ra \Omega_{A/\!\vL}\ot_{A}\vG$-extension of \(\vG\) above \(\Omega_{A/\!\vL}\ra {P^1}\ra A\); see Definition \ref{def.ext}. Let \(\vG\ot_{A}{P^1}\) denote the \(A\)-algebra defined via the \(j_{1}\) tensor product. The extension \(\vG\ot_{A}{P^1}\ra \vG\) is trivial (split by \({\id_{\vG}}\ot 1_{{P^1}} \)) and the `difference' of the \({P^1}\)-algebras \({P^1}\ot_{A}\vG\) and \(\vG\ot_{A}{P^1}\) is an element in \(\gxt{1}{\vG}{L^{\tn{gr}}_{\vG/A}}{\Omega_{A/\!\vL}\ot_{A}\vG}\) by Proposition \ref{prop.ext} (ii), which equals \(\kappa(\vG/A/\vL)\); see \cite[III 2.1.5]{ill:71}. 
Moreover, the difference \(1_{{P^1}}\ot_{j_{2}}\!\id_{\vG}-\id_{\vG}\!\ot _{j_{1}}\!1_{{P^1}}\) induces \(d_{A/\!\vL}\ot 1_{B}\) 
which is mapped to \(\kappa(\vG/A/\vL)\) by \(\partial\) in the Jacobi-Zariski sequence of \(\vL\ra A\ra \vG\) obtained from \eqref{eq.dist}:
%%%
\begin{equation}\label{eq.let}
\gder{\vL}{\vG}{\Omega_{A/\!\vL}\ot_{A}\vG}\xra{\,i^*\,}\der{\vL}{A}{\Omega_{A/\!\vL}\ot_{A}B}\xra{\;\partial\;}\gxt{1}{\vG}{L^{\tn{gr}}_{\vG/A}}{\Omega_{A/\!\vL}\ot_{A}\vG};
\end{equation} 
see \cite[III 1.2.6.5 and 1.2.7]{ill:71}. Hence \(\kappa(\vG/A/\vL)=0\) if and only if there is a (graded) derivation \(\mc{D}\) in \(\gder{\vL}{\vG}{\Omega_{A/\!\vL}\ot_{A}\vG}\) which maps to \(d_{A/\!\vL}\ot 1_{B}\). Then 
\begin{equation}
\sigma:=j_{2}\ot{\id_{\vG}}-(\gamma_{A/\!\vL}\ot{\id_{\vG}})\mc{D}\co \vG\ra {P^1}\ot_{A}\vG
\end{equation}
is an \(A\)-algebra map (since \(\sigma_{\vert A}=j_{1}\ot 1_{B}\)) which splits \({P^1}\ot_{A}\vG\ra\vG\). Conversely, such a splitting \(\sigma\) induces via \(j_{2}\ot{\id_{\vG}}-\,\sigma\) a derivation mapping to \(d_{A/\!\vL}\ot 1_{B}\).

The \emph{Kodaira-Spencer map of \(\vL\ra A\ra\vG\)} 
\begin{equation}\label{eq.KSmap}
g^{\vG}\!\co \Der_{\vL}(A)\lra\gxt{1}{\vG}{L^{\tn{gr}}_{\vG/A}}{\vG}
\end{equation}
is defined by \(D\mapsto (f^D\!\ot{\id}_{\vG})_*\kappa(\vG/A/\vL)\) where \(f^{D}\!\co \Omega_{A/\!\vL}\ra A\) corresponds to \(D\). Pushout of \eqref{eq.pp} by \(f^{D}\!\ot{\id_{\vG}}\) equals the corresponding algebra $(A\ra\vG)$-extension of \(\vG\) above \(A[\vare]\ra A\) given by \(g^{\vG}(D)\); see Proposition \ref{prop.omap} (ii). 
%%%%%
Note that the connecting map $\partial\co \der{\vL}{A}{J_0}\ra\gxt{1}{\vG}{L^{\tn{gr}}_{\vG/A}}{J}$ for any $\vG$-module $J=J_0{\oplus} J_1$ by naturality has the same description: $\partial(D)=(f^D\!\ot{\id}
_{\vG})_*\kappa(\vG/A/\vL)$.

In the special case \(\vG=B{\oplus}{N}\) where \({N}\) is a \(B\)-module,  
\(\gxt{1}{\vG}{L^{\tn{gr}}_{\vG/B}}{\Omega_{B/\!\vL}\ot_{B}\vG}\cong \xt{1}{B}{{N}}{\Omega_{B\!/\!\vL}\ot_{B}{N}}\) and the Kodaira-Spencer class \(\kappa(\vG/B/\vL)\) corresponds to the (cohomological) \emph{Atiyah class} \(\at_{B\!/\!\vL}({N})\in\xt{1}{B}{{N}}{\Omega_{B\!/\!\vL}\ot_{B}{N}}\); cf.\ \cite[IV 2.3.6-7]{ill:71}. 
In particular, $\at_{B\!/\!\vL}({N})=\partial(\vG/B/\vL)(d_{B/\!\vL})$. 
The class is represented by the short exact sequence of \(B\)-modules
\begin{equation}\label{eq.at}
\at_{B\!/\!\vL}({N})\co \quad 0\ra\Omega_{B\!/\!\vL}\ot_{B}{N}\lra P^1_{B\!/\!\vL}\ot_{B}{N}\lra{N}\ra 0\,.
\end{equation}
%%%%%%%%%%%%%%% PROPOSITION %%%%%%%%%%%%
\begin{prop}\label{prop.at}%1
Consider the sequence of \(A\ra B\ra \vG\) in \textup{Theorem \ref{thm.lang}} with $\vG=B{\oplus} N$\textup{,} $J_0=\Omega_{A/\!\vL}\ot_{A}B$ and $J_1=\Omega_{A/\!\vL}\ot_{A}N$\textup{:}
\begin{equation*}
\dots \xra{\,\partial\,}\xt{1}{B}{{N}}{\Omega_{A/\!\vL}\ot_{A}{N}}\xra{\,u\,}\gxt{1}{\vG}{L^{\tn{gr}}_{\vG/A}}{\Omega_{A/\!\vL}\ot_{A}\vG}\xra{i^{*}}
\xt{1}{B}{L_{B/A}}{\Omega_{A/\!\vL}\ot_{A}B}
\end{equation*}
\begin{enumerate}[leftmargin=2.4em, label=\textup{(\roman*)}]
\item The map \(i^{*}\) takes the Kodaira-Spencer class \(\kappa(\vG/A/\vL)\) to \(\kappa(B/A/\vL)\)\textup{.}
\item Assume \(\kappa(B/A/\vL)=0\) and choose an \(A\)-algebra splitting \(\sigma\co  B\ra P^1_{A/\!\vL}\ot_{A}B\)\textup{.} There is an element 
\begin{equation*}
\kappa(\sigma,{N})\in\xt{1}{B}{{N}}{\Omega_{A/\!\vL}\ot_{A}{N}},
\end{equation*}
appropriately natural in \((B/A/\vL,\sigma,{N})\)\textup{,} which \(u\) maps to \(\kappa(\vG/A/\vL)\)\textup{.} Moreover\textup{,} \(\kappa(\sigma,{N})\) is represented by the degree \(1\) part of \eqref{eq.pp} 
\begin{equation*}
\kappa(\sigma,{N})\co\quad0\ra \Omega_{A/\!\vL}\ot_{A}{N}\lra P^1_{A/\!\vL}\ot_{A}{N}\lra{N}\ra 0
\end{equation*}
considered as a short exact sequence of \(B\)-modules via \(\sigma\)\textup{.}
\item Let \(D(\sigma)\) in \(\Der_{\vL}(B,\Omega_{A/\!\vL}\ot_{A}B)\) be the derivation induced by \(\sigma\) and \(f^{D(\sigma)}\) in \(\hm{}{B}{\Omega_{B/\!\vL}}{\Omega_{A/\!\vL}\ot_{A}B}\) the corresponding homomorphism\textup{.} Then 
\begin{equation*}
\kappa(\sigma,{N})=(f^{D(\sigma)}\ot{\id_{N}})_{*}\at_{B/\!\vL}({N}).
\end{equation*}
For each \(Y\in\Der_{\vL}(A)\) let \(X_{\sigma}(Y)\) denote \((f^{Y}\!\ot{\id_{B}})_{*}D(\sigma)\in\Der_{\vL}(B)\)\textup{.} Then
\begin{equation*}
(f^{Y}\!\ot{\id_{N}})_{*}\kappa(\sigma,{N})=(f^{X_{\sigma}(Y)}\ot{\id_{N}})_{*}\at_{B\!/\!\vL}({N})\:\:\text{in}\:\: \xt{1}{B}{{N}}{{N}}.
\end{equation*}
\end{enumerate}
\end{prop}
%%%%%%%%%
\begin{proof}
(i) The degree zero part of \eqref{eq.pp} gives the image \(i^{*}\kappa(\vG/A/\vL)\) represented by the algebra extension
\begin{equation}
\kappa(B/A/\vL)\co \quad 0\ra\Omega_{A/\!\vL}\ot_{A} B\lra P^1_{A/\!\vL}\ot_{A}B\lra B\ra 0
\end{equation}

(ii) The degree one part of \eqref{eq.pp} is \emph{a priori} a short exact sequence of \(P^1_{A/\!\vL}\ot_{A}B\)-modules.
The splitting \(\sigma\) makes it to a short exact sequence of \(B\)-modules which defines \(\kappa(\sigma,{N})\). The naturality follows from the naturality of \(\kappa(\vG/A/\vL)\).  
 
(iii) There is a \(B\)-algebra homomorphism \(h^{\sigma}\co P^1_{B/\!\vL}\ra P^1_{A/\!\vL}\ot_{A}B\) defined by \(h^{\sigma}(t_{1}\ot t_{2})=\sigma(t_{1})\cdot1_{{P^1}}\ot t_{2}\) which we claim makes the following diagram commutative:
\begin{equation}
\begin{aligned}
\xymatrix@C+18pt@R-6pt@H-30pt{
\at_{B\!/\!\vL}({N})\co\!\!\!\!\!\!\!\!\!\!\!\!\!\!\!\!\!\!\!\!  & 0\lra\Omega_{B\!/\!\vL}\ot_{B}{N}\ar[r]^(0.55){\gamma_{B/\!\vL}\ot{\id_{N}}}\ar@<3.6ex>[d]_(0.46){f^{D(\sigma)}\ot{\id_{N}}} & P^1_{B\!/\!\vL}\ot_{B}{N}\ar[r]\ar@<0.5ex>[d]_(0.46){h^{\sigma}\ot{\id_{N}}} &{N}\lra 0\ar@<-3.1ex>@{=}[d] \\
\kappa(\sigma,{N})\co\!\!\!\!\!\!\!\!\!\!\!\!\!\!\!\!\!\!\!\!\!\!\!\!  & 0\lra \Omega_{A/\!\vL}\ot_{A}{N}\ar[r]_(0.55){\gamma_{A/\!\vL}\ot{\id_{N}}} & P^1_{A/\!\vL}\ot_{A}{N}\ar[r] &{N}\lra 0
}
\end{aligned}
\end{equation}
For \(t\) in \(B\) we have by the definitions 
\begin{equation}
h^{\sigma}\gamma_{B/\!\vL}d_{B/\!\vL}(t)=h^{\sigma}(1_{B}\ot t-t\ot 1_{B})=1_{{P^1}}\ot t-\sigma(t)=(\gamma_{A/\!\vL}\ot 1_{B})D(\sigma)(t)
\end{equation}
and the claim follows. Then the pushout of \(\at_{B/\!\vL}({N})\)  by \(f^{D(\sigma)}\ot{\id_{N}}\) gives \(\kappa(\sigma,{N})\). Since \((f^{Y}\ot{\id_{N}})\circ(f^{D(\sigma)}\ot{\id_{N}})\) equals \(f^{X_{\sigma}(Y)}\ot{\id_{N}}\) the second part of (iii) follows from the first. 
\end{proof}
We name \(\kappa(\sigma,{N})=\kappa(B/A/\vL,\sigma,{N})\) the \emph{Kodaira-Spencer class of \((B/A/\vL,\sigma,{N})\)}.
Define the \emph{Kodaira-Spencer map of \((B/A/\vL,\sigma,{N})\)}
\begin{equation}\label{eq.modKS}
g^{(\sigma,\,{N})}\co \Der_{\vL}(A)\lra\xt{1}{B}{{N}}{{N}}
\end{equation}
by \(g^{(\sigma,\,{N})}(D):=(f^{D}\ot{\id})_{*}\kappa(\sigma,{N})\).

In the case \(B=A\ot_{\vL}B_0\) there is a canonical \(A\)-algebra splitting 
\begin{equation}
A\ot_{\vL}B_0\ra P^1_{A/\!\vL}\ot_{A}A\ot_{\vL}B_0\cong P^1_{A/\!\vL}\ot_{\vL}B_0\quad\tn{given by}\quad s\ot t\mapsto j_{1}(s)\ot t\,,
\end{equation} 
\(\kappa(B/A/\vL)=0\) and we get a canonical Kodaira-Spencer class $\kappa(j_1\ot{\id},{N})$. In particular, if $B=A$ (and $B_0=\vL$) then $\kappa(j_1\ot{\id},{N})=\at_{A/\vL}(N)$. 
%%%%%%%%%% REMARK %%%%%%%%%%%
\begin{rem}
By \eqref{eq.let} with $\vG=B$, \(\kappa(B/A/\vL)=0\) if and only if the canonical map \(\Omega_{A/\!\vL}\ot_{A}B\ra\Omega_{B/\!\vL}\) is split injective. By \cite[{\bf{0}} 20.5.7]{EGAIV01} the latter is equivalent to \(B\) being a \emph{formally smooth \(A\)-algebra relative to \(\vL\)} \textup{(}with discrete topology\textup{)}; cf.\ \cite[{\bf{0}} 19.9.1]{EGAIV01}, or, equivalently: All \(\vL\)-split \(A\)-algebra extensions \(0\ra M\ra B'\ra B\ra0\) with \(M^{2}=0\) are trivial; cf.\ \cite[{\bf{0}} 19.9.8.1]{EGAIV01}. 
\end{rem}
%%%%%%%%%% EXAMPLE %%%%%%%%%%%
\begin{rem}
The long exact sequence in Theorem \ref{thm.lang} for \(\vL\ra B\ra\vG\) with \(\vG=B{\oplus}{N}\) and \(J=J_{0}{\oplus}J_{1}\) 
implies the following characterisation. An element \(\mc{D}\in\gder{\vL}{\vG}{J}\) is given by its degree \(0\) restriction \(D:=i^{*}(\mc{D})\in\der{\vL}{B}{J_{0}}\) and its degree \(1\) restriction \(\nabla_{\!D}:=\mc{D}_{\vert{N}}\in\hm{}{\vL}{{N}}{J_{1}}\) which should satisfy the following Leibniz rule: For all \(t\) in \(B\) and \(n\) in \({N}\) 
\begin{equation}\label{eq.Leib}
\nabla_{\!D}(tn)=t\nabla_{\!D}(n)+D(t)n\,.
\end{equation}
Since \(\kappa(\vG/A/\vL)=\partial(d_{A/\!\vL}\ot 1_{B})\) in \eqref{eq.let}, \(\kappa(\vG/A/\vL)=0\) if and only if there exists a \(D\in\Der_{\vL}(B,\Omega_{A/\!\vL}\ot B)\) which restricts to \(d_{A/\!\vL}\ot 1_{B}\) and a \(\nabla_{\!D} \in\hm{}{\vL}{{N}}{\Omega_{A/\!\vL}\ot {N}}\) satisfying \eqref{eq.Leib}. As a well known special case (\(A=B\)) we get \(\at_{B\!/\!\vL}({N})=0\) if and only if there exists a \(\nabla\in \hm{}{\vL}{{N}}{\Omega_{B\!/\!\vL}\ot {N}}\) satisfying \eqref{eq.Leib} with \(D=d_{B/\!\vL}\in\Der_{\vL}(B,\Omega_{B\!/\!\vL})\) (i.e.\ \(\nabla\) is a `classical' connection), or equivalently, there is a (graded) derivation \(\mc{D}\in \gder{\vL}{\vG}{\Omega_{B\!/\!\vL}\ot_{B}\vG}\) restricting to \(d_{B/\!\vL}\).

The existence of various splittings like connections (covariant derivations) $\nabla_{\!(-)}$ of the restriction $i^*$ has been studied in particular cases of singularities, e.g. E. Eriksen and T.S. Gustavsen \cite{eri/gus:06,eri/gus:08},  Gustavsen and Ile \cite{gus/ile:08,gus/ile:10} and (by \cite{gus/ile:08}) also C.P.M. Kahn \cite{kah:88} and K. Behnke \cite{beh:89}.  See also L. K{\"a}llstr{\"o}m \cite{kal:05} for related general results.
\end{rem}
%%%%%%%%%%%%%%%%%%%%%%%%%%%%%%%%%%%%
%%%%%%%%%%% SECTION %%%%%%%%%%%%%%%%%%
%%%%%%%%%%%%%%%%%%%%%%%%%%%%%%%%%%%%
\section{Conditions for versality}\label{sec.versal1}
In order to use (consequences of) Artin's Approximation Theorem \cite{art:69} as extended by D.\ Popescu \cite{pop:86,pop:90} we fix an excellent ring \(\vL\) (cf. \cite[7.8.2]{EGAIV2}), a field $k$ and a finite ring homomorphism \(\vL\ra k\) with kernel denoted \(\fr{m}_{\vL}\). Put \(k_{0}=\vL/\fr{m}_{\vL}\). Define $\He$ to be the category of surjective maps of $\vL$-algebras $S\ra k$ where $S$ is a noetherian, henselian, local ring. A morphism is a map of $\vL$-algebras $S_1\ra S_2$ which is local and commutes with the given maps to $k$. A map \(h\co S\ra \mc{A}\) of local henselian rings is \emph{algebraic} if $h$ factors as $S\ra \mc{A}^{\tn{ft}}\ra \mc{A}$ where the first map is of finite type and the second is the henselisation in a maximal ideal. Let \(\Ar\) denote the full subcategory of artin rings in \(\He\). 

Let $\cat{F}\ra\He$ be a cofibred category\footnote{A fibred category mimics pull-backs. We  work with rings instead of (affine) schemes. A cofibred category is a functor \(p:\cat{F}\ra\cat{C}\) such that the functor of opposite categories \(p^{\text{op}}:\cat{F}^{\text{op}}\ra\cat{C}^{\text{op}}\) is a fibred category as defined in A.\ Vistoli's \cite{vis:05}. See also \cite[\href{https://stacks.math.columbia.edu/tag/06GA}{Tag 06GA}]{SP}.}. For all maps $f\co R\ra S$ in $\He$ and any object $b$ in $\cat{F}(R)$ we fix a base change $f_*b$ in $\cat{F}(S)$, also denoted $b_S$. For an object $a$ in $\cat{F}(S)$, let $\cat{F}_a\ra \He/S$ denote the induced cofibred category of maps $a'\ra a$ in $\cat{F}$. Suppose $R\ra S$ and $S'\ra S$ are maps in $\He$ with the latter being an infinitesimal extension (i.e. a surjection with nilpotent kernel). Then $S'{\times}_SR$ is in $\He$; \cite[Chap. I, \S2.2]{ray:70}. Assume that the composition $R\ra S_{\tn{red}}$ is surjective. We will call such data a \emph{test situation}. There is a natural \emph{test map}
\begin{equation}\label{eq.test}
\cat{F}_a(S'\times_SR)\lra \cat{F}_a(S')\times\cat{F}_a(R)\,.
\end{equation}
Let $F=\wbar{\,\cat{F}}$ and $F_a=\wbar{\,\cat{F}}_a$ denote the associated functors from $\He$ to $\Sets$.
%%%%%%%%% DEFINITION %%%%%%%%%%%%%
\begin{defn}\label{def.S1}%3 
For a cofibred category $\cat{F}\ra\He$ there are conditions:
\begin{itemize}
\item[(S1')] The test map \eqref{eq.test} is an equivalence for all test situations.
\item[(S1a)] The test map for $F$ is surjective for all test situations.
\item[(S1b)] If in addition $S=S_{\tn{red}}$, $b$ is an object in $\cat{F}(R)$ mapping to $a$ and $I$ is a finite $S$-module, then the natural map $F_b(R{\oplus}I)\ra F_a(S{\oplus}I)$ is bijective.
\item[($\liminj$)] $\cat{F}\ra\He$ is \emph{locally of finite presentation} if for all filtering direct limits $\liminj S_i$ in $\He$ the natural map $\liminj\cat{F}(S_i)\ra \cat{F}(\liminj S_i)$ is an equivalence of categories.
\item[(iso)] Assume $S$ is algebraic and $I\sbeq S$ is an ideal. Put $S_n=S/I^{n+1}$ and $S_I\hspace{-0.18em}\hat{}=\limproj S_n$. If $a$, $b$ are objects in $\cat{F}(S_I\hspace{-0.18em}\hat{}\,)$ and $(\theta_n\co a_n\cong b_n)$ is a system of compatible isomorphisms of base changes to the $S_n$, then there is an isomorphism $a\cong b$ compatible with $\theta_0$.
\end{itemize}
\end{defn}
If (S1') then (S1) ($=$ (S1a) and (S1b)). If (S1) then the map in (S1b) is $R$-linear. In that case there is a further condition:
%%%%%%%%% DEFINITION %%%%%%%%%%%%%
\begin{defn}\label{def.S2}%2
\quad
\begin{itemize}
\item[(S2)] $F_a(S{\oplus}I)$ is a finite $S$-module for any $S=S_{\tn{red}}$ in $\He$, $a$ in $\cat{F}(S)$ and finite $S$-module $I$.
\end{itemize}
\end{defn} 
The condition \emph{$\cat{F}(k)$ is equivalent to a one-object, one-morphism category} will in the following  be assumed for all cofibred categories $\cat{F}$ over $\He$ or $\Ar$.  

Assume that $\cat{F}$ and $\cat{G}$ are cofibred categories over $\He$ which are locally of finite presentation. A map \(\phi\co \cat{F}\ra \cat{G}\) is \emph{smooth} (formally smooth) if, for all surjections \(f\co S'\ra S\) in \(\He\) (respectively in \(\Ar\)), the natural map  
\begin{equation}\label{eq.vers}
(f_*,\phi(S'))\co F(S')\ra F(S)\times_{G(S)}G(S')
\end{equation}
is surjective. Put $h_R=\hm{}{\He}{R}{-}$. Let $v$ be an object in \(\cat{F}(R)\) and let $c_v\co h_R\ra F$ denote the corresponding Yoneda map. 
If \(R\) is algebraic as \(\vL\)-algebra and $c_v$ is smooth (an isomorphism) then $v$ is \emph{versal} (respectively universal). Moreover, $v$ (or a formal object $v=(v_{n})$ in $\limproj \cat{F}(R/\fr{m}_{R}^{n+1})$) is \emph{formally versal} if $c_v$ restricted to $\Ar$ is formally smooth.

For an object $a$ in $\cat{F}(S)$ define $\cat{Ex}_{\cat{F}}(a,I)$ as the category of infinitesimal extensions of $(S,a)$ by the finite $S$-module $I$. An object $(R,b)\ra (S,a)$ in $\cat{Ex}_{\cat{F}}(a)$ is an extension $I\ra R\xra{\,r\,} S$ with $r$ in $\He$ (in particular $I^2=0$), an object $b$ in $\cat{F}(R)$ and a map $b\ra a$ in $\cat{F}$ inducing an isomorphism $b_S\cong a$. A morphism is a map of such pairs above $(S,a)$ which restricts to $\id_I$. Let $\mr{Ex}_{\cat{F}}(a,I)$ denote the set of isomorphism classes in $\cat{Ex}_{\cat{F}}(a,I)$. Note that $\mr{Ex}_{\cat{F}}(a,I)$ always contains (the isomorphism class of) the trivial extension, denoted $0$. The following is Flenner's \cite[3.2 and Bemerkung p. 457]{fle:81} in our context.
%%%%%%%%%%%%%%%%% PROPOSITION %%%%%%%%%%%%%%%%%%
\begin{prop}\label{prop.ex}%2
Given a cofibred category $\cat{F}\ra\He$ satisfying \textup{(S1), (S2)} and \textup{($\liminj$)} in \textup{Definitions \ref{def.S1}} and \textup{\ref{def.S2}.} Assume $a$ is an object in $\cat{F}(S)$\textup{.} The following are equivalent conditions\textup{.}
\begin{enumerate}[leftmargin=2.4em, label=\textup{(\roman*)}]
\item $a$ is formally versal\textup{.}
\item $\mr{Ex}_{\cat{F}}(a,k)=\{0\}$\textup{.}
\item $\mr{Ex}_{\cat{F}}(a,I)=\{0\}$ for all finite $S$-modules $I$\textup{.}
\end{enumerate}
\end{prop}
Proposition \ref{prop.ex} implies the following simplification of \cite[3.3]{art:74}; see the proof of von Essen's \cite[1.3]{ess:90}.
%%%%%%%%%%%%%%%%% PROPOSITION %%%%%%%%%%%%%%%%%%
\begin{prop}\label{prop.fvers}%2(3)
Assume the cofibred category $\cat{F}\ra\He$ satisfies \textup{(S1), (S2)} for algebraic objects\textup{,} as well as \textup{($\liminj$)} and \textup{(iso)} in \textup{Definitions \ref{def.S1}} and \textup{\ref{def.S2}.} Suppose $v$ is an object in $\cat{F}(R)$ with $R$ algebraic as $\vL$-algebra\textup{.} If $v$ is formally versal\textup{,} then $v$ is versal\textup{.}
\end{prop}
%%%%%%%%% DEFINITION %%%%%%%%%%%%%
\begin{defn}\label{def.obtheory}%2
A cofibred category $\cat{F}\ra\Ar$ has an \emph{obstruction theory} if there is a \(k\)-linear functor \(\cH_{\cat{F}}^{2}\co\Mod_{k}\ra\Mod_{k}\) and for each small surjection \(\pi\co R\ra S\) in \(\Ar\) (i.e. with kernel \(I\) such that \(\fr{m}_{R}{\cdot} I=0\)) and each object $a$ in \(\cat{F}(S)\) there is an element \(\mr{o}(\pi,a)\in \cH_{\cat{F}}^{2}(I)\) which is zero if and only if there exists a \(b\in \cat{F}(R)\) with \(\pi_*b\cong a\). The obstruction should be functorial with respect to such lifting situations.
Cf. \cite[2.6]{art:74}. 
\end{defn}
Assume $k_0\ra k$ is a separable field extension and $\cat{F}\ra \Ar$ is a cofibred category with associated functor $F\co \Ar\ra \Sets$. Then a formally versal object in $\cat{F}$ with base $R^{\cat{F}}$ in $\He$ is \emph{minimal} if the induced map $\hm{}{{\He}}{R^{\cat{F}}}{k[\vare]}\ra F(k[\vare])=:t_{\cat{F}\hspace{-0.1em}/\hspace{-0.1em}\vL}$ is a bijection (cf. \cite[\href{https://stacks.math.columbia.edu/tag/06IL}{Tag 06IL}]{SP}). The following result is essentially \cite[6.5]{ile:21}.
%%%%%%%%%%%%%%%% LEMMA %%%%%%%%%%%%%%%%
\begin{lem}\label{lem.res}%4(5)
Suppose $k_0\ra k$ is a separable field extension and \(\phi\co \cat{F}\ra \cat{G}\) is a map of cofibred categories over $\Ar$ which have minimal formally versal formal objects with base rings \(R^{\cat{F}}\) and \(R^{\cat{G}}\) in $\He$ which both are algebraic over $\vL$ \textup{(}or complete\textup{).} Put $V=\ker\{t_{\cat{G}\hspace{-0.04em}/\hspace{-0.1em}\vL}^{*}\ra t_{\cat{F}\hspace{-0.1em}/\hspace{-0.1em}\vL}^{*}\}\)\textup{.} Assume\textup{:}
\begin{enumerate}[leftmargin=2.4em, label=\textup{(\roman*)}]
\item The map \(t_{\cat{F}\hspace{-0.1em}/\hspace{-0.1em}\vL}\ra t_{\cat{G}\hspace{-0.04em}/\hspace{-0.1em}\vL}\) is injective\textup{.}
\item There are obstruction theories for \(\cat{F}\) and \(\cat{G}\) such that \(\mr{o}_{\cat{G}}(\pi,\phi_{S}(a))=0\) implies \(\mr{o}_{\cat{F}}(\pi,a)=0\) for any small surjection \(\pi\co R\ra S\) in $\Ar$ and object $a$ in \(\cat{F}(S)\)\textup{.}
\end{enumerate}
Then every \(f\co R^{\cat{G}}\ra R^{\cat{F}}\) in \(\He\) lifting \(\phi\) is surjective and the ideal \(\ker\hspace{-0.1em}f\) is generated by any lifting of any \(k\)-basis for \(V\)\textup{.} In particular\textup{,} \(\ker\hspace{-0.1em} f\) is generated by \textup{`}linear forms\textup{'} modulo \(\im\fr{m}_{\vL}{\cdot} R^{\cat{G}}\)\textup{.}
\end{lem}
%%%%%%%%%%%%%%%%%%%%%%%%%%%%%%%%%%%%
%%%%%%%%%%% SECTION %%%%%%%%%%%%%%%%%%
%%%%%%%%%%%%%%%%%%%%%%%%%%%%%%%%%%%%
\section{Versal deformations of pairs}\label{sec.versal2}
We prove existence of a versal deformation of a pair (algebra, module) with equidimensional and isolated singularity by establishing the necessary conditions in \cite[3.2]{art:74}.

Let $A$ be an algebraic \(k\)-algebra and \(N\) a finite \(A\)-module. Then a \emph{deformation} of the pair \((A,N)\) to an object $\vL\ra S\ra k$ in $\He$ is a commutative diagram of $\vL$-algebra homomorphisms 
\begin{equation}\label{eq.dA}
\begin{aligned}
\xymatrix@C-3pt@R-6pt@H-0pt{
S\ar[r]\ar[d] & \mc{A}\ar[d] 
\\
k\ar[r] & A
}
\end{aligned}
\end{equation}  
and a map of \(\mc{A}\)-modules \(\mc{N}\ra N\) such that:
\begin{enumerate}[leftmargin=2.4em, label=\textup{(\roman*)}]
\item The map \(S\ra \mc{A}\) is flat, local and algebraic.
\item The map \(\mc{A}\ra A\) induces an isomorphism \(\mc{A}\ot_{S}k\cong A\).
\item \(\mc{N}\) is an \(S\)-flat and finite \(\mc{A}\)-module.
\item The map \(\mc{N}\ra N\) induces an isomorphism \(\mc{N}\ot_{S}k\cong N\).
\end{enumerate}
For brevity we will say that the pair \((\mc{A},\mc{N})\) is a deformation of \((A,N)\) to \(S\). Let \((\mc{A}_{i},\mc{N}_{i})\) be a deformation of \((A,N)\) to \(S_{i}\) for \(i=1,2\).
Then a map between the two deformations \((g,f,\alpha)\co(S_{1}\ra\mc{A}_{1},\mc{N}_{1})\ra (S_{2}\ra\mc{A}_{2},\mc{N}_{2})\) is a map \(g\co S_{1}\ra S_{2}\) in \(\He\), a map of \(S_{1}\)-algebras \(f\co\mc{A}_{1}\ra \mc{A}_{2}\) and a map of \(\mc{A}_{1}\)-modules \(\alpha\co\mc{N}_{1}\ra\mc{N}_{2}\) such that the maps commute with the maps to \((k\ra A,N)\). If \(\mc{A}_{2}'=\mc{A}_{1}\hot_{S_{1}}S_{2}\) denotes the henselisation of \(\mc{A}_{1}\ot_{S_{1}}S_{2}\) then \(\mc{A}_{2}'\) is an algebraic and flat \(S_{2}\)-algebra and \((\mc{A}_{2}',\mc{N}_{1}\ot_{\mc{A}_{1}}\mc{A}_{2}')\) is a deformation of \((A,N)\) to \(S_{2}\). If the induced map \(({\id},f',\alpha')\co(S_2\ra\mc{A}'_{2},\mc{N}_{1}\ot_{\mc{A}_{1}}\mc{A}_{2}')\ra (S_2\ra \mc{A}_{2},\mc{N}_{2})\) is an isomorphism (i.e. in each factor) then $(g,f,\alpha)$ is a cocartesian map. We obtain a cofibred category $\cat{Def}_{(A,N)}\ra \He$ and a (forgetful) map of cofibred categories $\cat{Def}_{(A,N)}\ra \cat{Def}_{A}$.
It induces a map of associated functors of isomorphism classes $\df{}{(A,N)}\ra \df{}{A}$. See also \cite[Sec. 2]{ile:21}.
%%%%%%%%%%%%%%%% LEMMA %%%%%%%%%%%%%%%%
\begin{lem}\label{lem.S1}%2 
Suppose $A$ is an algebraic $k$\textup{-}algebra and $N$ is a finite $A$-module\textup{.} Then 
\begin{enumerate}[leftmargin=2.4em, label=\textup{(\roman*)}]
\item $\cat{Def}_{(A,N)}$ satisfies \textup{($\liminj$).}
\item $\cat{Def}_{(A,N)}$ satisfies \textup{(S1').}
\end{enumerate}
\end{lem}
\begin{proof}
(i) Since an object $(S\ra\mc{A},\mc{N})$ in $\cat{Def}_{(A,N)}(S)$ has $\mc{A}$ algebraic over $S$ and $\mc{N}$ finite over $\mc{A}$ this is a standard argument (a finite number of coefficients in a finite number of equations and \cite[\href{https://stacks.math.columbia.edu/tag/00R6}{Tag 00R6}]{SP}). Cf. \cite[\href{https://stacks.math.columbia.edu/tag/02JO}{Tag 02JO}]{SP}.

(ii) With notation as before \eqref{eq.test}, suppose $a'=(\mc{A}',\mc{N}')$ and  $a^{\dprime}=(\mc{A}^{\dprime},\mc{N}^{\dprime})$ are deformations of $(A,N)$ to $S'$ and $R$ respectively such that the induced deformations $a'_S$ and $a^{\dprime}_S$ are isomorphic to a deformation $a=(\mc{A},\mc{N})$. Then $a'{\times}_aa^{\dprime}:=(\mc{A}'{\times}_{\hspace{-0.12em}\mc{A}}\mc{A}^{\dprime},\mc{N}'{\times}_{\hspace{-0.08em}\mc{N}\hspace{0,1em}}\mc{N}^{\dprime})$ is a deformation of $(A,N)$ to $S'{\times}_{\hspace{-0.08em}S}R$ inducing $a'$ and $a^{\dprime}$; see \cite[3.4]{sch:68} and (the proof of) \cite[4.2]{rim:72}. On the other hand, if $b=(\mc{B},\mc{M})$ is a deformation of $(A,N)$ to $S'{\times}_{\hspace{-0.08em}S}R$ with $b_{S'}\cong a'$ and $b_{R}\cong a^{\dprime}$ then $b$ is naturally isomorphic to $a'{\times}_aa^{\dprime}$. The test map \eqref{eq.test} is full and faithful by the definition of fibre products.
\end{proof}
%%%%%%%%%%%% LEMMA %%%%%%%%%%%%%
\begin{lem}\label{lem.AQhens}%3(4)
Suppose \(S\) is local henselian ring with \(S/\fr{m}_{S}\cong k\) and \(h^{\tn{ft}}\co S\ra \mc{A}^{\tn{ft}}\) is a finite type and flat ring homomorphism\textup{.} 
Suppose \(\vG^{\tn{ft}}=\mc{A}^{\tn{ft}}{\oplus} \mc{N}^{\tn{ft}}\) is a graded \(S\)\textup{-}algebra with an $S$\textup{-}flat and finite \(\mc{A}^{\tn{ft}}\)\textup{-}module \(\mc{N}^{\tn{ft}}\) in degree \(1\)\textup{.} Let \(\mc{A}\) denote the henselisation of \(\mc{A}^{\tn{ft}}\) in a maximal ideal \(\fr{m}\)\textup{.} Put \(\vG=\mc{A}\ot_{\mc{A}^{\tn{ft}}}\vG^{\tn{ft}}\)\textup{.}
Let \(J^{\tn{ft}}=J_0^{\tn{ft}}{\oplus}\hspace{0.12em}\hspace{-0.06em}J_1^{\tn{ft}}\) be a graded \(\vG^{\tn{ft}}\)\textup{-}module with the finite \(\mc{A}^{\tn{ft}}\)\textup{-}module \(J_i^{\tn{ft}}\) in degree \(i\)\textup{.}

For all $n$ there are natural \(\mc{A}^{\tn{ft}}_{\fr{m}}\)\textup{-}linear maps of graded Andr{\'e}-Quillen cohomology
\begin{equation*}
(*)\qquad\gxt{n}{\vG^{\tn{ft}}}{L^{\tn{gr}}_{\vG^{\tn{ft}}\!/\!S}}{J^{\tn{ft}}}_{\fr{m}}
\lra
\gxt{n}{\vG}{L^{\tn{gr}}_{\vG\!/\!S}}{\mc{A}\ot_{\mc{A}^{\tn{ft}}} J^{\tn{ft}}}\,.
\end{equation*}
Put $A^{\tn{ft}}{\oplus}N^{\tn{ft}}=\vG^{\tn{ft}}\ot_Sk$ and \(\fr{m}_{0}=\fr{m}A^{\tn{ft}}\) and assume that \(V=\Spec A^{\tn{ft}}{\smallsetminus}\{\fr{m}_{0}\}\) is smooth over \(k\) and \(N^{\tn{ft}}\) restricted to \(V\) is locally free\textup{.} Then $\gxt{n}{\vG}{L^{\tn{gr}}_{\vG\!/\!S}}{\mc{A}\ot_{\mc{A}^{\tn{ft}}} J^{\tn{ft}}}$ is finite as \(S\)\textup{-}module and $(*)$ is an isomorphism for $n>0$\textup{.}
\end{lem}
%%%%%%%%%%
\begin{proof}
Put $J=\mc{A}\ot_{\mc{A}^{\tn{ft}}} J^{\tn{ft}}$.
The map $(*)$ is given as the composition  
\begin{equation}\label{eq.id1}
\gxt{n}{\vG^{\tn{ft}}}{L^{\tn{gr}}_{\vG^{\tn{ft}}\!/\!S}}{J^{\tn{ft}}}_{\fr{m}}
\xra{{\id}\ot 1}
\gxt{n}{\vG^{\tn{ft}}}{L^{\tn{gr}}_{\vG^{\tn{ft}}\!/\!S}}{J^{\tn{ft}}}_{\fr{m}}\ot_{\hspace{-0.09em}\mc{A}^{\tn{ft}}_{\fr{m}}}\mc{A}
\cong
\xt{n}{\vG}{L_{\!\mc{A}\hspace{-0.06em}/\!S}}{J}
\end{equation}
where the last map is the flat base-change map in Lemma \ref{lem.ft}.

Put $A\op N=\vG\ot_Sk$.
Let \(\fr{a}\sbeq \mc{A}^{\tn{ft}}\) denote an ideal defining the non-smooth locus of \(h^{\tn{ft}}\); cf. \cite[\href{https://stacks.math.columbia.edu/tag/07C4}{Tag 07C4}]{SP}. Then \(\fr{a}_{0}=\fr{a}A^{\tn{ft}}\) defines the non-smooth locus of \(k\ra A^{\tn{ft}}\). Put \(\wbar{\mc{A}}^{\tn{ft}}=\mc{A}^{\tn{ft}}/\fr{a}\). Since \(A^{\tn{ft}}/\fr{a}_{0}\) has finite length by assumption, \(S\ra \wbar{\mc{A}}^{\tn{ft}}\) is quasi-finite at \(\fr{m}\) by \cite[\href{https://stacks.math.columbia.edu/tag/00PK}{Tag 00PK}]{SP}.
Moreover, since \(S\) is henselian there is a Zariski neighbourhood \(U\) of \(\fr{m}\) in \(\Spec \mc{A}^{\tn{ft}}\) such that the non-smooth locus \(U\cap V(\fr{a})\) is finite over \(S\); cf. \cite[\href{https://stacks.math.columbia.edu/tag/04GG}{Tag 04GG}]{SP}. Suppose $n>0$. Then the support of \(\cH^{n}=\xt{n}{\mc{A}^{\tn{ft}}}{L_{\!\mc{A}^{\tn{ft}}\!/S}}{J_0^{\tn{ft}}}\) in \(U\) is contained in \(U\cap V(\fr{a})\) by \cite[III 3.1.2]{ill:71} and Lemma \ref{lem.ft}. Since the localisation \(\cH^{n}_{\fr{m}}\) equals \(\cH^{n}\) restricted to \(U\), it follows that \(\cH^{n}_{\fr{m}}\) is finite as \(S\)-module. Since \(\wbar{\mc{A}}^{\tn{ft}}_{\fr{m}}\) is henselian by \cite[Chap. I, \S2.2]{ray:70}, ${\id}\ot 1\co \cH^{n}_{\fr{m}} \ra \cH^{n}_{\fr{m}}\ot_{\hspace{-0.09em}\mc{A}^{\tn{ft}}_{\fr{m}}}\mc{A}$ is an isomorphism.

Let \(\fr{a}'\sbeq \mc{A}^{\tn{ft}}\) denote an ideal defining the locus where \(\mc{N}^{\tn{ft}}\) is not locally free. Then \(\fr{a}'_{0}=\fr{a}'\!A^{\tn{ft}}\) defines the singular locus of \(N\). Again there is a Zariski neighbourhood \(U'\) of \(\fr{m}\) such that \(\cE^{n}=\xt{n}{\mc{A}^{\tn{ft}}}{\mc{N}^{\tn{ft}}}{J_1^{\tn{ft}}}\) restricted to \(U'\) equals \(\cE^{n}_{\fr{m}}\) and is finite as \(S\)-module. Then ${\id}\ot 1\co \cE^{n}_{\fr{m}} \ra \cE^{n}_{\fr{m}}\ot_{\hspace{-0.09em}\mc{A}^{\tn{ft}}_{\fr{m}}}\mc{A}$ is an isomorphism.
By the natural long exact sequence in Theorem \ref{thm.lang} and the \(5\)-lemma, ${\id}\ot 1$ in \eqref{eq.id1} is an isomorphism.
\end{proof}
Suppose \(A\) is a \(k\)-algebra and \(N\) a finite \(A\)-module such that the pair is algebraic, i.e. there is a finite type $k$-algebra $A^\tn{ft}$ with a maximal ideal $\fr{m}_0$ and a finite $A^\tn{ft}$-module $N^{\tn{ft}}$ such that $A\cong(A^\tn{ft})^{\tn{h}}_{\fr{m}_0}$ and $N\cong A\ot_{A^\tn{ft}} N^{\tn{ft}}$.
Then \(A\) is an \emph{isolated singularity over \(k\)} if \(A^{\tn{ft}}\) can be taken to be smooth over \(k\) at all points in \(\Spec A^{\tn{ft}} \smallsetminus\{\fr{m}_{0}\}\). The pair \((A,N)\) has an \emph{isolated singularity over \(k\)} if, in addition, \(N^{\tn{ft}}_{\fr{p}}\) is a free \(A^{\tn{ft}}_{\fr{p}}\)-module for all points $\fr{p}\neq \fr{m}_{0}$ in a Zariski neighbourhood of $\fr{m}_{0}$. 
%%%%%%%%%%%%%%%% LEMMA %%%%%%%%%%%%%%%%
\begin{lem}\label{lem.S2}%1
Suppose the pair $(A,N)$ has an isolated singularity\textup{.} Then
\begin{enumerate}[leftmargin=2.4em, label=\textup{(\roman*)}] 
\item $\cat{Def}_{(A,N)}$ satisfies \textup{(S2).}
\end{enumerate}
\begin{enumerate}[resume*]
\item If $\Spec A$ in addition is equidimensional\textup{,} then $\cat{Def}_{(A,N)}$ satisfies \textup{(iso).}
\end{enumerate}
\end{lem}
%%%%%%%%
\begin{proof}
(i) For any $S=S_{\tn{red}}$ in $\He$, $a=(\mc{A},\mc{N})$ in $\cat{Def}_{(A,N)}(S)$ and finite $S$-module $I$, put $\vG=\mc{A}{\oplus}\mc{N}$ and $J=\vG\ot_S I$. Then $\df{}{(\mc{A},\mc{N})}(S{\oplus}I)\cong \gxt{1}{\vG}{L^{\tn{gr}}_{\vG/S}}{J}$ by Proposition \ref{prop.ext}. Then Lemma \ref{lem.AQhens} gives the result.

(ii) With notation as in Definition \ref{def.S1}, put $\wbar{S}=S_I\hspace{-0.18em}\hat{}\;$. Let \(\ulset{i}{\,a}=(\,\ulset{i}{\mc{A}},\ulset{i}{\mc{N}})\) for \(i=1,2\) be two objects in \(\cat{Def}_{(A,N)}(\wbar{S})\) and let \(\{\theta_{n}\co \ulset{1}{\,a_{n}}\cong\ulset{2}{\,a_{n}}\}\) be a tower of isomorphisms between the \(S_{n}\)-truncations. Let $\ulset{i}{\mc{A}}^{\tn{ft}}$ be an \(\wbar{S}\)-flat finite type  representative of\, $\ulset{i}{\mc{A}}$ such that \(\ulset{i}{\mc{A}}^{\tn{ft}}\ot_{\!\wbar{S}}k\) has a single non-smooth closed point. We proceed as in the proof of \cite[Lemme p. 600]{elk:73}. Put $\ulset{i}I=I\!\cdot\ulset{i}{\mc{A}}^{\tn{ft}}$. The composition
\begin{equation}
\ulset{1}{\mc{A}}^{\tn{ft}}\lra (\,\ulset{1}{\mc{A}}^{\tn{ft}})_{\ulset{1}I}\hspace{-0.38em}^{\wh{\,}}\hspace{0.38em}\xra{\hat{\hspace{0.4em}\theta\hspace{0.4em}}}(\,\ulset{2}{\mc{A}}^{\tn{ft}})_{\ulset{2}I}\hspace{-0.38em}^{\wh{\,}}\hspace{0.38em}
\end{equation}
gives the $(\,\ulset{2}{\mc{A}}^{\tn{ft}})_{\ulset{2}I}\hspace{-0.38em}^{\wh{\,}}\hspace{0.38em}$-algebra map $f\co(\,\ulset{2}{\mc{A}}^{\tn{ft}})_{\ulset{2}I}\hspace{-0.38em}^{\wh{\,}}\hspace{0.38em}\ot_{\!\wbar{S}}\,\ulset{1}{\mc{A}}^{\tn{ft}}\ra (\,\ulset{2}{\mc{A}}^{\tn{ft}})_{\ulset{2}I}\hspace{-0.38em}^{\wh{\,}}\hspace{0.38em}$. After possibly inverting some element in $\ulset{1}{\mc{A}}^{\tn{ft}}$ we may assume that the conditions in \cite[Th{\'e}or{\`e}me 2 bis]{elk:73} are satisfied (cf.\ \cite[Part 2, Thm. 3.1]{art:76}). We get an $\ulset{2}{\mc{A}}$-algebra map $g\co \ulset{2}{\mc{A}}\ot_{\!\wbar{S}\,}\ulset{1}{\mc{A}}^{\tn{ft}}\ra \ulset{2}{\mc{A}}$ approximating $f$ to order $1$. Precomposing with $1\ot{\id}\co\ulset{1}{\mc{A}}^{\tn{ft}}\ra  \ulset{2}{\mc{A}}\ot_{\!\wbar{S}\,}\ulset{1}{\mc{A}}^{\tn{ft}}$ gives an $\wbar{S}$-algebra map which factors through the henselisation $\ulset{1}{\mc{A}}^{\tn{ft}}\ra\ulset{1}{\mc{A}}$ giving an $\wbar{S}$-algebra map $\tilde{\theta}^0\co \ulset{1}{\mc{A}}\ra \ulset{2}{\mc{A}}$ which modulo $I^2$ equals the degree $0$-part $\theta_1^0\co\ulset{1}{\mc{A}}_1\ra\ulset{2}{\mc{A}}_1$ of $\theta_1$. Then $\tilde{\theta}^0$ is an isomorphism (e.g. similarily to the proof of \cite[6.3]{ile:21}).
By `linear approximation', cf. \cite[6.1]{ile:12}, $\tilde{\theta}^0$ is extended to an isomorphism of the pairs \(\ulset{1}\,a\cong\ulset{2}\,a\) which lifts \(\theta_{1}\). 
\end{proof}
The following theorem is deduced from arguments of Elkik and von Essen, but with certain technical twists. In von Essen's \cite{ess:90} the algebra is only deforming trivially and is not a local henselian ring. To apply Elkik's \cite{elk:73} we work with a finite type version of the groupoid. Let $(A^{\tn{ft}},N^{\tn{ft}})$ be a pair where $A^{\tn{ft}}$ is finite type $k$-algebra and $N^{\tn{ft}}$ is a finite $A^{\tn{ft}}$-module such that $(A^{\tn{ft}},N^{\tn{ft}})$ has an isolated singularity. There is a cofibred category $\cat{Def}^{\tn{ft}}_{(A^{\tn{ft}},N^{\tn{ft}})}\ra\He$ where an object is a flat and finite type pair $a=(S\ra \mc{A}^{\tn{ft}},\mc{N}^{\tn{ft}})$ mapping to $a_0=(k\ra A^{\tn{ft}},N^{\tn{ft}})$ such that $a\ot_Sk\cong a_0$ and the singular locus of $a$ is finite over $S$. Base change is by ordinary tensor product.
%%%%%%%%%%%%%%%%% THEOREM %%%%%%%%%%%%%%%%%%
\begin{thm}\label{thm.vers}%4(9)
Suppose the pair \((A,N)\) has an isolated singularity over the field \(k\) with \(A\) equidimensional\textup{.}
Then \(\cat{Def}_{(A,N)}\) has a versal object\textup{.}
\end{thm}
%%%%%%%
\begin{proof}
We apply \cite[3.2]{art:74} with the extension to arbitrary excellent coefficient rings given by \cite{con/jon:02} to show the existence of a formally versal object in \(\cat{Def}_{(A,N)}\). By Lemmas \ref{lem.S1} and \ref{lem.S2}, $\cat{Def}_{(A,N)}$ satisfies ($\liminj$), (S1'), (S2) and (iso.).

For effectivity, note that the restriction $\cat{Def}^{\tn{ft}}_{(A^{\tn{ft}},N^{\tn{ft}})}\ra\Ar$ satisfies (S1') and (S2). 
By \cite[2.11]{sch:68} there exists a formally versal formal object \(((\mc{A}_{n}^{\tn{ft}},\mc{N}^{\tn{ft}}_{n}))\) in \(\limproj\cat{Def}^{\tn{ft}}_{(A^{\tn{ft}},N^{\tn{ft}})}(S_{n})\) where \(S_{n}=S/\fr{m}_{S}^{n+1}\) for some \(S=\hat{S}\) in \(\He\). 
By \cite[Th\'{e}or\`{e}m 7, p.\ 595]{elk:73} (cf.\ \cite[Part 2, Thm. 5.1]{art:76}) there exists an object \(S\ra \mc{A}^{\tn{ft}}\) in \(\cat{Def}^{\tn{ft}}_{A^{\tn{ft}}}(S)\) which induces \(\{\mc{A}_{n}^{\tn{ft}}\}\). 
Put $\mc{A}^{*}=\limproj\mc{A}^{\tn{ft}}_n$.
Then \(\mc{N}^{*}:=\limproj\mc{N}_{n}^{\tn{ft}}\) is an \(S\)-flat finite \(\mc{A}^{*}\)-module; \cite[\href{https://stacks.math.columbia.edu/tag/0912}{Tag 0912}, \href{https://stacks.math.columbia.edu/tag/031D}{Tag 031D}]{SP}.  

Let $f\co\Spec\mc{A}^*\ra \Spec S$ denote the induced map and  
$U^*$ the set of points $x\in\Spec \mc{A}^*$ such that $\mc{N}^*\!\ot_{\!S}k(f(x))$ is locally free at $x$. Let $J^*$ denote the ideal defining $\Spec\mc{A}^*\smallsetminus U^*$ with reduced structure. 
Put $\wbar{\mc{A}}^*=\mc{A}^*/J^*$, \(J^{\tn{ft}}=\ker(\mc{A}^{\tn{ft}}\ra \wbar{\mc{A}}^*)\), $\wbar{\mc{A}}^{\tn{ft}}=\mc{A}^{\tn{ft}}/J^{\tn{ft}}$, and $\wbar{A}^{\tn{ft}}=A^{\tn{ft}}/(J^*{\cdot}A^{\tn{ft}})$. Since $\wbar{A}^{\tn{ft}}$ has finite length, $\wbar{\mc{A}}^*$ is finite as $S$-module; \cite[\href{https://stacks.math.columbia.edu/tag/031D}{Tag 031D}]{SP}. Moreover, the inclusion $\wbar{\mc{A}}^{\tn{ft}}\ra \wbar{\mc{A}}^*$ is an isomorphism by Nakayama's lemma. Let \(\mc{A}'\) denote the henselisation of $\mc{A}^{\tn{ft}}$ in $J^{\tn{ft}}\cap\,\fr{m}_S{\cdot} \mc{A}^{\tn{ft}}\!$. Put $J=J^{\tn{ft}}{\cdot} \mc{A}'$. Then $\mc{A}'/J\cong \wbar{\mc{A}}^{\tn{ft}}\cong \wbar{\mc{A}}^*$ which is complete with respect to $\fr{n}=\fr{m}_{S}\mc{A}'$. It follows that $\mc{A}'/J^s\oplus\mc{A}'/\fr{n}^s$ and its submodule $\mc{C}_s:=\mc{A}'/(J^s\cap\fr{n}^s)$ is $\fr{n}$-complete for all $s$. By Artin-Rees there is a $t\geq s$ such that $J^{t}\cap\fr{n}^{t}\sbeq (J{\cdot}\,\fr{n})^s{\cdot}\,\fr{n}^{t-s}\sbeq (J\cap\fr{n})^s$. Since the image of $\fr{n}$ is contained in $\Rad\mc{C}_t$ and $\mc{C}_t$ surjects onto $\mc{A}'/(J\cap\fr{n})^s$ it follows that the latter is $\fr{n}$-complete for all $s$. This implies that the natural map $\limproj \mc{A}'/(J\cap\fr{n})^s\ra \mc{A}^{*}$ has an inverse. 
Since \(\mc{N}^{*}\) is locally free on the complement of \(V((J\cap\fr{n})\mc{A}^{*})\), the conditions in \cite[Th\'{e}or\`{e}m 3]{elk:73} hold for the henselian pair $(\mc{A}',J\cap\fr{n})$. We obtain a finite \(\mc{A}'\)-module \(\mc{N}'\) which induces \(\mc{N}^{*}\). Moreover, \(\mc{N}'\) is $S$-flat by \cite[\href{https://stacks.math.columbia.edu/tag/0523}{Tag 0523}]{SP}. Let $(\mc{A},\mc{N})$ denote the henselisation of the pair $(\mc{A}',\mc{N}')$ in the maximal ideal $(\mc{A}'\ra A)^{-1}(\fr{m}_A)$.

We claim that the henselisation map \(\phi\co\cat{Def}^{\tn{ft}}_{(A^{\tn{ft}},N^{\tn{ft}})}\ra\cat{Def}_{(A,N)}\) is formally smooth. It follows that the object \((\mc{A},\mc{N})\) in \(\cat{Def}_{(A,N)}(S)\) is formally versal. For the claim, put \(\vG^{\tn{ft}}_0=A^{\tn{ft}}{\oplus}N^{\tn{ft}}\) and \(\vG_0=A{\oplus} N\) and let \(\pi\co S^{\dprime}\ra S^{\prime}\cong S^{\dprime}\!/I\) be a small surjection in \(\Ar\) and \(\vG^{\tn{ft}}\) a deformation of \(\vG^{\tn{ft}}_{0}\) over $S^{\prime}$. The map 
\begin{equation}\label{eq.h}
\gxt{n}{\vG_0^{\tn{ft}}}{L^{\tn{gr}}_{\vG_0^{\tn{ft}}\!/k}}{\vG_0^{\tn{ft}}}\ot_{k}I
\lra
\gxt{n}{\vG_0}{L^{\tn{gr}}_{\vG_0\!/k}}{\vG_0}\ot_{k}I
\end{equation}
is an isomorphism for $n=1,2$ by Lemma \ref{lem.AQhens}. We have that \eqref{eq.h} for $n=2$ takes the obstruction \(\ob(\pi,\vG^{\tn{ft}})\in\gxt{2}{\vG_0^{\tn{ft}}}{L^{\tn{gr}}_{\vG_0^{\tn{ft}}\!/k}}{\vG_0^{\tn{ft}}}\ot_{k}I$ for lifting \(\vG^{\tn{ft}}\) along \(\pi\) to the corresponding obstruction $\ob(\pi,\vG)$ for the henselised pair, see Proposition \ref{prop.omap}. Similarly the torsor action commutes with the map \eqref{eq.h} for $n=1$. Formal smoothness of $\phi$ follows by the standard obstruction argument; cf. the proof of Corollary \ref{cor.vers}.

By \cite[3.2]{art:74} and \cite[1.5]{con/jon:02} there is an algebraic \(\vL\)-algebra \(R\) in $\He$ and a formally versal object \((\mc{A}^{\tn{v}},\mc{N}^{\tn{v}})\) in \(\cat{Def}_{(A,N)}(R)\). By Proposition \ref{prop.fvers} we conclude that \((\mc{A}^{\tn{v}},\mc{N}^{\tn{v}})\) is versal. 
\end{proof}
Fix a flat and algebraic ring homomorphism $\vL\ra \mc{A}$. 
Assume $\mc{A}\ot_{\vL}k\cong A$. Let $q\co \mc{A}\ra A$ denote the induced map. Note that $\id\co \He\ra \He$ is a cofibred category. Then $\vL\ra \mc{A}$ and $q$ defines a map of cofibred categories $\He\ra \cat{Def}_A$ where the object $\vL\ra S\ra k$ maps to  $q\hot(S{\ra} k)\co\mc{A}_S=\mc{A}\hot_{\vL}S\ra A$. For a finite $A$-module $N$, define $\cat{Def}^{\mc{A}}_N$ as the ($2$-category) fibre product $\He\times_{\cat{Def}_A}\cat{Def}_{(A,N)}$. An object in $\cat{Def}^{\mc{A}}_N$ is an object $b=(S\ra\mc{B}\ra A,\,\mc{N}\ra N)$ in $\cat{Def}_{(A,N)}$ such that $\mc{B}\ra A$ is isomorphic to $\mc{A}_S\ra A$ in $\cat{Def}_A$. Then $b$ is called an $\mc{A}$-deformation of $N$ to $S$. Maps of deformations are defined as above Lemma \ref{lem.S1}. Then $\cat{Def}^{\mc{A}}_N$ is a cofibred subcategory of $\cat{Def}_{(A,N)}$ locally of finite presentation satisfying (S1'). Denote the associated deformation functor \(\df{{\mc{A}}}{N}\). A special (`classical') case is given by \(\vL=k\) and \(q=\id\co {\mc{A}}=A\). See also \cite[Sec. 2]{ile:21}.

An argument similar to the proof of Theorem \ref{thm.vers} gives the following (cf. \cite[2.4]{ess:90}).
%%%%%%%%%%%%%%%%% PROPOSITION %%%%%%%%%%%%%%%%%%
\begin{prop}\label{prop.vers}%1
Assume $N$ is locally free on the complement of the closed point in $\Spec A$\textup{.} Then \(\cat{Def}^{\mc{A}}_{N}\) has a versal object\textup{.}
\end{prop}
%%%%%%%%%%%%%%%%% COROLLARY %%%%%%%%%%%%%%%%%%
\begin{cor}\label{cor.vers}%2
If $\xt{2}{A}{N}{N}=0$ then the forgetful map $\phi\co \cat{Def}_{(A,N)}\ra \cat{Def}_{A}$ is formally smooth\textup{.} If in addition the pair \((A,N)\) has an isolated singularity over the field \(k\) with \(A\) equidimensional\textup{,} then $\phi$ is smooth\textup{.}
\end{cor}
\begin{proof}
Results in Section \ref{sec.coh} imply that $\phi$ is formally smooth:
Consider a deformation $(S\ra\mc{A},\mc{N})$ in $\cat{Def}_{(A,N)}$. Put \(\vG=\mc{A}{\oplus}\mc{N}\) and \(\vG_{0}=A{\oplus} N\). Suppose $I$ is an $S$-module with $\fr{m}_S\cdot I=0$. For all \(n\) there are isomorphisms
\begin{equation}\label{eq.defgrade}
\begin{aligned}
\gxt{n}{\vG}{L^{\tn{gr}}_{\vG/S}}{\vG\ot_{S} I}
&\cong
\gxt{n}{\vG}{L^{\tn{gr}}_{\vG/S}}{\vG\ot_{S}k\ot_k I}
\cong
\gxt{n}{\vG}{L^{\tn{gr}}_{\vG/S}}{\vG_{0}}\ot_k I
\\
&\cong
\gxt{n}{\vG_{0}}{L^{\tn{gr}}_{\vG_{0}/k}}{\vG_{0}}\ot_{k} I\quad\tn{by \eqref{eq.e}}.
\end{aligned}
\end{equation}
Given a map of deformations $(S'\ra\mc{A}')\ra (S\ra\mc{A})$ in $\cat{Def}_{A}$ where $\pi\co S'\ra S$ is a surjection in $\Ar$ with kernel $I$. We may assume $\fr{m}_{S'}{\cdot} I=0$. By Theorem \ref{thm.lang} (i), the obstruction $\ob(\pi,\vG)$ maps to $\ob(\pi,\mc{A})=0$ by the restriction 
\begin{equation}
i^*\ot{\id}\co \gxt{2}{\vG_0}{L^{\tn{gr}}_{\vG_0{/}\vL}}{\vG_0}\ot_kI\lra\xt{2}{A}{L_{A/\vL}}{A}\ot_kI
\end{equation}
which is injective by assumption. By Proposition \ref{prop.ext} there is a deformation $S'\ra\vG^{\dprime}=\mc{A}^{\dprime}{\oplus}\mc{N}^{\dprime}$ inducing $S\ra\vG$. The difference of the deformations $\mc{A}'$ and $\mc{A}^{\dprime}$ is an element $\xi_0\in \xt{1}{A}{L_{A/k}}{\vG_{0}}\ot_{k} I$. By assumption and Theorem \ref{thm.lang} there is a $\xi\in \gxt{1}{\vG_{0}}{L^{\tn{gr}}_{\vG_{0}/k}}{\vG_{0}}\ot_{k}I$ restricting to $\xi_0$. The torsor action of $\xi$ on $S'\ra\vG^{\dprime}$ gives a deformation $S'\ra\vG'$ inducing both $S\ra\vG$ and $S'\ra\mc{A}'$. By the additional assumption Theorem \ref{thm.vers} gives a versal element $a=(R\ra\mc{A}^{\tn{v}},\mc{N}^{\tn{v}})$. Since $\phi$ is formally smooth $\phi(a)=(R\ra\mc{A}^{\tn{v}})$ is formally versal in $\cat{Def}_{A}$. By Proposition \ref{prop.fvers}, $\phi(a)$ is versal and then $\phi$ is smooth.
\end{proof}
%%%%%%%%%%%%%%%%%%%%%%%%%%%%%%%%%%%%%%%%%%%%%%
%%%%%%%%%%%%%%%%%%%% SECTION %%%%%%%%%%%%%%%%%%%%%
%%%%%%%%%%%%%%%%%%%%%%%%%%%%%%%%%%%%%%%%%%%%%%
\section{Cohomology of Cohen-Macaulay approximations}\label{sec.cohmap}
We would like to apply Lemma \ref{lem.res} to the Cohen-Macaulay approximation maps for deformations $\sigma_X\co \df{}{(A,N)}\ra \df{}{(A,X)}$ defined in \cite[Sec. 3]{ile:21} for $X=M,L'$, see Lemma \ref{lem.defmap}.
By the proof of Corollary \ref{cor.vers}, the $\df{}{(A,X)}$ have obstruction theories (Definition \ref{def.obtheory}).
We construct natural maps $\sigma^n_j$ between the Andr\'e-Quillen cohomology groups of the pairs. Injectivity of a map of $\Ext^n$s for $n=1,2$ imply the conditions in Lemma \ref{lem.res} since the torsor action commutes with the $\sigma^1_j$ and $\sigma^2_j(\tn{obstruction})=\tn{obstruction}$. The conclusions are given in Theorems \ref{thm.defext1} and \ref{thm.defext2}.
%%%%%%%%%%%%%% PROPOSITION %%%%%%%%%%%%%
\begin{prop}\label{prop.cohmap}%1
Let \(h\co S\ra \mc{A}\) be a finite type or algebraic Cohen\textup{-}Macaulay map and let \(\mc{N}\) be an \(S\)\textup{-}flat finite \(\mc{A}\)-module\textup{.} 
Consider the relative Cohen-Macaulay approximation sequences in \eqref{eq.CMseq}\textup{.}
Let \(\mc{X}_j\) denote \(\mc{N}\)\textup{,} \(\mc{M}\) and \(\mc{L}'\) for \(j=0,1,2\) respectively\textup{,} and put \(\vG_j=\mc{A}\oplus\mc{X}_j\)\textup{.} Let \(I\) be any \(S\)\textup{-}module\textup{.}
Then there are natural maps of distinguished triangles obtained from transitivity triangles \eqref{eq.dist}\textup{:}
\begin{equation*}
\xymatrix@C-3pt@R-6pt@H-0pt{
\grhm{}{\vG_0}{L^{\tn{gr}}_{\vG_0/\mc{A}/S}}{\vG_0\ot I}\ar@<-3.45em>[d]^(0.44){(\id,\pi)^{*}}
&
& \grhm{}{\vG_0}{L^{\tn{gr}}_{\vG_0/\mc{A}/S}}{\vG_0\ot I}\ar@<-3.45em>[d]^{(\id,\iota)_{*}}
\\
\grhm{}{\vG_1}{L^{\tn{gr}}_{\vG_1/\mc{A}/S}}{\vG_0\ot I}
& \tn{and}
& \grhm{}{\vG_0}{L^{\tn{gr}}_{\vG_0/\mc{A}/S}}{\vG_2\ot I}
\\
\grhm{}{\vG_1}{L^{\tn{gr}}_{\vG_1/\mc{A}/S}}{\vG_1\ot I}\ar@<3.45em>[u]_(0.44){(\id,\pi)_{*}}
&
& \grhm{}{\vG_2}{L^{\tn{gr}}_{\vG_2/\mc{A}/S}}{\vG_2\ot I}\ar@<3.45em>[u]_(0.44){(\id,\iota)^{*}}
}
\end{equation*}
The induced maps of graded cohomology 
\begin{align*}
(\id,\pi)_{*} & \co \gxt{n}{\vG_1}{L^{\tn{gr}}_{\vG_1/S}}{\vG_1\ot I}\lra
\gxt{n}{\vG_1}{L^{\tn{gr}}_{\vG_1/S}}{\vG_0\ot I}\quad\tn{and} 
\\
(\id,\iota)^{*} & \co\gxt{n}{\vG_2}{L^{\tn{gr}}_{\vG_2/S}}{\vG_2\ot I}\lra
\gxt{n}{\vG_0}{L^{\tn{gr}}_{\vG_0/S}}{\vG_2\ot I}
\end{align*}
are isomorphisms for \(n>0\) and surjections for \(n=0\)\textup{.}
\end{prop}
%%%%%%%%
\begin{proof}
In the derived category \(\grhm{}{\vG}{L^{\tn{gr}}_{\vG/-}}{-}\) is represented by the Yoneda complex \(\ghm{\bdot}{\vG}{L^{\tn{gr}}_{\vG/-}}{-}\) since \(L^{\tn{gr}}_{\vG/-}\) is a bounded above complex of \(\vG\)-projective modules; cf.  \cite[\href{https://stacks.math.columbia.edu/tag/0A66}{Tag 0A66}]{SP}. The distinguished triangles are obtained by applying \(\ghm{\bdot}{\vG}{-}{\vG\ot I}\) to the triangles \(L^{\tn{gr}}_{\vG_i/\mc{A}/S}\).
The map \((\id,\pi)\co \vG_1\ra\vG_0\) gives a map of triangles \(L^{\tn{gr}}_{\vG_1/\mc{A}/S}\ra L^{\tn{gr}}_{\vG_0/\mc{A}/S}\) and a map \(\vG_1\ot I\ra\vG_0\ot I\) which induce the maps in the left column.
By the identification 
\begin{equation}
\gxt{n}{\vG_1}{L^{\tn{gr}}_{\vG_1/\mc{A}}}{\vG_j\ot I}\cong \xt{n}{\mc{A}}{\mc{M}}{\mc{X}_j\ot I}
\end{equation}
in \eqref{eq.AQext}, \((\id,\pi)_{*}\) equals \(\pi_{*}\co \xt{n}{\mc{A}}{\mc{M}}{\mc{M}\ot I}\ra \xt{n}{\mc{A}}{\mc{M}}{\mc{N}\ot I}\). Since \(\vG_j\) is \(S\)-flat base change theory implies that \(\xt{n}{\mc{A}}{\mc{M}}{\mc{L}\ot I}=0\) for all \(n>0\); \cite[5.1(ii)]{ile:12}. Then \(\pi_{*}\) is an isomorphism for \(n>0\) and surjective for \(n=0\). Since 
\begin{equation}
\gxt{n}{\vG_1}{L_{\mc{A}/S}\ot^{\BB{L}}\vG_1}{\vG_j\ot I}\cong\xt{n}{\mc{A}}{L_{\mc{A}/S}}{\mc{A}\ot I}\, \tn{ for all }\, n\,,
\end{equation}
the claims about cohomology follow by diagram chase in the ladder obtained from the map \((\id,\pi)_{*}\) of triangles.  
The \((\id,\iota)\)-case is similar.
\end{proof}
Composition of \((\id,\pi)^{*}\) (respectively \((\id,\iota)_{*})\) with the inverse of the isomorphism \((\id,\pi)_{*}\) (respectively \((\id,\iota)^{*}\)) obtained from Proposition \ref{prop.cohmap} gives the natural maps $\sigma^{n}_{1}$ (respectively $\sigma^{n}_{2}$) in the commutative diagram
\begin{equation}\label{eq.ladder}
\xymatrix@C-9pt@R-6pt@H-0pt{
\cT^{n-1}\ar[r]^(0.3){\partial_0}\ar@{=}[d] 
& \xt{n}{\mc{A}}{\mc{N}}{\mc{N}\ot I}\ar[r]\ar[d]^(0.43){\eta^n_j} 
& \gxt{n}{\vG_0}{L^{\tn{gr}}_{\vG_0/S}}{\vG_0\ot I}\ar[r]\ar[d]^(0.43){\sigma^n_j}
& \cT^n\ar[r]^(0.25){\partial_0}\ar@{=}[d] 
& \xt{n+1}{\mc{A}}{\mc{N}}{\mc{N}\ot I}\ar[d]^(0.43){\eta^{n+1}_j}
\\
\cT^{n-1}\ar[r]^(0.3){\partial_j}
& \xt{n}{\mc{A}}{\mc{X}_j}{\mc{X}_j\ot I}\ar[r] 
& \gxt{n}{\vG_j}{L^{\tn{gr}}_{\vG_j/S}}{\vG_j\ot I}\ar[r] 
& \cT^n\ar[r]^(0.25){\partial_j}
& \xt{n+1}{\mc{A}}{\mc{X}_j}{\mc{X}_j\ot I}
}
\end{equation} 
of the long exact sequences in Theorem \ref{thm.lang} where the connecting maps $\partial_i=\partial(L^{\tn{gr}}_{\vG_i/\mc{A}/S})$ for $i=0,1,2$ and $\cT^n=\xt{n}{\mc{A}}{L_{\mc{A}/S}}{\mc{A}\ot I}$. The composition 
\begin{equation}\label{eq.eta1}
\xymatrix@C-3pt@R-6pt@H-0pt{
\xt{n}{\mc{A}}{\mc{N}}{\mc{N}\ot I}\ar[r]^(0.49){\pi^*}
& \xt{n}{\mc{A}}{\mc{M}}{\mc{N}\ot I}
& \xt{n}{\mc{A}}{\mc{M}}{\mc{M}\ot I}\ar[l]_(0.49){\simeq}
}
\end{equation}
is $\eta^{n}_{1}$ and similarily $\eta^n_2=(\iota^*)^{-1}\iota_*$; see \cite[5.4.1]{ile:21}.

The $\sigma^i_j$ and $\eta^i_j$ are compatible with obstructions and torsor actions. 
%%%%%%%%% COROLLARY %%%%%%%%%
\begin{cor}\label{cor.omap}%1
Suppose \(\alpha\co 0\ra I\ra R\ra S\ra 0\) is an extension of rings\textup{;} cf\textup{. Definition \ref{def.ext}.} Assume \(j\in\{1,2\}\)\textup{.} Then 
\begin{enumerate}[leftmargin=2.4em, label=\textup{(\roman*)}]
\item \(\sigma^{2}_{j}(\ob(\alpha,\vG_0))=\ob(\alpha,\vG_j)\). 
\end{enumerate}
Let \(\beta_{0}\co 0\ra \mc{A}\ot I\ra \mc{B}\ra \mc{A}\ra 0\) be an $1\ot\id_I$-extension of \(\mc{A}\) above \(\alpha\)\textup{.} Then 
\begin{enumerate}[resume*]
\item \(\eta^{2}_{j}(\ob(\beta_{0},\mc{N}))=\ob(\beta_{0},\mc{X}_j)\)\textup{.}
\item If \(\xi\in\xt{1}{\mc{A}}{L_{\mc{A}/S}}{\mc{A}\ot I}\) then \(\eta^{2}_{j}(\ob(\beta_{0}+\xi,\mc{N}))=\ob(\beta_{0}+\xi,\mc{X}_j)\)\textup{.}
\end{enumerate}
Let \({}^{0\hspace{-0.15em}}\beta\co 0\ra\vG_0\ot I\ra 
\mc{B}\op\mc{N}^{\tn{e}}\ra\vG_0\ra 0\) denote a $1\ot\id_I$-extension of \(\vG_0\) above \(\alpha\) and let \({}^{j\hspace{-0.2em}}\beta\co 0\ra\vG_j\ot I\ra\mc{B}\op\,\mc{X}_j^{\tn{e}}\ra\vG_j\ra 0\) be the associated $1\ot\id_I$-extension induced by a relative MCM approximation $\mc{X}_1^{\tn{e}}\ra \mc{N}^{\tn{e}}$\textup{,} respectively a relative FID hull $\mc{N}^{\tn{e}}\ra\mc{X}_2^{\tn{e}}$ \textup{(}cf\textup{.} \cite[5.1, 5.7]{ile:12}\textup{).}
\begin{enumerate}[resume*]
\item The torsor action on \({}^{0\hspace{-0.15em}}\beta\) is compatible with the torsor action on \({}^{j\hspace{-0.2em}}\beta\) through \(\sigma_{j}^{1}\)\textup{.}
\item If the degree \(0\)-part $\ulset{i}{\beta_0}$ of the extensions \({}^{i\hspace{-0.2em}}\beta\) \textup{(}\(i=0,1,2\)\textup{)} is a fixed extension \(\beta_{0}\)\textup{,} then the torsor action of \(\xt{1}{\mc{A}}{\mc{N}}{\mc{N}\ot_{S}I}\) on the $(\phi=\id)$\textup{-}extension 
\({}^{0\hspace{-0.15em}}\beta_{1}\) \textup{(}in degree $1$\textup{)} above \(\beta_{0}\) is compatible through \(\eta_{j}^{1}\) with the torsor action of \(\xt{1}{\mc{A}}{\mc{X}_j}{\mc{X}_j\ot_{S}I}\) on the $\id$\textup{-}extension \({}^{j\hspace{-0.2em}}\beta_{1}\) above \(\beta_{0}\)\textup{.}
\end{enumerate}
\end{cor}
%%%%%%%%%%%%%
\begin{proof}
Proposition \ref{prop.omap} (i) gives 
\begin{equation}
\begin{aligned}
(\id,\pi)^{*}\ob(\alpha,\vG_0)
&=\ob(\alpha,\vG_1,1_{\vG_0}\ot\id_{I})=(\id,\pi)_{*}\ob(\alpha,\vG_1)
\\
(\id,\iota)_{*}\ob(\alpha,\vG_0)
&=\ob(\alpha,\vG_0,1_{\vG_2}\ot\id_{I})=(\id,\iota)^{*}\ob(\alpha,\vG_2)
\end{aligned}
\end{equation}
which is (i) and (together with Theorem \ref{thm.lang}) also (ii). Since the connecting maps in \eqref{eq.ladder}
commutes with \(\eta^{2}_{j}\), Theorem \ref{thm.lang} implies (iii). The compatibility of torsor actions in Proposition \ref{prop.omap} gives (iv) and (v).
\end{proof}
%%%%%%%%% COROLLARY %%%%%%%%%
\begin{cor}\label{cor.cohmap}%2
Assume $n\geq 1$\textup{.}
\begin{enumerate}[leftmargin=2.4em, label=\textup{(\roman*)}]
\item If \(\pi^{*}\co \xt{n}{\mc{A}}{\mc{N}}{\mc{N}\ot I}\ra\xt{n}{\mc{A}}{\mc{M}}{\mc{N}\ot I}\) is injective then \(\sigma^{n}_{1}\)  is injective\textup{.} If \(\pi^{*}\) is an isomorphism in degree \(n\) and injective in degree \(n+1\) then \(\sigma^{n}_{1}\) is an isomorphism\textup{.}
\item If \(\iota_{*}\co \xt{n}{\mc{A}}{\mc{N}}{\mc{N}\ot I}\ra\xt{n}{\mc{A}}{\mc{N}}{\mc{L}'\ot I}\) is injective then \(\sigma^{n}_{2}\) is injective\textup{.} If \(\iota_*\) is an isomorphism in degree \(n\) and injective in degree \(n+1\) then \(\sigma^{n}_{2}\) is an isomorphism\textup{.}
\end{enumerate}
\end{cor}
%%%%%%%%%
\begin{proof}
(i) Consider the commutative diagram \eqref{eq.ladder} with $j=1$. Note that $\eta^n_1$ is injective (respectively an isomorphism) by \eqref{eq.eta1}. Diagram chase in \eqref{eq.ladder} gives the results. The argument for (ii) is similar.
\end{proof}
%%%%%%%%%%%%%%%%%%%%%%%%%%%%%%%%%%%%%%%%%%%%%%
%%%%%%%%%%%%%%%%%%%% SECTION %%%%%%%%%%%%%%%%%%%%%
%%%%%%%%%%%%%%%%%%%%%%%%%%%%%%%%%%%%%%%%%%%%%%
\section{Cohen-Macaulay approximation of versal deformations}\label{sec.appl}
The following notions are used throughout this section.
Let \(A\) be a Cohen-Macaulay local algebraic \(k\)-algebra and \(N\) a finite \(A\)-module.
Fix a minimal MCM approximation \(0\ra L\ra M\xra{\,\pi\,} N\ra0\) and a minimal FID hull \(0\ra N\xra{\,\iota\,} L'\ra M'\ra0\) as in \eqref{eq.CMseq0}. For each deformation \((S\ra \mc{B},\mc{N})\) of \((k\ra A,N)\) to \(S\) in $\He$ there are maps of short exact sequences:
\begin{equation}\label{eq.defapprx}
\begin{aligned}
\xymatrix@C-6pt@R-22pt@H-0pt{
0\ar[r] & \mc{L}\ar[dd]\ar[r] & \mc{M}\ar[dd]\ar[r] & \mc{N}\ar[dd]\ar[r] & 0 
&&  
0\ar[r] & \mc{N}\ar[dd]\ar[r] & \mc{L}'\ar[dd]\ar[r] & \mc{M}'\ar[dd]\ar[r] & 0
\\
&&&&& \hspace{-2em}\tn{and}\hspace{-2em} &&&&
\\
0\ar[r] & L\ar[r] & M\ar[r] & N\ar[r] & 0
&&
0\ar[r] & N\ar[r] & L'\ar[r] & M'\ar[r] & 0
}
\end{aligned}
\end{equation} 
where the upper sequences are maps of \(S\)-flat finite \(\mc{B}\)-modules and the vertical maps induce isomorphisms of short exact sequences after applying \(-\ot_{S}k\). The upper sequences are minimal and unique up to isomorphism and are called a \(\cat{MCM}_{\mc{B}/S}\)-approximation, respectively a \(\hat{\cat{D}}_{\mc{B}/S}\)-hull of \((S\ra \mc{B},\mc{N})\) in \cite{ile:12,ile:21}. 
Henselian base change of the upper sequences along a map \(S\ra S'\) gives new such sequences. We fix a choice of diagrams \eqref{eq.defapprx} for each deformation \((S\ra \mc{B},\mc{N})\). 
%%%%%%%%%%%%%% LEMMA %%%%%%%%%%%%%%
\begin{lem}[{\cite[3.1]{ile:21}}]\label{lem.defmap}%1
If \(X\) is any of the modules \(M,L',L,M'\) and \(\mc{X}\) denotes the deformation of $X$ in \eqref{eq.defapprx} then the association \((S\ra \mc{B},\mc{N})\mapsto (S\ra \mc{B},\mc{X})\) induces well\textup{-}defined maps of deformation functors
\begin{equation*}
\sigma_{X}\co\df{}{(A,N)}\lra \df{}{(A,X)}\,.
\end{equation*}
For a flat and algebraic ring homomorphism $\vL\ra \mc{A}$ the association \((S\ra \mc{A}_S,\mc{N})\mapsto (S\ra \mc{A}_S,\mc{X})\) induces well\textup{-}defined maps of deformation functors of $\mc{A}$-modules
\begin{equation*}
\sigma^{\mc{A}}_{X}\co\df{\mc{A}}{N}\lra \df{\mc{A}}{X}\,.
\end{equation*}
\end{lem}
For the next two results suppose $k_0\ra k$ is a separable field extension and $R_{Y}$ in $\He$ is an algebraic and minimal base ring for a formally versal element in $\df{}{(A,Y)}$ (or in \(\df{\mc{A}}{Y}\)). Versal elements exist if $(A,N)$ has an isolated singularity by Theorems \ref{thm.defgrade} and \ref{thm.defgrade2}.
%%%%%%%%%%%%%%%%%%%%%%% TEOREM %%%%%%%%%%%%%%%%%%%%%%%
\begin{thm}\label{thm.defext1}%0(2)
If\, \(\iota_{*}\co\xt{j}{A}{N}{N}\ra\xt{j}{A}{N}{L'}\) is injective for \(j=1,2\) then 
\begin{equation*}
R_{N}\cong R_{L'}/J
\end{equation*}
where the ideal \(J\) is generated by elements lifting a \(k\)\textup{-}basis of $\ker(\sigma_{L'}(k[\vare])^*)$ and $\sigma_{L'}(k[\vare])^*$ is the $k$-dual of the map of relative Zariski vector spaces \textup{(}cf\textup{.} \textup{Lemma \ref{lem.res}).} 
In particular \(J\) is generated by \textup{`}linear forms\textup{'} modulo \(\im\fr{m}_{\vL}{\cdot} R_{L'}\)\,\textup{.}
\end{thm}
%%%%%%%%%%%%%
\begin{proof}
Suppose \((S\ra \mc{A},\mc{N})\) is a deformation of \((A,N)\) and \(R\ra S\) is a small surjection in \(\He\) with kernel \(I\). Put \(\vG=\mc{A}{\oplus}\mc{N}\) and \(\vG_{0}=A{\oplus} N\). By \eqref{eq.defgrade} it follows from Proposition \ref{prop.ext} that \(\df{}{(A,N)}\) has an obstruction theory (Definitition \ref{def.obtheory}) with \(\cH^{2}(I)=\gxt{2}{\vG_{0}}{L^{\tn{gr}}_{\vG_{0}/k}}{\vG_{0}}\ot_{k}I\). Moreover, \eqref{eq.defgrade}, the conditions on \(\iota_{*}\), and Corollaries \ref{cor.omap} and \ref{cor.cohmap} imply that conditions (i) and (ii) in Lemma \ref{lem.res} are satisfied and the result follows.
\end{proof}
A similar proof gives:
%%%%%%%%%%%%%%%%%%%%%%% TEOREM %%%%%%%%%%%%%%%%%%%%%%%
\begin{thm}\label{thm.defext2}%1(3)
If\, \(\pi^{*}\co \xt{j}{A}{N}{N}\ra\xt{j}{A}{M}{N}\) is injective for \(j=1,2\) then 
\begin{equation*}
R_{N}\cong R_{M}/J
\end{equation*}
where the ideal \(J\) is generated by elements lifting a \(k\)\textup{-}basis of $\ker(\sigma_{M}(k[\vare])^*)$\textup{.}
In particular \(J\) is generated by \textup{`}linear forms\textup{'} modulo \(\im\fr{m}_{\vL}{\cdot} R_{M}\)\,\textup{.}
\end{thm}
In the isolated singularity case Theorem \ref{thm.vers} allows the main results in \cite{ile:21} to be strengthened.
%%%%%%%%%%%%%%%%%%%%%%% THEOREM %%%%%%%%%%%%%%%%%%%%%%%
\begin{thm}\label{thm.defgrade}%6(9)
Suppose \((A,N)\) has an isolated singularity\textup{.} Then\textup{:}
\begin{enumerate}[leftmargin=2.4em, label=\textup{(\roman*)}]
\item The functors \(\df{}{(A,N)}\)\textup{,} \(\df{}{(A,L')}\) and  \(\df{}{(A,L)}\) have versal elements\textup{.}

\item If\, \(\xt{1}{A}{N}{M'}=0\) then the map \(\sigma_{L'}\co \df{}{(A,N)}\ra\df{}{(A,L')}\) is smooth\textup{.}

\item If\, \(\xt{j}{A}{N}{M'}=0\) for \(j=0,1\) then \(\sigma_{L'}\) is an isomorphism\textup{.}
\item If\, \(\xt{2}{A}{N}{M}=0\) then the map $\sigma_{L}\co \df{}{(A,N)}\lra\df{}{(A,L)}$ is smooth\textup{.}
\item If\, \(\xt{j}{A}{N}{M}=0\) for \(j=1,2\) then $\sigma_L$ is an isomorphism\textup{.}
\end{enumerate}
\end{thm}
%%%%%%%
\begin{proof}
(i) Since \(A\) is CM, \(\Spec A\) is equidimensional and \(\df{}{(A,N)}\) has a versal element by Theorem \ref{thm.vers}. Let \(\fr{p}\in \Spec A\), \(\fr{p}\neq\fr{m}_{A}\). Localising the Cohen-Macaulay approximation diagram \cite[p. 6]{aus/buc:89} in \(\fr{p}\) gives a commutative diagram with four short exact sequences and the same homological properties:
\begin{equation}\label{eq.CMdiag}
\xymatrix@C-3pt@R-10pt@H-6pt{
L_{\fr{p}}\ar@{=}[d]\ar@{^{(}->}[r] & M_{\fr{p}}\ar@{->>}[r]^(0.48){\pi}\ar@{}[dr]|-{\Box}\ar@{_{(}->}[d] & N_{\fr{p}}\ar@{_{(}->}[d]^(0.44){\iota}
\\
L_{\fr{p}}\ar@{^{(}->}[r] & \omega^{\oplus n}_{\fr{p}}\ar@{->>}[r]\ar@{->>}[d] & L'_{\fr{p}}\ar@{->>}[d]
\\
& M'_{\fr{p}}\ar@{=}[r] 
& M'_{\fr{p}} 
&
}
\end{equation} 
Since \(N_{\fr{p}}\cong A_{\fr{p}}^{\oplus r}\) by assumption, \(\pi\) splits and \(L_{\fr{p}}\) is contained in \(\Add\{\omega_{\fr{p}}\}\) by \cite[3.7]{aus/buc:89}. Since \(A_{\fr{p}}\) is a regular local ring, \(\omega_{\fr{p}}\cong A_{\fr{p}}\) and \(L_{\fr{p}}\) and \(M_{\fr{p}}\) are finite, free \(A_{\fr{p}}\)-modules. Then \(M'_{\fr{p}}\) and \(L'_{\fr{p}}\) are finite, free \(A_{\fr{p}}\)-modules, too. By Theorem \ref{thm.vers}, \(\df{}{(A,L')}\) and  \(\df{}{(A,L)}\) have versal elements. 
Then (ii-v) follow from (i) and \cite[5.5]{ile:21}.
\end{proof}
A similar proof (using \cite[5.7]{ile:21}) gives:
%%%%%%%%%%%%%%%%%%%%%%% THEOREM %%%%%%%%%%%%%%%%%%%%%%%
\begin{thm}\label{thm.defgrade2}%5(7)
Suppose \((A,N)\) has an isolated singularity\textup{.} Then\textup{:}
\begin{enumerate}[leftmargin=2.4em, label=\textup{(\roman*)}]
\item The functors \(\df{}{(A,N)}\)\textup{,} \(\df{}{(A,M)}\) and  \(\df{}{(A,M')}\) have versal elements\textup{.}
\item If\, \(\xt{1}{A}{L}{N}=0\) then the map \(\sigma_{M}\co \df{}{(A,N)}\ra\df{}{(A,M)}\) is smooth\textup{.}
\item If\, \(\xt{j}{A}{L}{N}=0\) for \(j=0,1\) then \(\sigma_{M}\) is an isomorphism\textup{.}
\item If\, \(\xt{2}{A}{L'}{N}=0\) then the map $\sigma_{M'}\co \df{}{(A,N)}\lra\df{}{(A,M')}$ is smooth\textup{.}
\item If\, \(\xt{j}{A}{L'}{N}=0\) for \(j=1,2\) then $\sigma_{M'}$ is an isomorphism\textup{.}
\end{enumerate}
\end{thm}
By Proposition \ref{prop.vers} analogous statements to Theorems \ref{thm.defgrade} and \ref{thm.defgrade2} follows for the $\sigma^{\mc{A}}_{X}$-maps.
%%%%%%%%%%%%%% COROLLARY %%%%%%%%%%%%%%%
\begin{cor}\label{cor.defgrade}%0
Assume $N$ is locally free on the complement of the closed point in $\Spec A$\textup{.} 
Then the functors \(\df{\mc{A}}{N}\)\textup{,} \(\df{\mc{A}}{L'}\) and  \(\df{\mc{A}}{L}\) have versal elements\textup{.} Moreover\textup{,} the statements \textup{(ii-v)} in \textup{Theorem \ref{thm.defgrade}} hold if $\sigma_X$ is changed to $\sigma^{\mc{A}}_{X}$\textup{.}
\end{cor}
%%%%%%%%%%%%%% COROLLARY %%%%%%%%%%%%%%%
\begin{cor}\label{cor.defgrade2}%0
Assume $N$ is locally free on the complement of the closed point in $\Spec A$\textup{.}
Then the functors \(\df{\mc{A}}{N}\)\textup{,} \(\df{\mc{A}}{M}\) and  \(\df{\mc{A}}{M'}\) have versal elements\textup{.} Moreover\textup{,} the statements \textup{(ii-v)} in \textup{Theorem \ref{thm.defgrade2}} hold if $\sigma_X$ is changed to $\sigma^{\mc{A}}_{X}$\textup{.}
\end{cor}
%%%%%%%%%%% REMARK %%%%%%%%%%%%
\begin{rem}\label{rem.base}
The statements about smoothness in the above results have the following consequences. Let $R_Y$ be the base ring of a versal element in $\df{}{(A,Y)}$.
\begin{enumerate}[leftmargin=2.4em, label=\textup{(\roman*)}]
\item There is map  $f\co R_X\ra R_N$ in $\He$ lifting $\sigma_X\co \df{}{(A,N)}\ra \df{}{(A,X)}$ to $f^*\co h_{R_N}\ra h_{R_X}$. Formal versality or formal smoothness is not enough. 

\item There is a finite-type-over-$\vL$ representative $f^{\tn{ft}}\co R^{\tn{ft}}_X\ra R^{\tn{ft}}_N$ of $f$ which is smooth. 

\item One can also assume that there are finite type and $R^{\tn{ft}}_N$-flat pairs $(R^{\tn{ft}}_N\ra\mc{A}^{\tn{ft}},\mc{N}^{\tn{ft}})$ and $(R^{\tn{ft}}_N\ra\mc{A}^{\tn{ft}},\mc{X}^{\tn{ft}})$ with relative Cohen-Macaulay approximations $\mc{M}^{\tn{ft}}\ra \mc{N}^{\tn{ft}}$ and $\mc{N}^{\tn{ft}}\ra\mc{L}'{}^{\tn{ft}}$ and inducing versal elements after henselisation. 
\end{enumerate}
\end{rem}
%%%%%%%%%%%%%%%%%%%%%%%%%%%%%%%%%%%%%
\section{Framings and examples}\label{sec.frame}
%%%%%%%%%%%%%%%%%%%%%%%%%%%%%%%%%%%%%
We define a cofibred category of deformations of framed pairs. Suppose $A$ is an algebraic $k$-algebra and \(N\) is a finite \(A\)-module. Fix an $A$-linear surjection $A^{\oplus n}\ra N$, called a framing. A deformation of the pair $(k\ra A,A^n\ra N)$ to $S$ in $\He$ is a pair of commutative diagrams
\begin{equation}\label{eq.Ak}
\begin{aligned}
\xymatrix@C-3pt@R-8pt@H-0pt{
S\ar[r]\ar[d] & \mc{A}\ar[d]_(0.4){\tau}
\\
k\ar[r] & A
}
\qquad
\xymatrix@C-3pt@R-8pt@H-0pt{
\mc{A}^{\oplus n}\ar[r]\ar[d]_(0.36){\tau^{\oplus n}} & \mc{N}\ar[d]
\\
A^{\oplus n}\ar[r] & N
}
\end{aligned}
\end{equation}
with $\mc{A}^{\oplus n}\ra \mc{N}$ $\mc{A}$-linear such that forgetting the framings gives a deformation of pairs $(\mc{A},\mc{N})\ra (A,N)$. 
A map of deformations of framed pairs is a map of deformations of pairs \((g,f,\alpha)\co (S_1\ra\mc{A}_{1},\mc{N}_{1})\ra (S_2\ra\mc{A}_{2},\mc{N}_{2})\), see below \eqref{eq.dA}, commuting with the framings: $\mc{A}_1^{\oplus n}\ra \mc{N}_{1}$ composed with $\alpha\co \mc{N}_{1}\ra \mc{N}_{2}$ equals the composition of $f^{\oplus n}\co \mc{A}_1^{\oplus n}\ra\mc{A}_2^{\oplus n}$ with $\mc{A}_2^{\oplus n}\ra \mc{N}_{2}$. The map $(g,f,\alpha)$ is cocartesian if it is cocartesian as a map of (unframed) pairs. Let $\cat{Def}_{(A,A^n\ra N)}\ra \He$ denote the resulting cofibred category. Forgetting the framing gives a map of cofibred categories $\cat{Def}_{(A,\,A^n\ra N)}\ra\cat{Def}_{(A,\,N)}$ which induces a map of the associated deformation functors $\df{}{(A,\,A^n\ra N)}\ra\df{}{(A,\,N)}$.
%%%%%%%%%%%%% LEMMA %%%%%%%%%%%%% 
\begin{lem}\label{lem.Ak}%3
Suppose $A$ is an algebraic $k$-algebra and \(N\) is a finite \(A\)-module with a framing $A^{\oplus n}\ra N$\textup{.}
\begin{enumerate}[leftmargin=2.4em, label=\textup{(\roman*)}]
\item The forgetful map $\df{}{(A,\,A^n\ra N)}\ra\df{}{(A,\,N)}$ is smooth\textup{.}

\item If $A$ in addition is in $\He$ and the framing is the residue map $A\ra k$ then 
%\begin{equation*}
$
\df{}{(A,\,A\ra k)}\ra\df{}{(A,\,k)} 
$ 
%\end{equation*}
is an isomorphism\textup{.}
\end{enumerate}
\end{lem}
\begin{proof}
(i) Suppose $R\ra S$ is a surjection in $\He$ and $b=(\mc{A},\mc{N})$ is a deformation to $R$ of $(A,N)$ such that the base change $(\mc{A}_S,\mc{N}_S)$ has a framing $a=(\mc{A}_S,p\co\mc{A}_S^n\ra \mc{N}_S)$ deforming $(A,A^n\ra N)$. Any lifting $q\co\mc{A}^n\ra \mc{N}$ of $p$ makes $(\mc{A},q)$ a deformation of $(A,A^n\ra N)$ mapping to $a$ by base change and mapping to $b$ by forgetting the framing.

(ii) Assume $(\mc{A}_1,\,p_1\co\mc{A}_1\ra\mc{N}_1)$ and $(\mc{A}_2,\,p_2\co\mc{A}_2\ra\mc{N}_2)$ are two deformations of $(A,A\ra k)$ to $S$ such that $(\mc{A}_1,\mc{N}_1)$ and $(\mc{A}_2,\mc{N}_2)$ are isomorphic deformations of $(A,\,k)$. We may assume that $\mc{A}_1=\mc{A}_2=\mc{A}$. There is a unit $u\in\mc{A}$ such that $p_1(-)=p_2(u\cdot -)$. Then $\ker p_1=\ker p_2$ and $\alpha(n):=p_2(p_1^{-1}(n))$ gives a well defined isomorphism $(\id,\alpha)\co (\mc{A},p_1)\ra (\mc{A},p_2)$.
\end{proof}
%%%%%%%%%%%%% PROPOSITION %%%%%%%%%%%%%
\begin{prop}\label{prop.Ak}%2
Suppose $A$ is an algebraic $k$-algebra in $\He$\textup{.} Assume the map $(R\ra\mc{A})\ra(k\ra A)$ is a versal element for $\df{}{A}\co \He\ra\Sets$\textup{.} Let $\mc{A}\ra \mc{A}_{\mc{A}}=\mc{A}\hot_R\mc{A}$ denote the base change of the versal element by $R\ra\mc{A}$ and consider $\mc{A}$ as cyclic $\mc{A}_{\mc{A}}$\textup{-}module by the multiplication map\textup{.} 
\begin{enumerate}[leftmargin=2.4em, label=\textup{(\roman*)}]
\item The natural map of pairs $b\co (\mc{A}\ra \mc{A}_{\mc{A}},\mc{A})\ra(k\ra A,\,k)$ is a versal element for $\df{}{(A,\,k)}$\textup{.}
\item If $L'$ is the minimal FID hull of $k$ and $\dim A\geq 2$ then the $\vL$-algebra $\mc{A}$ is a versal base for $\df{}{(A,L')}$ and for $\df{}{(A,Q')}$ where \(Q'=\hm{}{A}{\omega_{A}}{L'}\) has finite projective dimension\textup{.}
\end{enumerate}
\end{prop}
\begin{proof}
(i) By \cite[6.7]{ile:14} (phrased for $k_0=k$, but the proof is valid in our case too) the natural map of framed pairs 
\begin{equation}
a\co (\mc{A}\ra \mc{A}_{\mc{A}},\mc{A}_{\mc{A}}\ra\mc{A})\lra(k\ra A,\,A\ra k)
\end{equation}
is a versal element for $\df{}{(A,\,A\ra k)}$. By Lemma \ref{lem.Ak} the image $b$ of $a$ is a versal element for $\df{}{(A,\,k)}$.

(ii) The first part is (i) and Theorem \ref{thm.defgrade}. The second part follows then from the eqivalence of cofibred categories
\begin{equation}\label{eq.LQ}
\cat{Def}_{(A,L')}\xra{\;\,\simeq\;} \cat{Def}_{(A,Q')}
\end{equation}
defined by $(S\ra \mc{B},\mc{L}')\mapsto (S\ra \mc{B},\hm{}{\mc{B}}{\omega_{\mc{B}/S}}{\mc{L}'})$, see \cite[6.10(i)]{ile:12}. Applying $\hm{}{A}{\omega_{A}}{-}$ to a finite $\Add\{\omega_A\}$-resolution of $L'$ gives a finite $A$-free resolution of $Q'$.
\end{proof}
%%%%%%%%%%%%% COROLLARY %%%%%%%%%%%%%
\begin{cor}\label{cor.Ak1}%2
Suppose $A=(k[x_1,\dots,x_m]/(f))^{\tn{h}}$ where $f=f_1,\dots,f_c$ is a regular sequence contained in $(x_1,\dots,x_m)^2$\textup{.} Assume $A$ has an isolated singularity and $k_0=k$\textup{.} Let $R\ra\mc{A}$ be a minimal versal element for $\df{}{A}$\textup{.}

Then $\mc{A}$ is the henselisation of a smooth $\vL$-algebra and is a versal base ring for $\df{}{(A,\,k)}$\textup{.} If $\dim A\geq 2$ this is also a versal base ring for $\df{}{(A,L')}$ where $L'$ is the minimal FID hull of $k$\textup{.}
\end{cor}
\begin{proof}
By Proposition \ref{prop.ext} and Lemma \ref{lem.AQhens}, $\dim_k\df{}{A}(k[\vare])=\dim_k\xt{1}{A}{L_{A/k}}{A}$ is a finite number $n$. Then $R=\vL[t_1\,,\dots,\,t_{m+n-c}]^{\tn{h}}$ is a versal base ring for $\df{}{(A,A\ra k)}$ by \cite[6.9]{ile:14} and by Lemma \ref{lem.Ak} also for $\df{}{(A,\,k)}$. The last part is Theorem \ref{thm.defgrade}.  
\end{proof}
\begin{ex}
\label{ex.Ak1}
Assume $\dim A=2$. Applying $(-)^{\vee}=\hm{}{A}{-}{\omega_A}$ to the minimal free resolution of the residue field $\dots\xra{d_3} A^{\beta_2}\xra{d_2}A^{\beta_1}\xra{(\ul{x})} A\thr k$ gives the commutative diagram
\begin{equation}\label{ex.dim2CMdiag}
\xymatrix@C-2pt@R-8pt@H+15pt{
0\ar[r] & \omega_A\ar[r]^(0.48){(\ul{x})^{\text{tr}}}\ar@{=}[d] & \omega_A^{\oplus\beta_1}\ar[r]\ar@{=}[d] & (\syz{}{}\,\fr{m}_A)^{\vee}\ar@{}[dr]|-(0.38){\Box}\ar[r]^(0.48){\pi}\ar@{_{(}->}[d] & \xt{2}{A}{k}{\omega_A}\ar[r]\ar@{_{(}->}[d]^(0.4){\iota} & 0
\\
0\ar[r] & \omega_A\ar[r]^(0.48){(\ul{x})^{\text{tr}}} & \omega_A^{\oplus\beta_1}\ar[r]^(0.48){d_2^{\vee}} & \omega_A^{\oplus \beta_2}\ar[r]\ar@{->>}[d] & L'\ar[r]\ar@{->>}[d] 
& 0
\\
&&& (\syz{}{2}\,\fr{m}_A)^{\vee}\ar@{=}[r] 
& (\syz{}{2}\,\fr{m}_A)^{\vee}
&
}
\end{equation}
with exact rows and short exact columns.
By local duality theory $\xt{2}{A}{k}{\omega_A}\cong k$; cf. \cite[3.3.10-11]{bru/her:98}. Then \eqref{ex.dim2CMdiag} gives the Cohen-Macaulay approximation diagram \cite[p. 6]{aus/buc:89}, cf. \eqref{eq.CMdiag}, of $k$.
Then $\pdim Q=1$.  
By Corollary \ref{cor.vers} and \eqref{eq.LQ}, $\df{}{(A,L)}\cong \df{}{(A,Q)}\ra\df{}{A}$ is formally smooth.
Theorem \ref{thm.defgrade} gives $\df{}{(A,\,k)}\cong \df{}{(A,L')}$. If $R\ra\mc{A}$ is a versal element for $\df{}{A}$ then $\mc{A}$ is a versal base ring for $\df{}{(A,L')}$ by Proposition \ref{prop.Ak}. 
\end{ex}
%%%%%%%%%% EXAMPLE %%%%%%%%%%%
\begin{ex}\label{ex.dim1dim0}
Assume $\dim A=1$. Applying $(-)^{\vee}=\hm{}{A}{-}{\omega_A}$ to the minimal free resolution $\dots\xra{d_2}A^{\beta_1}\xra{(\ul{x})} A\thr k$ gives the MCM approximation
\begin{equation}\label{eq.dim1dim0}
0\ra\omega_A\lra\fr{m}_A^{\vee}\xra{\;\pi\;}\xt{1}{A}{k}{\omega_A}\ra 0
\end{equation}
of $k\cong \xt{1}{A}{k}{\omega_A}$. 
Assume $A$ is Gorenstein. Then $\sigma_{M}\co \df{}{(A,\,k)}\ra \df{}{(A,\fr{m}^{\vee})}$ is smooth by Theorem \ref{thm.defgrade2}. If $R\ra\mc{A}$ is a versal element for $\df{}{A}$ then $\mc{A}$ is a versal base ring for $\df{}{(A,\fr{m}^{\vee})}$ by Proposition \ref{prop.Ak}. But we can say a little more in this case. If $A$ is not regular then $\fr{m}_A^{\vee}\ot_Ak\cong k^{\oplus 2}$; cf. \cite[6.1.iv]{ile:14}. Hence $\hm{}{A}{\fr{m}^{\vee}}{k}\ra \hm{}{A}{A}{k}$ is surjective, $\partial\co \hm{}{A}{A}{k}\ra \xt{1}{A}{k}{k}$ is zero, $\eta^1_1\co \xt{1}{A}{k}{k}\cong \xt{1}{A}{\fr{m}^{\vee}}{\fr{m}^{\vee}}$ while $\eta^2_1\co\xt{2}{A}{k}{k}\ra \xt{2}{A}{\fr{m}^{\vee}}{\fr{m}^{\vee}}$ is injective. Put ${}_{\ccirc\!}\mr{T}^{n}_{(A,X)}=\gxt{n}{A}{L^{\tn{gr}}_{A{\oplus}X/k}}{A}$. Then $\sigma^n_1\co {}_{\ccirc\!}\mr{T}^{n}_{(A,\,k)}\ra {}_{\ccirc\!}\mr{T}^{n}_{(A,\,\fr{m}^{\vee})}$ is an isomorphism for $n=1$ and injective for $n=2$ by Corollary \ref{cor.cohmap}. If $k_0\ra k$ is separable, the minimal versal base rings for $\df{}{(A,\,k)}$ and $\df{}{(A,\fr{m}^{\vee})}$ are isomorphic by Theorem \ref{thm.defext2}. This is not necessarily true for other finite length modules $N$.
\end{ex}
%%%%%%%%%% EXAMPLE %%%%%%%%%%%
\begin{ex}\label{ex.T1}
Suppose $f=f_1,\dots,f_c$ is sequence in $(x_1,\dots,x_m)^2\sbeq k[x_1,\dots,x_m]$ which gives a regular sequence in $P=k[x_1,\dots,x_m]^{\tn{h}}$. Put $I=(f)P$. Assume $A=P/I$ has an isolated singularity over $k$ and $\dim A\geq 1$. Then 
the second fundamental sequence $I/I^2\ra\Omega_{P/k}\ot_PA\ra \Omega_{A/k}$ is short exact and gives a projective resolution of $\Omega_{A}=\Omega_{A/k}$. Applying $(-)^*=\hm{}{A}{-}{A}$ gives the exact sequence
\begin{equation}\label{eq.d2fund}
0\la \xt{1}{A}{\Omega_{A}}{A} \lla (I/I^2)^*\lla(\Omega_{P}\ot_PA)^*\lla \Omega_{A}^* \la 0
\end{equation}
and $\xt{1}{A}{\Omega_{A}}{A}\cong \xt{1}{A}{L_{A/k}}{A}=:\mr{T}^{1}_{\!A}$ which has finite length. 

Assume $\dim A=2$. Let $d_2$ denote the composition  $A^{\beta_2}\!\!\xra{\;\vare\;}\Omega_{A}^{*}\ra (\Omega_{P}\ot_PA)^*$ where $\vare$ is a minimal cover. Applying $(-)^*$ to the splitting of \eqref{eq.d2fund} gives the commutative diagram 
\begin{equation}\label{eq.dd2fund2}
\xymatrix@C-3pt@R-8pt@H+15pt{
0\ar[r] & I/I^2\ar[r]\ar@{=}[d] & \Omega_{P}\ot_PA\ar[r]\ar@{=}[d] & \Omega_A^{**}\ar@{}[dr]|-(0.44){\Box}\ar[r]^(0.44){\pi}\ar@{_{(}->}[d] & (\mr{T}^{1}_{\!A})^{\vee}\ar[r]\ar@{_{(}->}[d]^(0.44){\iota} & 0
\\
0\ar[r] & I/I^2\ar[r] & \Omega_{P}\ot_PA\ar[r]^(0.52){d_2^{*}} & A^{\oplus\beta_2}\ar@{->>}[d]\ar[r] & L'\ar[r]\ar@{->>}[d] & 0
\\
&&& (\Syz\Omega_{A}^{*})^*\ar@{=}[r]
& (\Syz\Omega_{A}^{*})^*
&&
}
\end{equation}
with exact rows and short exact columns inducing the Cohen-Macaulay approximation diagram; cf. \eqref{eq.CMdiag}, where $N=(\mr{T}^{1}_{\!A})^{\vee}=\xt{2}{A}{\mr{T}^{1}_{\!A}}{A}\cong\coker\{\Omega_A\hra \Omega_A^{**}\}$, $M=\Omega_A^{**}$, $L=\Omega_A$, $L'=\coker\{\Omega_A\hra A^{\beta_2}\}$ and $M'=(\Syz\Omega_{A}^{*})^*$.
Since $\mr{T}^{1}_{\!A}$ is a Cohen-Macaulay $A$-module of codimension $2$, $\xt{n}{A}{\mr{T}^{1}_{\!A}}{A}=0$ for all $n\neq 2$ and $\mr{T}^{1}_{\!A}\cong(\mr{T}^{1}_{\!A})^{\vee\vee}$; cf. \cite[3.3.10-11]{bru/her:98}. By cohomology-and-base change for $\Ext$; cf. \cite[2.4]{ile:12}, dualisation extends to an isomorphism $\df{}{(A,\mr{T}^{1}_{\!A})}\cong \df{}{(A,(\mr{T}^{1}_{\!A})^{\vee})}$. 
Composing with $\sigma_{L'}$ gives an isomorphism $\df{}{(A,\mr{T}^{1}_{\!A})}\cong \df{}{(A,L')}$ by Theorem \ref{thm.defgrade}.

If $\dim A=1$ a similar argument gives the MCM approximation
$0\ra I/I^2\ra \syz{}{}(\mr{T}^{1}_{\!A})^*\ra (\mr{T}^{1}_{\!A})^{\vee}\ra 0$ 
where $N=(\mr{T}^{1}_{\!A})^{\vee}=\xt{1}{A}{\mr{T}^{1}_{\!A}}{A}$. The FID hull is given as $0\ra (\mr{T}^{1}_{\!A})^{\vee}\ra \Omega_A\ra\Omega_{A}^{**}\ra 0$.
Composing $\df{}{(A,\mr{T}^{1}_{\!A})}\cong \df{}{(A,(\mr{T}^{1}_{\!A})^{\vee})}$ with $\sigma_M$ gives a map $\df{}{(A,\mr{T}^{1}_{\!A})}\ra \df{}{(A,\,\syz{}{}(\mr{T}^{1}_{\!A})^*\!)}$ which is smooth by Theorem \ref{thm.defgrade2}.

If $\dim A\geq 3$ then $\Omega_A\cong \Omega_{A}^{**}$.
\end{ex}
%%%%%%%%%% REMARK %%%%%%%%%%%
\begin{rem}\label{rem.Mar}
If $\dim A=2$, the fundamental module $F_A$ is given by the MCM approximation sequence $0\ra \omega_A\ra F_A\ra \fr{m}_A\ra 0$ corresponding to $1\in k\cong\xt{2}{A}{k}{\omega_A}\cong\xt{1}{A}{\fr{m}_A}{\omega_A}$. If $k=\BB{C}$ and $A$ is graded the map $\delta\co\Omega_A\ra \fr{m}_A$ corresponding to the Euler derivation, i.e. $\delta(\mr{d}x_i)=\deg(x_i)x_i$, extends to a surjection $\tilde{\delta}\co\Omega^{**}_A\ra \fr{m}_A$ with kernel $\omega_A$; cf. \cite[2.1]{beh:89}. Then $\Omega^{**}_A\cong F_A$. A. Martsinkovsky conjectured that $\Omega^{**}_A\cong F_A$ if and only if $A$ is graded \cite[p. 287]{mar:90}, proving it for quotient singularities and certain hypersurfaces. K. Behnke proved it for rational surface singularities with reduced fundamental cycle, for minimally elliptic singularities and for cusps in \cite{beh:89}. 

Then, in the graded c.i. case $F_A$ approximates $(\mr{T}^{1}_{\!A})^{\vee}$ as well as $\fr{m}_A$; \eqref{eq.dd2fund2}. Moreover, in the graded hypersurface case $\mr{T}^{1}_{\!A}$ is Gorenstein and $(\mr{T}^{1}_{\!A})^{\vee}\cong\mr{T}^{1}_{\!A}$. 

The conjecture was extended by Martsinkovsky to higher dimensions in \cite{mar:91}, with the `if'-case proved for hypersurfaces. Extended to $\dim A=1$, Martsinkovsky's formula gives $\syz{}{}(\mr{T}^{1}_{\!A})^*$ as the MCM approximation of $k$ in the graded case. In the c.i. case $\syz{}{}(\mr{T}^{1}_{\!A})^*$ also approximates $(\mr{T}^{1}_{\!A})^{\vee}$, see Example \ref{ex.T1}.
\end{rem}
%%%%%%%%%%%%%%%%%%%%%%%%%%%%%%%%%%%%%
\section{An application to rational double points}\label{sec.ADE}
%%%%%%%%%%%%%%%%%%%%%%%%%%%%%%%%%%%%%
Suppose $A$ is a ring and $0\ra\Syz N\xra{\,j\,} G\xra{\,p\,} N\ra 0$ a short exact sequence of $A$-modules where $G$ is projective\textup{.}  Let $\Syz$ also denote the composition $\xt{n}{A}{N}{N}\xra{\,\partial\,}\xt{n+1}{A}{ N}{\Syz N}\xla{\;\simeq\;}\xt{n}{A}{\Syz N}{\Syz N}$ for $n\geq 1$\textup{.}
%%%%%%%%%%% LEMMA %%%%%%%%%%%%%
\begin{lem}[{cf. \cite[2.2.8]{kal:05}}]\label{lem.atSyz}%1
Given a ring homomorphism $\vL\ra A$\textup{,} put $\Omega=\Omega_{A/\!\vL}$ and $\im=\im(\id_{\Omega}\!\ot j)$\textup{,} and let $q\co \Omega\ot\Syz N\linebreak[1]\ra\im$ denote the induced surjection\textup{.}  
\begin{enumerate}[leftmargin=2.4em, label=\textup{(\roman*)}]
\item  In the diagram of natural maps
\begin{equation*}
\xymatrix@C-0pt@R-10pt@H-0pt{
\at_{A\!/\!\vL}(N)\in\xt{1}{A}{N}{\Omega\ot N}\ar[r]^(0.6){\partial_*}
& \xt{2}{A}{N}{\im}
\\
-\at_{A\!/\!\vL}(\Syz N)\in\xt{1}{A}{\Syz N}{\Omega\ot \Syz N}\ar[r]^(0.67){q_*}
& \xt{1}{A}{\Syz N}{\im}\ar[u]_(0.47){\partial^*}
}
\end{equation*}
the images of the Atiyah classes have opposite signs\textup{:} 
\begin{equation*}
\partial_*(\at_{A\!/\!\vL}(N))=-(\partial^* q_*)(\at_{A\!/\!\vL}(\Syz N))\in\xt{2}{A}{N}{\im}\,.
\end{equation*}
\item Let $g^N$ denote the Kodaira-Spencer map for $\vL\ra A\ra A\oplus N$ in \textup{\eqref{eq.KSmap}.} The following diagram of natural maps is commutative\textup{:}
\begin{equation*}
\xymatrix@C+20pt@R-10pt@H-0pt{
\Der_{\vL}(A)\ar[d]_(0.44){(-1)}\ar[r]^(0.45){g^{N}}
& \xt{1}{A}{N}{N}\ar[d]^(0.45){\Syz}
\\
\Der_{\vL}(A)\ar[r]^(0.4){g^{\Syz N}}
& \xt{1}{A}{\Syz N}{\Syz N}
}
\end{equation*}
\end{enumerate}
\end{lem}
%%%%%%%
\begin{proof}
(i):
The Atiyah class $\at_{A\!/\!\vL}(N)$ is represented by the sequence \eqref{eq.at}. Splicing $\im\ra\Omega\ot G\ra \Omega\ot N$ with $\at_{A\!/\!\vL}(N)$ at the tail gives the exact sequence
\begin{equation}\label{eq.delKN}
\xymatrix@C+0pt@R-10pt@H-0pt{
0\ar[r]
& \im\ar[r]^(0.4){i}
& \Omega\ot G\ar[r]
& P^1_{A/\!\vL}(N)\ar[r]
& N\ar[r]
& 0
}
\end{equation}
which represents $\partial_*(\at_{A\!/\!\vL}(N))\in\xt{2}{A}{N}{\im}$. Pushout of $\at_{A\!/\!\vL}(\Syz N)$ along $\Omega\ot \Syz N\linebreak[1]\ra\im$ gives the exact sequence
\begin{equation}\label{eq.delKSyzN}
\xymatrix@C+0pt@R-10pt@H-0pt{
0\ar[r]
& \im\ar[r]^{i'}
& E\ar[r]
& \Syz N\ar[r]
& 0
}
\end{equation}
which represents the image of $\at_{A\!/\!\vL}(\Syz N)$ in $\xt{1}{A}{\Syz N}{\im}$.
Note that $P^1(\Syz N)\linebreak[1]\ra P^1(G)$ factors through the induced injection $e\co E\ra P^1(G)$. The third row in the following commutative diagram is obtained by splicing \eqref{eq.delKSyzN} with $\Syz N\ra G\ra N$ in the tail and represents the image of $\at_{A\!/\!\vL}(\Syz N)$ in $\xt{2}{A}{N}{\im}$:
\begin{equation}\label{eq.deld}
\xymatrix@C+6pt@R-6pt@H-0pt{
0\ar[r]
& \im\ar[r]^(0.4){i}
& \Omega\ot G\ar[r]
& P^1(N)\ar[r]
& N\ar@{=}[d]\ar[r]
& 0
\\
0\ar[r]
& \im\ar@{=}[d]\ar[r]^(0.32){-i{\oplus} i'}\ar[u]^{(-1)}
& \Omega\ot G\oplus E\ar[d]^(0.44){\pr_2}\ar[r]^{[\gamma\;\;e]}\ar[u]_{\pr_1}
& P^1(G)\ar[d]^(0.44){\bar{\mu}\ot {\id}}\ar[r]\ar[u]_{{\id}\ot p}
& N\ar@{=}[d]\ar[r]
& 0
\\
0\ar[r]
& \im\ar[r]^{i'}
& E\ar[r]
& G\ar[r]
& N\ar[r]
& 0
}
\end{equation}
Since the first row is \eqref{eq.delKN} and the rows are exact, (i) follows.

(ii): Let $D$ be an element in $\Der_{\vL}(A)$ and $f^D\in\hm{}{A}{\Omega}{A}$ the corresponding map. Then $f^D\ot{\id_G}\co \Omega\ot G\ra G$ restricts to a map $\ol{f^D\ot{\id}}\co \im\ra \Syz N$ which precomposed with $q$ gives $f^D\ot{\id}\co \Omega\ot \Syz N\ra \Syz N$. By (i) the following commutative diagram of natural maps is pointed 
($\at=\at_{A\!/\!\vL}$ and $\Syz =\Syz N$)
\begin{equation}
\begin{aligned}
\xymatrix@C-2pt@R-6pt@H-0pt{
\at(N)\in\xt{1}{A}{N}{\Omega\ot N}\hspace{-1.5em}\ar[dd]_(0.46){(f^D\ot{\id})_*}\ar[dr]^(0.53){\partial_*}
&& \xt{1}{A}{\Syz}{\Omega\ot N}\ni \at(\Syz)\ar[d]^(0.44){q_*}
\\
& \hspace{-5em}\partial_*(\at(N))\in\xt{2}{A}{N}{\im}\ar[d]^(0.46){(\ol{f^D\ot{\id}})_*}
& \xt{1}{A}{\Syz}{\im}\ni q_*(\at(\Syz))\hspace{-2em}\ar[d]^(0.46){(\ol{f^D\ot{\id}})_*}\ar[l]_(0.56){-\partial^*}^(0.55){\cong}
\\
g^N(D)\in\xt{1}{A}{N}{N}\ar[r]^(0.57){\partial_*}
& \xt{2}{A}{N}{\Syz}
& \xt{1}{A}{\Syz}{\Syz}\ni g^{\Syz}(D)\hspace{-2em}\ar[l]_(0.56){-\partial^*}^(0.55){\cong}
}
\end{aligned}
\end{equation}
and the bottom row gives $(-1)$ times the $\Syz$-map.
\end{proof}
%%%%%%%%%%%%% COROLLARY %%%%%%%%%%%%%
\begin{cor}\label{cor.ADE}
Suppose $\vL=k_0=k=\bar{k}$ and $\Char k =0$\textup{.}
Assume $A$ is a two dimensional rational double point \textup{(}$A_n$\textup{,} $D_n$ or $E_n$\textup{)} and $F_A=F$ is the fundamental module of $A$\textup{.} Let $R_k$\textup{,} $R_{\fr{m}}$ and $R_F$ be minimal versal base rings for $\df{}{(A,\,k)}$\textup{,} $\df{}{(A,\fr{m})}$ and $\df{}{(A,F)}$\textup{.} 

Then $R_{\fr{m}}\cong R_F$ and $R_k\cong R_F/(l)\cong \mc{A}$ where $l\in\fr{m}_{R_F}{\smallsetminus}\,\fr{m}_{R_F}^2$ and $R\ra \mc{A}$ is the minimal versal element for $\df{}{A}$\textup{.} In particular\textup{,} $R_F/(l)$ is the henselisation of a smooth $k$-algebra of dimension $n+2$\textup{.}
\end{cor}
\begin{proof}
Put $\fr{m}=\fr{m}_A$. There are maps of deformation functors
\begin{equation}\label{eq.kmF}
\df{}{(A,\,k)}\xra{\;\simeq\;}\df{}{(A, A\ra k)}\lra \df{}{(A,\fr{m})}\xra{\;\sigma_F\;} \df{}{(A,F)}
\end{equation}
where the isomorphism is given by Lemma \ref{lem.Ak}, the middle map is given by taking the syzygy of the framing and $\sigma_F$ is smooth by Theorem \ref{thm.defgrade2}.
There are corresponding maps $R_F\ra R_{\fr{m}}\ra R_k$. 
There are also obvious maps $\df{A}{k}\ra \df{}{(A,\,k)}\ra \df{}{A}$ and this sequence maps to the corresponding $\df{A}{\fr{m}}\ra \df{}{(A,\fr{m})}\ra \df{}{A}$ by taking the syzygy of the deformed modules in the two first. By inserting $k[\vare]$ we get a commutative diagram
\begin{equation}\label{eq.Syz}
\begin{aligned}
\xymatrix@C+20pt@R-10pt@H-0pt{
\Der_{k}(A)\ar[d]^(0.44){(-1)}\ar[r]^(0.45){g^{k}}
& \xt{1}{A}{k}{k}\ar[d]^(0.45){\Syz}\ar[r]
& {}_{\ccirc\!}\mr{T}^{1}_{(A,\,k)}\ar[d]^{\sigma}\ar[r]
& \mr{T}^{1}_{\!A}\ar@{=}[d]
\\
\Der_{k}(A)\ar[r]^(0.45){g^{\fr{m}}}
& \xt{1}{A}{\fr{m}}{\fr{m}}\ar[r]
& {}_{\ccirc\!}\mr{T}^{1}_{(A,\,\fr{m})}\ar[r]
& \mr{T}^{1}_{\!A}
}
\end{aligned}
\end{equation}
by Theorem \ref{thm.lang} and Lemma \ref{lem.atSyz}. Since $\xt{1}{A}{k}{A}=0$, $\Syz$ is injective and so is $\sigma$ by the diagram.

The minimal resolution of the residue field
\begin{equation}\label{eq.fp}
k\thl A\lla A^{\oplus 3}\lla A^{\oplus 4}\xla{\;\bar{\phi}\;}A^{\oplus 4} \xla{\;\bar{\psi}\;}A^{\oplus 4}\xla{\;\bar{\phi}\;}\dots
\end{equation}
shows that $\rho\in \hm{}{A}{A^4}{A}$ with $\rho\bar{\phi}=0$ is a linear combination of  rows in $\bar{\psi}$ which implies that $\rho$ factors through $\fr{m}_A\subset A$. I.e. the natural map $\xt{2}{A}{k}{\fr{m}_A}\ra\xt{2}{A}{k}{A}\cong k$ is surjective. Then $\xt{1}{A}{\fr{m}_A}{\fr{m}_A}\cong k^4$ since $\xt{1}{A}{k}{k}\cong (\fr{m}/\fr{m}^2)^*\linebreak[1]\cong k^3$. Since $F^{\vee}=\hm{}{A}{F}{A}\cong F$ and $\Syz F\cong F^{\vee}$ (for a simple proof, see \cite[3.6]{iya/wem:10}), we have $F\cong \Syz F$. Then
\begin{equation}
\begin{aligned}
\xt{1}{A}{F}{F}&\cong \xt{1}{A}{\Syz F}{F}\cong \xt{2}{A}{F}{F}\cong \xt{2}{A}{F}{\fr{m}_A}
\\
& \cong \xt{1}{A}{F}{k}\cong k^{\oplus 4}
\end{aligned}
\end{equation}
Since $\xt{1}{A}{\fr{m}_A}{\fr{m}_A}\ra\xt{1}{A}{F}{\fr{m}_A}\cong\xt{1}{A}{F}{F}$ is surjective, it is an isomorphism. 

For any MCM $A$-module $M$ the restriction map $\gDer_{k}(A{\oplus} M)\ra \Der_{k}(A)$ in Theorem \ref{thm.lang} has a section; cf. \cite[5.8]{gus/ile:08} and \cite[15]{eri/gus:08}, hence $\xt{1}{A}{F}{F}\ra {}_{\ccirc\!}\mr{T}^{1}_{(A,F)}$ is injective and so is $\xt{1}{A}{\fr{m}}{\fr{m}}\ra {}_{\ccirc\!}\mr{T}^{1}_{(A,\fr{m})}$ by \eqref{eq.ladder} and $\xt{1}{A}{k}{k}\ra {}_{\ccirc\!}\mr{T}^{1}_{(A,k)}$ by \eqref{eq.Syz} since $\Syz$ is injective. 
Since $\sigma_F$ is smooth, ${}_{\ccirc\!}\mr{T}^{1}_{(A,\fr{m})}\ra {}_{\ccirc\!}\mr{T}^{1}_{(A,F)}$ is surjective. By \eqref{eq.ladder}, ${}_{\ccirc\!}\mr{T}^{1}_{(A,\fr{m})}\cong {}_{\ccirc\!}\mr{T}^{1}_{(A,F)}$ and $R_F\cong R_{\fr{m}}$.
Since $\dim_k\mr{T}^{1}_{A}=n$ and $\dim_k{}_{\ccirc\!}\mr{T}^{1}_{(A,\,k)}=n+2$ by Corollary \ref{cor.Ak1}, $\dim_k{}_{\ccirc\!}\mr{T}^{1}_{(A,F)}$ is either $n+3$ or $n+4$. If $A=P/(f)$ with $P$ regular then $A'=P[\vare]/(f+\vare)$ with $\vare^2=0$ corresponds to an element in $\df{}{A}(k[\vare])$. An extension of $F$ to $A'$ would be given by a matrix factorisation $(\phi+\vare\xi_1,\,\psi+\vare\xi_2)$ of $f+\vare$; cf. \eqref{eq.fp}. Then $(\phi+\vare\xi_1)(\psi+\vare\xi_2)=f{\cdot} \id+\vare(\xi_1\psi+\phi\xi_2)$, but $\xi_1\psi+\phi\xi_2$ cannot equal $\id$ since the entries of $\phi$ and $\psi$ are contained in $\fr{m}_P$. Hence $_{\ccirc\!}\mr{T}^{1}_{(A,F)}\ra \mr{T}^{1}_{A}$ is not surjective so $\dim_k{}_{\ccirc\!}\mr{T}^{1}_{(A,F)}=n+3$. By Corollary \ref{cor.Ak1}, $R_k\cong \mc{A}$ is formally smooth and $\dim R_k=n+2$. Then $R_k\cong R_F/(l)$ with $l\in\fr{m}_{R_F}{\smallsetminus}\,\fr{m}_{R_F}^2$.
\end{proof}
\begin{rem}
In \cite[5.11]{gus/ile:18} we claim that $\Spec R_F$ consists of two components. The larger is $\Spec R_F/(l)\cong \Spec R_k$ given in Corollary \ref{cor.ADE}. The smaller has dimension $n$ while the intersection has dimension $n-1$.
\end{rem}
\providecommand{\bysame}{\leavevmode\hbox to3em{\hrulefill}\thinspace}
\providecommand{\MR}{\relax\ifhmode\unskip\space\fi MR }
\providecommand{\MRhref}[2]{%
  \href{http://www.ams.org/mathscinet-getitem?mr=#1}{#2}
}
\providecommand{\href}[2]{#2}

%\bibliography{/Users/ile/Documents/Arbeid/Aktuelle/Generiske/Matematikk10}
\end{document}